\documentclass{class}

\title{Mapping tori of $A_{\infty}$-autoequivalences and Legendrian lifts of exact Lagrangians in circular contactizations}
\author{Adrian Petr}
\date{}

\hypersetup{
    pdfauthor   = {Adrian Petr},
    pdftitle    = {Legendrian lift of an exact Lagrangian in the circular contactization},
    pdfsubject  = {Article},
    pdfkeywords = {keywords},
}

\theoremstyle{plain}
\newtheorem{thm}{Theorem}[section]
\newtheorem{prop}[thm]{Proposition}
\newtheorem{coro}[thm]{Corollary}
\newtheorem{lemma}[thm]{Lemma}
\newtheorem{conj}[thm]{Conjecture}

\theoremstyle{remark}
\newtheorem{rmk}[thm]{Remark}
\newtheorem{rmks}[thm]{Remarks}

\theoremstyle{definition}
\newtheorem{defin}[thm]{Definition}
\newtheorem{exa}[thm]{Example}

\newcommand\mcA{\mathcal{A}}
\newcommand\mcB{\mathcal{B}}
\newcommand\mcC{\mathcal{C}}
\newcommand\mcD{\mathcal{D}}
\newcommand\mcE{\mathcal{E}}
\newcommand\mcF{\mathcal{F}}
\newcommand\mcG{\mathcal{G}}
\newcommand\mcH{\mathcal{H}}
\newcommand\mcI{\mathcal{I}}

\newcommand\mcL{\mathcal{L}}
\newcommand\mcM{\mathcal{M}}
\newcommand\mcN{\mathcal{N}}
\newcommand\mcO{\mathcal{O}}
\newcommand\mcP{\mathcal{P}}

\newcommand\mcR{\mathcal{R}}
\newcommand\mcS{\mathcal{S}}

\begin{document}

\selectlanguage{english}	
	
\maketitle

\begin{abstract}
	
	We study mapping tori of quasi-autoequivalences $\tau : \mcA \to \mcA$ which induce a free action of $\mathbf{Z}$ on objects. More precisely, we compute the mapping torus of $\tau$ when it is strict and acts bijectively on hom-sets, or when the $A_{\infty}$-category $\mcA$ is directed and there is a bimodule map $\mcA (-, -) \to \mcA (-, \tau (-))$ satisfying some hypotheses.
	Then we apply these results in order to link together the Fukaya $A_{\infty}$-category of a family of exact Lagrangians, and the Chekanov-Eliashberg DG-category of Legendrian lifts in the circular contactization.
	
\end{abstract}

\tableofcontents 

\theoremstyle{plain}
\newtheorem*{propo}{Proposition}
\newtheorem{thmintro}{Theorem}
\renewcommand*{\thethmintro}{\Alph{thmintro}}
\newtheorem*{corointro}{Corollary}

\theoremstyle{remark}
\newtheorem*{rmkintro}{Remark}
\newtheorem*{rmksintro}{Remarks}

\section*{Introduction}
\addcontentsline{toc}{section}{Introduction} 

Legendrian contact homology was introduced by Chekanov \cite{Che02} and Eliashberg \cite{Eli98}, and it fits into the Symplectic Field Theory introduced in \cite{EGH00}. It has been rigorously defined in the contactization of a Liouville manifold in \cite{EES07}, following \cite{EES05}. The importance of Legendrian contact homology goes beyond its applications to the Legendrian isotopy problem: for example, it was used by Bourgeois, Ekholm and Eliashberg in \cite{BEE12} to compute symplectic invariants of Weinstein manifolds, and in a different way by Chantraine, Dimitroglou Rizell, Ghiggini and Golovko in \cite{CDGG17} to prove a generation result for the wrapped Fukaya category of Weinstein manifolds. 

The motivation for this paper is the study of Legendrian contact homology in subcritically fillable and Boothby-Wang contact manifolds, the latter being named after \cite{BW58}. This has been done combinatorially in dimension three by Ekholm and Ng in \cite{EN15} for the subcritically fillable case, and by Sabloff in \cite{Sab03} for the Boothby-Wang case. The importance of the first kind of manifolds comes from the fact that every Weinstein manifold is obtained from a subcritical Weinstein manifold (of the form $\mathbf{C} \times P$ for some Weinstein manifold $P$) by attaching handles along Legendrian submanifolds in its boundary at infinity. The importance of the second kind of manifolds comes from a theorem of Donaldson in \cite{Don96}, which states that any integral symplectic manifold $\left( X, \omega \right)$ admits a symplectic submanifold $D \subset X$ of codimension $2$, such that $X \setminus D$ is a Liouville manifold whose boundary at infinity is a Boothby-Wang contact manifold. 
The first step before attacking both cases presented above is to study Legendrian contact homology in the circular contactization of a Liouville manifold. In fact, both subcritically fillable and Boothby-Wang contact manifolds can be seen as compactifications of such spaces. This paper links together the Fukaya $A_{\infty}$-category of a family of connected compact exact Lagrangians in a Liouville manifold $(P, \lambda)$, and the Chekanov-Eliashberg DG-category of Legendrian lifts in the circular contactization $\left( S^1 \times P, \ker ( d \theta - \lambda ) \right)$.

The strategy we follow is to lift the situation to the usual contactization $\mathbf{R} \times P$ which has been much more studied. This naturally leads to consider an $A_{\infty}$-category whose objects are the lifts in $\mathbf{R} \times P$ of our starting Legendrians,  and morphisms spaces are generated by Reeb chords. Moreover, the deck transformations of the cover $\mathbf{R} \to S^1$ induce an $A_{\infty}$-autoequivalence of this category. 
The rest of the proof has two main ingredients: 
\begin{enumerate}
	
	\item Functorial properties of the Legendrian invariants, which are used to bring us in a situation where we can apply the correspondence result of \cite{DR16} between discs in the symplectization $\mathbf{R} \times \mathbf{R} \times P$ and polygons in $P$.
	
	\item Two algebraic results of independent interest about mapping tori of $A_{\infty}$-autoequivalences, that allow us to bridge the gaps between the algebraic invariants we are interested in.  
	
\end{enumerate} 
We now proceed to describe the organization of the paper and state our main results.

\paragraph{Algebra.} 

In section \ref{subsection algebra}, we briefly recall the definitions of $A_{\infty}$-(co)categories and give references for standard notions that we do not recall, such as (co)bar, graded dual and Koszul dual constructions.
On the other hand, we discuss in some details the notions of modules over $A_{\infty}$-categories, as well as Grothendieck construction and homotopy pushout associated to a diagram of $A_{\infty}$-categories following \cite[section A.4]{GPS19}.
We use it to introduce the notion of ``cylinder object for an $A_{\infty}$-category'', which is supposed to mimic the corresponding notion in homotopy theory. 

\paragraph{Mapping torus of an $A_{\infty}$-autoequivalence.} 

In section \ref{subsection mapping torus}\footnote{In this section, $A_{\infty}$-categories are always assumed to be \emph{strictly unital} (see \cite[paragraph (2a)]{Sei08}).}, we define the mapping torus associated to a quasi-autoequivalence $\tau$ of an $A_{\infty}$-category $\mcA$ as the $A_{\infty}$-category
\[\mathrm{MT} (\tau) := \mathrm{hocolim} \left( 
\begin{tikzcd}
	\mcA \sqcup \mcA \ar[r, "\mathrm{id} \sqcup \tau"] \ar[d, "\mathrm{id} \sqcup \mathrm{id}" left] & \mcA \\
	\mcA .
\end{tikzcd} \right). \]
Observe that this terminology was also used in \cite{Kar21}, but we do not know if the two notions coincide. When considering an $A_{\infty}$-autoequivalence $\tau : \mcA \to \mcA$, we always assume that $\mcA$ is equipped with a $\mathbf{Z}$-splitting of $\mathrm{ob} (\mcA)$ compatible with $\tau$, which is a bijection 
\[\mathbf{Z} \times \mcE \xrightarrow{\sim} \mathrm{ob} \left( \mcA \right), \quad \left( n, E \right) \mapsto X^n \left( E \right) \]
such that $\tau \left( X^n \left( E \right) \right) = X^{n+1} \left( E \right)$ for every $n \in \mathbf{Z}$ and $E \in \mcE$ (see Definition \ref{definition group-action}).
This naturally turns $\mcA$ into an Adams-graded $A_{\infty}$-category, where the Adams-degree of a morphism in $\mcA (X^i (E), X^j (E'))$ is defined to be $(j-i)$.
It then follows that the mapping torus of $\tau$ is also Adams-graded.

Section \ref{subsection mapping torus} contains two results about mapping tori of $A_{\infty}$-autoequivalences: we choose to state only the most important in this introduction.
We denote by $\mathbf{F} \left[ t_m \right]$ the augmented Adams-graded associative algebra generated by a variable $t_m$ of bidegree $(m, 1)$.
Observe that if $\mcC$ is a subcategory of an $A_{\infty}$-category $\mcD$ with $\mathrm{ob} (\mcC) = \mathrm{ob} (\mcD)$, then $\mcC \oplus (t_m \mathbf{F} [t_m] \otimes \mcD)$ is naturally an Adams-graded $A_{\infty}$-category, where the Adams degree of $t_m^k \otimes x$ equals $k$.
Besides, if $\mcC$ is an $A_{\infty}$-category equipped with a $\mathbf{Z}$-splitting of $\mathrm{ob} \left( \mcC \right)$, we denote by $\mcC^0$ the full $A_{\infty}$-subcategory of $\mcC$ whose set of objects corresponds to $\{ 0 \} \times \mcE$.
Finally, we use the functor $\mcC \mapsto \mcC_m$ of Definition \ref{definition forgetful functor}. 

\begin{thmintro}\label{thm mapping torus in weak situation introduction}
	
	Let $\tau$ be a quasi-autoequivalence of an $A_{\infty}$-category $\mcA$, weakly directed with respect to some compatible $\mathbf{Z}$-splitting of $\mathrm{ob} \left( \mcA \right)$.
	Assume that there exists a closed degree $0$ bimodule map $f : \mcA_m \left( -, - \right) \to \mcA_m \left( -, \tau(-) \right)$ such that $f : \mcA_m \left( X^i(E) , X^j(E') \right) \to \mcA_m \left( X^i(E), X^{j+1} (E') \right)$ is a quasi-isomorphism for every $i < j$ and $E, E' \in \mcE$.
	Then there is a quasi-equivalence of Adams-graded $A_{\infty}$-categories 
	\[\mathrm{MT} (\tau) \simeq \mcA_m^0 \oplus \left( t_m \mathbf{F} \left[ t_m \right] \otimes \mcA_m \left[ f \left( \mathrm{units} \right)^{-1} \right]^0 \right) . \]
	
\end{thmintro}

\begin{rmksintro}
	
	\begin{enumerate}
		
		\item In \cite{Gan13}, the chain complex of $\mcA$-bimodule maps from the diagonal bimodule $\mcA \left( -, - \right)$ to some $\mcA$-bimodule $\mcB$ is called the two-pointed complex for Hochschild cohomology of $\mcA$ with coefficients in $\mcB$. According to \cite[Proposition 2.5]{Gan13}, this complex is quasi-isomorphic to the (ordinary) Hochschild cochain complex of $\mcA$ with coefficients in $\mcB$. In particular, the bimodule map $f$ in Theorem \ref{thm mapping torus in weak situation introduction} defines a class in the Hochschild cohomology of $\mcA_m$ with coefficients in $\mcA_m \left( -, \tau (-) \right)$.
		
		\item The $A_{\infty}$-category which computes the mapping torus in Theorem \ref{thm mapping torus in weak situation introduction} is very similar to the categories studied in \cite{Sei08bis}, with main difference the presence of curvature in Seidel's setting.
		
		\item The use of the functor $\mcC \mapsto \mcC_m$ in Theorem \ref{thm mapping torus in weak situation introduction} is not of any deep importance. It was convenient for us to introduce it here for our application to Legendrian contact homology (see Theorem \ref{thm mainthm} below).
		
	\end{enumerate}
	
\end{rmksintro}

\paragraph{Chekanov-Eliashberg DG-algebra.} 

In section \ref{section Legendrian invariants}, we recall the definition and functorial properties of the Chekanov-Eliashberg DG-category associated to a family of Legendrians in a hypertight contact manifold.

\paragraph{Legendrian lifts of exact Lagrangians in the circular contactization.} 

In section \ref{subsection proof of main thm}, we start with a family 
\[\mathbf{L} = \left( L(E) \right)_{E \in \mcE}, \quad \mcE = \left\{ 1, \dots, N \right\}, \]
of mutually transverse compact connected exact Lagrangian submanifolds in a Liouville manifold $\left( P, \lambda \right)$, and we study a Legendrian lift of $\mathbf{L}$ in the circular contactization $\left( S^1 \times P, \ker (d \theta - \lambda) \right)$.
More precisely, we assume\footnote{This can always be achieved by applying the Liouville flow in backwards time.} that there are primitives $f_E : L(E) \to \mathbf{R}$ of $\lambda_{| L(E)}$ such that $0 \leq f_1 < \dots < f_N \leq 1/2$,
and we consider the family of Legendrians
\[\mathbf{\Lambda^{\circ}} := (\Lambda^{\circ}(E))_{E \in \mcE}, \text{ where } \Lambda^{\circ}(E) = \left\{ (f_E(x), x) \in (\mathbf{R} / \mathbf{Z}) \times P \mid x \in L(E) \right\} . \]
We denote by $CE (\mathbf{\Lambda^{\circ}})$ the Chekanov-Eliashberg category of $\mathbf{\Lambda^{\circ}}$, by $\mcF uk (\mathbf{L})$ the Fukaya category generated by the Lagrangians $L(E)$ (see for example \cite[chapter 2]{Sei08}), and by $\overrightarrow{\mcF uk} (\mathbf{L})$ its directed subcategory (see \cite[paragraph (5n)]{Sei08}).

In order for the latter algebraic objects to be $\mathbf{Z}$-graded, we assume that $H_1(P)$ is free, that the first Chern class of $P$ (equipped with any almost complex structure compatible with $(- d \lambda)$) is 2-torsion, and that the Maslov class of the Lagrangians $L(E)$ vanish.  
As explained in section \ref{subsection Conley-Zehnder index}, the grading on $CE (\mathbf{\Lambda^{\circ}})$ depends on the choice of a symplectic trivialization of the contact structure along a fiber $h_0 = S^1 \times \{ a_0 \}$.
We denote by $CE_{-*}^r \left( \mathbf{\Lambda^{\circ}} \right)$ the Chekanov-Eliashberg DG-category of $\mathbf{\Lambda^{\circ}}$ with grading induced by the trivialization
\[\left( \xi^{\circ}_{| h_0}, d \alpha^{\circ} \right) \xrightarrow{\sim} \left( h_0 \times \mathbf{C}^n, dx \wedge dy \right), \quad \left( \left( \theta, a_0  \right), \left( \lambda_{a_0} (v), v \right) \right) \mapsto \left( \left( \theta, a_0  \right) , e^{2 i \pi r \theta} \psi (v) \right), \]
where  $\psi : \left( T_{a_0} P, - d \lambda_{a_0} \right) \xrightarrow{\sim} \left( \mathbf{C}^n, dx \wedge dy \right)$ is a symplectic isomorphism.

In this setting, $CE_{-*}^r \left( \mathbf{\Lambda^{\circ}} \right)$ is augmented (with the trivial augmentation) and Adams-graded (by the number of times a Reeb chord winds around the fiber).
As above, we denote by $\mathbf{F} \left[ t_m \right]$ the augmented Adams-graded associative algebra generated by a variable $t_m$ of bidegree $(m, 1)$.
Moreover, we denote by $E(-) = B(-)^{\#}$ (graded dual of bar construction) the Koszul dual functor (see  \cite[section 2]{LPWZ08} or \cite[section 2.3]{EL21}).

\begin{thmintro}\label{thm mainthm}
	
	Koszul duality holds for $CE_{-*}^r \left( \mathbf{\Lambda^{\circ}} \right)$, and there is a quasi-equivalence of augmented Adams-graded $A_{\infty}$-categories
	\[E \left( CE_{-*}^r \left( \mathbf{\Lambda^{\circ}} \right) \right) \simeq \overrightarrow{\mcF uk} (\mathbf{L}) \oplus \left( t_{2r} \mathbf{F} \left[ t_{2r} \right] \otimes \mcF uk (\mathbf{L}) \right). \]
	
\end{thmintro}

\begin{rmkintro}
		
	Koszul duality has many important consequences, see for example \cite{LPWZ08} or \cite{EL21}. In particular, by definition of Koszul duality (see \cite[Theorem 2.4]{LPWZ08} or \cite[Definition 17]{EL21}), Theorem \ref{thm mainthm} implies that there is a quasi-equivalence of augmented Adams-graded DG-categories
	\[CE_{-*}^r \left( \mathbf{\Lambda^{\circ}} \right) \simeq E \left( \overrightarrow{\mcF uk} (\mathbf{L}) \oplus \left( t_{2r} \mathbf{F} \left[ t_{2r} \right] \otimes \mcF uk (\mathbf{L}) \right) \right). \]
	Observe that this formula is closely related to Conjecture 6.3 in \cite{Sei08bis}, which was also discussed by Ganatra and Maydanskiy in the appendix of \cite{BEE12}.
	
\end{rmkintro}

We now give a corollary of the latter result.
If $B$ is a (unpointed) space, we consider its one-point compactification $B^*$ and view it as a pointed space (with base point the point at infinity). If moreover $X$ is a pointed space, we consider the half-smash product of $B$ and $X$,
\[X \rtimes B := X \wedge B^* \]
(where $\wedge$ denotes the smash product of pointed spaces).  
Finally, if $Y$ is a pointed space, we denote by $\Omega Y$ its loop space.

\begin{corointro}
	
	If $L$ is a connected compact exact Lagrangian and $\Lambda^{\circ}$ is a Legendrian lift of $L$ in the circular contactization, then there is a quasi-equivalence of augmented DG-algebras 
	\[ CE_{-*}^1 \left( \Lambda^{\circ} \right) \simeq C_{-*} \left( \Omega \left( \mathbf{CP}^{\infty} \rtimes L \right) \right). \]

\end{corointro}

\paragraph{Acknowledgments.}

This work is part of my PhD thesis that I did at Nantes Université under the supervision of Paolo Ghiggini and Vincent Colin, who I thank for their guidance and support. I also thank Baptiste Chantraine, Georgios Dimitroglou Rizell and Tobias Ekholm for helpful discussions.

% !TeX spellcheck = en_US
\section{Algebra}\label{subsection algebra}

In the following, $\mathbf{F}$ denotes the field $\mathbf{Z} / 2 \mathbf{Z}$. Vector spaces are always over $\mathbf{F}$.

\begin{defin}\label{definition category}
	
	An $A_{\infty}$-category $\mcA$ is the data of
	\begin{enumerate}
		
		\item a collection of objects $\mathrm{ob} \, \mcA$,
		
		\item for every objects $X,Y$, a graded vector space of morphisms $\mcA \left( X,Y \right)$,
		
		\item a family of degree $\left( 2-d \right)$ linear maps
		\[\mu^d : \mcA \left( X_0, X_1 \right) \otimes \dots \otimes \mcA \left( X_{d-1}, X_d \right) \to \mcA \left( X_0, X_d \right) \]
		indexed by the sequences of objects $\left( X_0, \dots, X_d \right)$, $d \geq 1$, such that
		\[\sum\limits_{0 \leq i < j \leq d} \mu^{d-(j-i)+1} \circ \left( \mathbf{1}^i \otimes \mu^{j-i} \otimes \mathbf{1}^{d-j} \right) = 0, \]
		for all $d \geq 1$.
	
	\end{enumerate}
	
\end{defin}

\begin{defin}\label{definition coalgebra}
	
	An $A_{\infty}$-cocategory $\mcC$ is the data of 
		\begin{enumerate}
		
		\item a collection of objects $\mathrm{ob} \, \mcC$,
		
		\item for every objects $X,Y$, a graded vector space of morphisms $\mcC \left( X,Y \right)$,
		
		\item a family of degree $\left( 2-d \right)$ linear maps
		\[\delta^d : \mcC \left( X_0, X_d \right) \to \bigoplus \limits_{d \geq 1} \bigoplus_{X_1, \dots, X_{d-1}} \mcC (X_0, X_1) \otimes \dots \otimes \mcC (X_{d-1}, X_d) \]
		indexed by the sequences of objects $\left( X_0, \dots, X_d \right)$, $d \geq 1$, such that
		\begin{itemize}
			
			\item for all $d \geq 1$,
			\[\sum\limits_{0 \leq i < j \leq d} \left( \mathbf{1}^i \otimes \delta^{j-i} \otimes \mathbf{1}^{d-j} \right) \circ \delta^{d-(j-i)+1} = 0, \]
			
			\item the map 
			\[C \to \prod_{d \geq 1} C^{\otimes d}, \quad x \mapsto \left( \delta^d \left( x \right) \right)_{d \geq 1} \]
			factors through the inclusion $\bigoplus\limits_{d \geq 1} C^{\otimes d} \to \prod\limits_{d \geq 1} C^{\otimes d}$.
			
		\end{itemize}
		
	\end{enumerate}
	
\end{defin}

\begin{rmks}
	
	\begin{enumerate}
		
		\item If $\mcE$ is some set, denote by $\mathbf{F}_{\mcE}$ the semi-simple algebra over $\mathbf{F}$ generated by elements $e_X$, $X \in \mcE$, such that
		\[e_X \cdot e_Y = 
		\left\{   
		\begin{array}{ll}
			e_X & \text{if } X = Y \\
			0 & \text{if } X \ne Y.
		\end{array}
		\right.  \]
		To any $A_{\infty}$-category $\mcA$ with $\mathrm{ob} (\mcA) = \mcE$, we can associate an $A_{\infty}$-algebra over $\mathbf{F}_{\mcE}$ where 
		\begin{itemize}

			\item the underlying graded vector space is $\bigoplus_{X,Y \in \mcE} \mcA (X, Y)$,
			
			\item given $x \in \mcA (X_0, Y_0)$, 
			\[e_X \cdot x = 
			\left\{   
			\begin{array}{ll}
				x & \text{if } X = X_0 \\
				0 & \text{if } X \ne X_0
			\end{array}
			\right. \text{ and } 
			x \cdot e_Y = 
			\left\{   
			\begin{array}{ll}
				x & \text{if } Y = Y_0 \\
				0 & \text{if } Y \ne Y_0,
			\end{array}
			\right. \]	
			
			\item operations are the same as on $\mcA$.

		\end{itemize}
		Conversely, to any $A_{\infty}$-algebra over $\mathbf{F}_{\mcE}$, one can associate an $A_{\infty}$-category with $\mathrm{ob} (\mcA) = \mcE$. 
		Note that the above discussion also applies to $A_{\infty}$-cocategories. As a result, the theory of $A_{\infty}$-(co)categories with $\mcE$ as set of objects is equivalent to the theory of $A_{\infty}$-(co)algebras over $\mathbf{F}_{\mcE}$.
	
		\item In this paper, we will appeal to several standard notions in the theory of $A_{\infty}$-(co)categories that we choose not to recall: instead, we list them and give corresponding references.
		\begin{itemize}
			
			\item For $A_{\infty}$-(co)maps, (co)augmentations and (co)bar, graded dual, Koszul dual constructions, see \cite[section 2]{EL21} (where everything is written in the language of $A_{\infty}$-(co)algebras over $\mathbf{F}_{\mcE}$).
			
			\item For general definitions and results about $A_{\infty}$-categories (in particular about homotopy between $A_{\infty}$-functors, homological perturbation theory, directed (sub)categories and twisted complexes), see \cite[chapter 1]{Sei08}.
			
			\item For quotient of $A_{\infty}$-categories, see \cite{LO06}, and for localization of $A_{\infty}$-categories, see \cite[section 3.1.3]{GPS20}.

		\end{itemize}
	
		\item An Adams-graded vector space is a $\mathbf{Z} \times \mathbf{Z}$-graded vector space: if $x$ is an element in the $(i,j)$ component, we say that $i$ is the cohomological degree of $x$, and $j$ is the Adams degree of $x$.
		An Adams-graded $A_{\infty}$-(co)category is an $A_{\infty}$-(co)category enriched over Adams-graded vector spaces, where the operations are required to be of degree $0$ with respect to the Adams grading. 
		See \cite{LPWZ08} for a treatment of Koszul duality in the context of Adams-graded $A_{\infty}$-algebras.
		
	\end{enumerate}

\end{rmks}

\subsection{Modules over $A_{\infty}$-categories}

Let $\mcC, \mcD$ be two $A_\infty$-categories, and let $\mcA, \mcB$ be two full subcategories of $\mcC, \mcD$ respectively.

\begin{defin}\label{definition bimodule}
	
	A $\left( \mcC, \mcD \right)$-bimodule $\mcM$ consists of the following data: 
	\begin{enumerate}
		
		\item for every pair $\left( X, Y \right) \in \mathrm{ob} \left( \mcC \right) \times \mathrm{ob} \left( \mcD \right)$, a vector space $\mcM \left( X, Y \right)$, 
		
		\item a family of degree $\left( 1-p-q \right)$ linear maps 
		\begin{align*}
		\mu_{\mcM} : \mcC \left( X_0, X_1 \right) \otimes & \dots \otimes \mcC \left( X_{p-1}, X_p \right) \otimes \mcM \left( X_p, Y_q \right) \\ 
		& \otimes \mcD \left( Y_q, Y_{q-1} \right) \otimes \dots \otimes \mcD \left( Y_1, Y_0 \right) \to \mcM \left( X_0, Y_0 \right)
		\end{align*}
		indexed by the sequences 
		\[\left( X_0, \dots, X_p, Y_0, \dots, Y_q \right) \in \mathrm{ob} \left( \mcC \right)^{p+1} \times \mathrm{ob} \left( \mcD \right)^{q+1}, \]
		which satisfy the relations 
		\begin{align*}
		\sum \mu_{\mcM} \left( \dots , \mu_{\mcC} \left( \dots \right) , \dots, u \dots \right) & + \sum \mu_{\mcM} \left( \dots , \mu_{\mcM} \left( \dots, u, \dots \right), \dots \right) \\ 
		& + \sum \mu_{\mcM} \left( \dots , u, \dots, \mu_{\mcD} \left( \dots \right) , \dots \right) = 0 .
		\end{align*}
		
	\end{enumerate}
	
	A degree $s$ morphism $t : \mcM_1 \to \mcM_2$ between two $\left( \mcC, \mcD \right)$-bimodules consists of a family of degree $\left( s-p-q \right)$ linear maps
	\begin{align*}
	t : \mcC \left( X_0, X_1 \right) \otimes & \dots \otimes \mcC \left( X_{p-1}, X_p \right) \otimes \mcM_1 \left( X_p, Y_q \right) \\ 
	& \otimes \mcD \left( Y_q, Y_{q-1} \right) \otimes \dots \otimes \mcD \left( Y_1, Y_0 \right) \to \mcM_2 \left( X_0, Y_0 \right)
	\end{align*}
	indexed by the sequences 
	\[\left( X_0, \dots, X_p, Y_0, \dots, Y_q \right) \in \mathrm{ob} \left( \mcC \right)^{p+1} \times \mathrm{ob} \left( \mcD \right)^{q+1}. \]
	The differential of such a morphism is defined by
	\begin{align*}
	\mu_{\mathrm{Mod}_{\mcC, \mcD}}^1 & \left( t \right) \left( \dots, u, \dots \right)  \\
	& = \sum t \left( \dots , \mu_{\mcC} \left( \dots \right) , \dots, u, \dots \right) + \sum t \left( \dots , \mu_{\mcM_1} \left( \dots, u, \dots \right), \dots \right) \\ 
	& +\sum t \left( \dots , u, \dots, \mu_{\mcD} \left( \dots \right), \dots \right) + \sum \mu_{\mcM_2} \left( \dots , t \left( \dots, u, \dots \right), \dots \right) .
	\end{align*} 
	Finally, the composition of $t_1 : \mcM_1 \to \mcM_2$ and $t_2 : \mcM_2 \to \mcM_3$ is such that
	\[	\mu_{\mathrm{Mod}_{\mcC}}^2 \left( t_1 , t_2 \right) \left( \dots, u, \dots \right)  = \sum t_2 \left( \dots , t_1 \left( \dots, u, \dots \right), \dots \right) .\]
	We denote by $\mathrm{Mod}_{\mcC, \mcD}$ the DG-category of $\left( \mcC, \mcD \right)$-bimodules.
	
\end{defin}

\begin{defin}\label{definition pullback bimodule}
	
	Let $\Phi_1, \Phi_2 : \mcC \to \mcD$ be two $A_{\infty}$-functors. Then there is a $\mcC$-bimodule $\mcD \left( \Phi_1 (-), \Phi_2 (-) \right)$ defined as follows:
	\begin{enumerate}
		
		\item on objects, it sends $(X_1, X_2)$ to $\mcD \left( \Phi_1 X_1, \Phi_2 X_2 \right)$,
		
		\item on morphisms, it sends a sequence $\left( \dots, y, \dots \right)$ in
		\begin{align*}
		\mcC \left( X_0, X_1 \right) \times & \dots \times \mcC \left( X_{p-1}, X_p \right)  \times \mcD \left( \Phi_1 X_p, \Phi_2 X_{p+1} \right) \\ 
		& \times \mcC \left( X_{p+1}, X_{p+2} \right) \times \dots \times \mcC \left( X_{p+q}, X_{p+q+1} \right)
		\end{align*}
		to 
		\begin{align*}
		\mu_{\mcD \left( \Phi_1 (-), \Phi_2 (-) \right)} & \left( \dots, y, \dots \right) \\ 
		& = \sum \mu_{\mcD} \left( \Phi_1 \left(\dots \right) , \dots , \Phi_1 \left( \dots \right) , y, \Phi_2 \left( \dots \right) , \dots , \Phi_2 \left( \dots \right) \right) .
		\end{align*}
		
	\end{enumerate}
	
\end{defin}

In the following, we will focus on \emph{left} $\mcC$-modules, which correspond to $\left( \mcC, \mathbf{F} \right)$-bimodules.
We denote by $\mathrm{Mod}_{\mcC}$ the DG-category of (left) $\mcC$-modules.

\begin{defin}\label{definition quasi-isomorphism between modules}
	
	Let $t : \mcM_1 \to \mcM_2$ be a degree $0$ closed $\mcC$-module map. We say that $t$ is a quasi-isomorphism if the induced chain map $t : \mcM_1 \left( X \right) \to \mcM_2 \left( X \right)$ is a quasi-isomorphism for every object $X$ in $\mcC$. (See \cite[section A.2]{GPS19} for a discussion on quasi-isomorphisms between $A_{\infty}$-modules).
	
\end{defin}

\begin{defin}\label{definition homotopy between morphisms of modules}
	
	Let $t, t' : \mcM_1 \to \mcM_2$ be two degree $0$ closed morphisms of $\mcC$-modules. A homotopy between $t$ and $t'$ is a $\mcC$-module map $h : \mcM_1 \to \mcM_2$ such that
	\[ t + t' = \mu^1_{\mathrm{Mod}_{\mcC}} \left( h \right) . \]
	
\end{defin}

\begin{defin}[See \cite{Sei08} paragraph (1l) and \cite{GPS19} section A.1]\label{definition Yoneda functor}
	 
	There is an $A_{\infty}$-functor 
	\[ \mcC \to \mathrm{Mod}_{\mcC}, \quad Y \mapsto \mcC \left( - , Y \right), \]
	called the Yoneda $A_{\infty}$-functor, defined as follows.
	For every object $X$, 
	\[ \mcC \left( -,Y \right) \left( X \right) = \mcC \left( X,Y \right). \]
	Besides, a sequence 
	\[ \left( x_0 , \dots , x_{d-1} \right) \in \mcC \left( X_0 , X_1 \right) \times \dots \times \mcC \left( X_{d-1} , X_d \right) \]
	acts on an element $u$ in $\mcC \left( X_d , Y \right)$ via the operations 
	\[ \mu_{\mcC \left( -,Y \right)} \left( x_0 , \dots , x_{d-1} , u \right) = \mu_{\mcC} \left( x_0 , \dots , x_{d-1} , u \right) . \]
	Finally, let 
	\[\mathbf{y} = \left( y_0 , \dots , y_{p-1} \right) \in \mcC \left( Y_0 , Y_1 \right) \times \dots \times \mcC \left( Y_{p-1} , Y_p \right) \]
	be a sequence of morphisms in $\mcC$. Then the Yoneda functor gives a morphism of $\mcC$-modules $t_{\mathbf{y}} : \mcC \left( -,Y_0 \right) \to \mcC \left( -,Y_p \right)$ which sends every sequence $\left( x_0 , \dots , x_{d-1} , u \right)$ as above to 
	\[ \mu_{\mcC} \left( x_0 , \dots , x_{d-1} , u , y_0, \dots, y_{p-1} \right) \in \mcC \left( X_0,Y_p \right) . \]
	
\end{defin}

We have the following important result.

\begin{prop}[Yoneda lemma]\label{prop Yoneda lemma}
	
	The Yoneda $A_{\infty}$-functor 
	\[ \mcC \to \mathrm{Mod}_{\mcC}, \quad Y \mapsto \mcC \left( - , Y \right) \]
	is cohomologically full and faithful.
	
\end{prop}

\begin{proof}
	
	This is Lemma 2.12 in \cite{Sei08}, and also Lemma A.1 in \cite{GPS19}.
	
\end{proof}

The Yoneda lemma has the following easy consequence. We state it for future reference.

\begin{coro}\label{coro closed module map homotopic to Yoneda module map}
	
	Every closed $\mcC$-module map $f : \mcC \left( -, X \right) \to \mcC \left( -, Y \right)$ is homotopic to the $\mcC$-module map $t_{f \left( e_X \right)}$ induced by $f \left( e_X \right) \in \mcC \left( X, Y \right)$. (see Definition \ref{definition Yoneda functor}).
	
\end{coro}

\begin{proof}
	
	According to the Yoneda lemma, $f$ is homotopic to $t_x$ for some closed $x$ in $\mcC \left( X , Y \right)$. Thus, there exists a $\mcC$-module map $h : \mcC \left( -, X \right) \to \mcC \left( -, Y \right)$ such that 
	\[f = t_x + \mu_{\mathrm{Mod}_{\mcC}}^1 \left( h \right) . \]
	Evaluating the latter relation at the unit $e_X \in \mcC \left( X, X \right)$ gives 
	\[f \left( e_X \right) = x + \mu_{\mcC}^1 \left( h e_X \right) . \]
	Therefore, $x$ is homotopic to $f \left( e_X \right)$, and this implies that $t_x$ is homotopic to $t_{f \left( e_X \right)}$ by the Yoneda lemma. Finally, we have that $f$ is homotopic to $t_{f \left( e_X \right)}$.
	
\end{proof}

\paragraph{Pullback of $A_{\infty}$-modules.}

\begin{defin}[See \cite{Sei08} paragraph (1k)]\label{definition pullback functor}
	 
	Let $\Phi : \mcC \to \mcD$ be an $A_{\infty}$-functor. Then there is a DG-functor 
	\[\Phi^* : \mathrm{Mod}_{\mcD} \to \mathrm{Mod}_{\mcC}, \quad \mcN \mapsto \Phi^* \mcN \]
	defined as follows.
	Let $\mcN$ be a $\mcD$-module. For every object $X$, 
	\[ \Phi^* \mcN \left( X \right) = \mcN \left( \Phi X \right) . \]
	Besides, a sequence 
	\[ \left( x_0 , \dots , x_{d-1} \right) \in \mcC \left( X_0 , X_1 \right) \times \dots \times \mcC \left( X_{d-1} , X_d \right) \]
	acts on an element $u \in \Phi^* \mcN \left( X_d \right)$ via the operations 
	\[ \mu_{\Phi^* \mcN} \left( x_0 , \dots , x_{d-1} , u \right) = \sum \mu_{\mcN} \left( \Phi \left( x_0 , \dots , x_{i_1-1} \right) , \dots , \Phi \left( x_{d-i_r} , \dots , x_{d-1} \right) , u \right) . \]
	Finally, let $t : \mcN_1 \to \mcN_2$ be a $\mcD$-module map. Then the above functor gives a $\mcC$-module map $\Phi^* t : \Phi^* \mcN_1 \to \Phi^* \mcN_2$ which sends every sequence $\left( x_0 , \dots , x_{d-1} , u \right)$ as above to 
	\[ \Phi^* t \left( x_0 , \dots , x_{d-1} , u \right) = \sum t \left( \Phi \left( x_0 , \dots , x_{i_1-1} \right) , \dots , \Phi \left( x_{d-i_r} , \dots , x_{d-1} \right) , u \right) . \]
	
\end{defin}

\begin{rmk}\label{rmk composition of pullback functors}
	
	Let $\Phi : \mcC \to \mcD$ be an $A_{\infty}$-functor, and let $\Psi : \mcD \to \mcE$ be another $A_{\infty}$-functor towards a third $A_{\infty}$-category $\mcE$. Then $\Phi^* \circ \Psi^* = \left( \Psi \circ \Phi \right)^*$ as DG-functors.

\end{rmk}

\begin{defin}\label{definition modules morphism induced by functor}
	
	Let $Y$ be an object of $\mcC$, and let $\Phi : \mcC \to \mcD$ be an $A_{\infty}$-functor.
	Then there is a degree $0$ closed $\mcC$-module map $t_{\Phi} : \mcC \left( - , Y \right) \to \Phi^* \mcD \left( - , \Phi \left( Y \right) \right)$ which sends any sequence 
	\[ \left( x_0 , \dots , x_{d-1} , u \right) \in \mcC \left( X_0 , X_1 \right) \times \dots \times \mcC \left( X_{d-1} , X_d \right) \times \mcC \left( X_d , Y \right) \]
	to 
	\[t_{\Phi} \left( x_0 , \dots , x_{d-1} , u \right) = \Phi \left( x_0 , \dots , x_{d-1} , u \right) \in \mcD \left( \Phi X_0 , \Phi Y \right) . \]
	
\end{defin}

\paragraph{Quotient of $A_{\infty}$-modules.}

\begin{defin}[See \cite{GPS20} section 3.1.3]\label{definition localization of modules}
	 
	There is a DG-functor 
	\[ \mathrm{Mod}_{\mcC} \to \mathrm{Mod}_{\mcC / \mcA}, \quad \mcM \, \mapsto \, _{\mcA \backslash} \mcM \]
	defined as follows. Let $\mcM$ be a $\mcC$-module. For every object $X$, 
	\begin{align*}
	_{\mcA \backslash} \mcM \left( X \right) & = \\
	& \mcM \left( X \right) \bigoplus \left( \bigoplus\limits_{\substack{ p \geq 1 \\ A_1,\dots,A_p \in \mcA}} \mcC \left( X , A_1 \right) [1] \otimes \dots \otimes \mcC \left( A_{p-1} , A_p \right) [1] \otimes \mcM \left( A_p \right) \right) .
	\end{align*}
	Besides, a sequence
	\[ \mathbf{x}_i = \left( x_i^0 , \dots , x_i^{p_i-1} \right) \in \mcC / \mcA \left( X_i , X_{i+1} \right) \quad \left( 0 \leq i \leq d-1 \right) \]
	acts on an element
	\[ \mathbf{u} = \left( x_d^0 , \dots , x_d^{p_d-1} , u \right) \in \, _{\mcA \backslash} \mcM \left( X_d \right) \]
	via the operations 
	\begin{align*}
	\mu_{_{\mcA \backslash} \mcM} & \left( \mathbf{x}_0 , \dots , \mathbf{x}_{d-1} , \mathbf{u} \right) = \\ 
	& \sum\limits_{ \substack{ 0 \leq i \leq p_0, 1 \leq j \leq p_d \\ i<j \text{ if } d=0 } } x_0^0 \otimes \dots \otimes x_0^{i-1} \otimes \mu_{\mcC} \left( x_0^i, \dots , x_d^{j-1} \right) \otimes x_d^j \otimes \dots \otimes x_d^{p_d-1} \otimes u \\
	& + \sum\limits_{ 0 \leq i \leq p_0 } x_0^0 \otimes \dots \otimes x_0^{i-1} \otimes \mu_{\mcM} \left( x_0^i , \dots , x_d^{p_d-1} , u \right) .
	\end{align*}
	Finally, let $t : \mcM_1 \to \mcM_2$ be a $\mcC$-module map. Then the above functor gives a $\mcC / \mcA$-module map $_{\mcA \backslash} t : \, _{\mcA \backslash} \mcM_1 \to \, _{\mcA \backslash} \mcM_2$ which sends every sequence $\left( \mathbf{x}_0 , \dots , \mathbf{x}_{d-1} , \mathbf{u} \right)$ as above to 
	\[ _{\mcA \backslash} t \left( \mathbf{x}_0 , \dots , \mathbf{x}_{d-1} , \mathbf{u} \right) = \sum\limits_{0 \leq i \leq p_0} x_0^0 \otimes \dots \otimes x_0^{i-1} \otimes t \left( x_0^i , \dots , x_d^{p_d-1} , u \right) . \] 
	
\end{defin}

\paragraph{Relations between pullback and quotient of $A_{\infty}$-modules.}

\begin{defin}\label{definition natural transformation}
	
	Let $\Phi : \mcC \to \mcD$ be an $A_{\infty}$-functor such that $\Phi \left( \mcA \right)$ is contained in $\mcB$, and let $X$ be a fixed object of $\mcC$. Then, for each $\mcD$-module $\mcN$, there is a chain map $_{\mcA \backslash} \left(\Phi^* \mcN \right) \left( X \right) \to \, _{\mcB \backslash} \mcN \left( \Phi X \right)$ which sends an element 
	\[ \mathbf{u} = \left( x^0 , \dots , x^{p-1} , u \right) \in \, _{\mcA \backslash} \left( \Phi^* \mcN \right) \left( X \right) \]
	to
	\[ \sum \Phi \left( x^0 , \dots , x^{i_1-1} \right) \otimes \dots \otimes \Phi \left( x^{i_r} , \dots , x^{p-1} \right) \otimes u \in \, _{\mcB \backslash} \mcN \left( \Phi X \right) . \]
	
	This defines a natural transformation between the functors $\mcN \mapsto \, _{\mcA \backslash} \left(\Phi^* \mcN \right) \left( X \right)$ and $\mcN \mapsto \, _{\mcB \backslash} \mcN \left( \Phi X \right)$ from $\mathrm{Mod}_{\mcD}$ to $\mathrm{Ch}$.
	In other words, for every $\mcD$-module map $t : \mcN_1 \to \mcN_2$, the following diagram of chain complexes commutes 
	\[\begin{tikzcd}
	_{\mcA \backslash} \left(\Phi^* \mcN_1 \right) \left( X \right) \ar[r] \ar[d, "_{\mcA \backslash} \left(\Phi^* t \right)"] & \, _{\mcB \backslash} \mcN_1 \left( \Phi X \right) \ar[d, "_{\mcB \backslash} t"] \\
	_{\mcA \backslash} \left(\Phi^* \mcN_2 \right) \left( X \right) \ar[r] & \, _{\mcB \backslash} \mcN_2 \left( \Phi X \right) 
	\end{tikzcd} . \]
	
\end{defin}

\begin{rmk}\label{rmk composition}
	
	Let $Y$ be an object of $\mcC$, and let $\Phi : \mcC \to \mcD$ be an $A_{\infty}$-functor such that $\Phi \left( \mcA \right)$ is contained in $\mcB$.
	Let $\widetilde{\Phi} : \mcC / \mcA \to \mcD / \mcB$ be the $A_{\infty}$-functor induced by $\Phi$ (see \cite[section 3]{LO06}).
	Localize the morphism $t_{\Phi} : \mcC \left( - , Y \right) \to \Phi^* \mcD \left( - , \Phi Y \right)$ of Definition \ref{definition modules morphism induced by functor} at $\mcA$ and evaluate at $X$ to get a chain map
	\[\mcC / \mcA \left( X,Y \right) = \, _{\mcA \backslash} \mcC \left( - , Y \right) \left( X \right) \xrightarrow{_{\mcA \backslash} t_{\Phi}} \, _{\mcA \backslash} \left( \Phi^* \mcD \left( - , \Phi Y \right) \right) \left( X \right) . \]
	Then the composition of this map with the chain map 
	\[ _{\mcA \backslash} \left( \Phi^* \mcD \left( - , \Phi Y \right) \right) \left( X \right) \to \, _{\mcB \backslash} \mcD \left( - , \Phi Y \right) \left( \Phi X \right) = \, \mcD / \mcB \left( \Phi X, \Phi Y \right) \]
	of Definition \ref{definition natural transformation} is the chain map $\widetilde{\Phi} : \mcC / \mcA \left( X,Y \right) \to \mcD / \mcB \left( \Phi X, \Phi Y \right)$.
	
\end{rmk}

\begin{prop}\label{prop quasi-isomorphism between localizations}
	
	Let $\Phi : \mcC_1 \to \mcC_2$ be an $A_{\infty}$-functor such that $\Phi \left( \mcA_1 \right)$ is contained in $\mcA_2$, and let $\widetilde{\Phi} : \mcC_1 / \mcA_1 \to \mcC_2 / \mcA_2$ be the $A_{\infty}$-functor induced by $\Phi$.
	
	Let $Y_1$ be an object of $\mcC_1$ and set $Y_2 := \Phi (Y_1)$.  Assume that there exists a $\mcC_i$-module $\mcM_{\mcC_i}$, a degree $0$ closed $\mcC_i$-module map $t_{\mcC_i} : \mcC_i \left( - , Y_i \right) \to \mcM_{\mcC_i}$ and a degree $0$ closed $\mcC_1$-module map $t_0 : \mcM_{\mcC_1} \to \Phi^* \mcM_{\mcC_2}$ such that the following diagram of $\mcC_1$-modules commutes  
	\[\begin{tikzcd}
	\mcC_1 \left( - , Y_1 \right) \ar[d, "t_{\mcC_1}"] \ar[r, "t_{\Phi}"] & \Phi^* \mcC_2 \left( - , Y_2  \right) \ar[d, "\Phi^* t_{\mcC_2}"] \\
	\mcM_{\mcC_1} \ar[r, "t_0"] & \Phi^* \mcM_{\mcC_2}
	\end{tikzcd} \]
	(see Definition \ref{definition modules morphism induced by functor} for the map $t_{\Phi}$). Then for every object $X$ in $\mcC_1$, there is a chain map $u : \, _{\mcA_1 \backslash} \mcM_{\mcC_1} \left( X \right) \to \, _{\mcA_2 \backslash} \mcM_{\mcC_2} \left( \Phi X \right)$ such that the following diagram of chain complexes commutes
	\[\begin{tikzcd}
	\mcC_1 / \mcA_1 \left( X , Y_1 \right) \ar[d, "_{\mcA_1 \backslash} t_{\mcC_1}"] \ar[r, "\widetilde{\Phi}"] & \mcC_2 / \mcA_2 \left( \Phi X , Y_2 \right) \ar[d, "_{\mcA_2 \backslash} t_{\mcC_2}"] \\
	_{\mcA_1 \backslash} \mcM_{\mcC_1} \left( X \right) \ar[r, "u"] & _{\mcA_2 \backslash} \mcM_{\mcC_2} \left( \Phi X \right) \\
	\mcM_{\mcC_1} \left( X \right) \ar[u, hook] \ar[r, "t_0"] & \mcM_{\mcC_2} \left( \Phi X \right) \ar[u, hook]
	\end{tikzcd} \]
	(the two lowest vertical maps are the inclusions). 
	If moreover the following holds:
	\begin{enumerate}
		
		\item for every objects $A$ in $\mcA_i$, the complexes $\mcM_{\mcC_i} \left( A \right)$ are acyclic,
		
		\item the maps $_{\mcA_i \backslash} t_{\mcC_i} : \, _{\mcA_i \backslash} \mcC_i \left( X , Y_i \right) \to \, _{\mcA_i \backslash} \mcM_{\mcC_i} (X)$ are quasi-isomorphisms, and
		
		\item the map $t_0 : \mcM_{\mcC_1} (X) \to \Phi^* \mcM_{\mcC_2} (X)$  is a quasi-isomorphism,
		
	\end{enumerate}
	then the map $\widetilde{\Phi} : \mcC_1 / \mcA_1 \left( X , Y_1 \right) \to \mcC_2 / \mcA_2 \left( \Phi X , Y_2 \right)$ is a quasi-isomorphism.
	
\end{prop}

\begin{proof}
	
	We apply the functor $\mcP \mapsto \, _{\mcA_1 \backslash} \mcP$ to the first diagram, we evaluate at $X$ and we use the natural map of Definition \ref{definition natural transformation} to get the following commutative diagram of chain complexes 
	\[\begin{tikzcd}
	_{\mcA_1 \backslash} \mcC_1 \left( - , Y_1 \right) \left( X \right) \ar[d, "_{\mcA_1 \backslash} t_{\mcC_1}"] \ar[r, "_{\mcA_1 \backslash} t_{\Phi}"] & _{\mcA_1 \backslash} \left( \Phi^* \mcC_2 \left( - , Y_2 \right) \right) \left( X \right) \ar[d, "_{\mcA_1 \backslash} \left( \Phi^* t_{\mcC_2} \right)"] \ar[r] & \, _{\mcA_2 \backslash} \mcC_2 \left( - , Y_2 \right) \left( \Phi X \right) \ar[d, "_{\mcA_2 \backslash} t_{\mcC_2}"] \\
	_{\mcA_1 \backslash} \mcM_{\mcC_1} \left( X \right) \ar[r, "_{\mcA_1 \backslash} t_0"] & _{\mcA_1 \backslash} \left( \Phi^* \mcM_{\mcC_2} \right) \left( X \right) \ar[r] & \, _{\mcA_2 \backslash} \mcM_{\mcC_2} \left( \Phi X \right)
	\end{tikzcd} . \]
	Then we compose the horizontal maps and we use Remark \ref{rmk composition} to get a commutative diagram of chain complexes
	\[\begin{tikzcd}
	\mcC_1 / \mcA_1 \left( X , Y_1 \right) \ar[d, "_{\mcA_1 \backslash} t_{\mcC_1}"] \ar[r, "\widetilde{\Phi}"] & \mcC_2 / \mcA_2 \left( \Phi X , Y_2 \right) \ar[d, "_{\mcA_2 \backslash} t_{\mcC_2}"] \\
	_{\mcA_1 \backslash} \mcM_{\mcC_1} \left( X \right) \ar[r, "u"] & _{\mcA_2 \backslash} \mcM_{\mcC_2} \left( \Phi X \right) \\
	\end{tikzcd} . \]
	This proves the first part of the Proposition because the following diagram of chain complexes commutes  
	\[\begin{tikzcd}
	_{\mcA_1 \backslash} \mcM_{\mcC_1} \left( X \right) \ar[r, "u"] & _{\mcA_2 \backslash} \mcM_{\mcC_2} \left( \Phi X \right) \\
	\mcM_{\mcC_1} \left( X \right) \ar[u, hook] \ar[r, "t_0"] & \mcM_{\mcC_2} \left( \Phi X \right) \ar[u, hook] \\
	\end{tikzcd} . \]
	The second part of the Proposition follows directly with \cite[Lemma 3.13]{GPS20}.
	
\end{proof}

\paragraph{Cone of module maps.}

\begin{defin}\label{definition cone of a morphism between modules}
	
	Let $t : \mcM_1 \to \mcM_2$ be a degree $0$ closed morphism of $\mcC$-modules. 
	We denote by 
	\[\mathrm{Cone} \left( \mcM_1 \xrightarrow{t} \mcM_2 \right) = \left[
	\begin{tikzcd}
		\mcM_1 \ar[d, "t"] \\
		\mcM_2
	\end{tikzcd}
	\right] \]
	the $\mcC$-module $\mcM$ defined as follows. For every object $X$ in $\mcC$, 
	\[\mcM \left( X \right) = \mcM_1 \left( X \right) \left[ 1 \right] \oplus \mcM_2 \left( X \right) \]
	as graded vector space, and any sequence 
	\[ \left( x_0 , \dots , x_{d-1} \right) \in \mcC \left( X_0 , X_1 \right) \times \dots \times \mcC \left( X_{d-1} , X_d \right) \]
	acts on an element $u_1 \oplus u_2 $ in $\mcM \left( X_d \right)$ via the operations 
	\begin{align*}
		\mu_{\mcM} \left( x_0 , \dots , x_{d-1} , u_1 \oplus u_2 \right) & = \\ \mu_{\mcM_1} \left( x_0, \dots, x_{d-1}, u_1 \right) 
		& \oplus \left( \mu_{\mcM_2} \left( x_0, \dots, x_{d-1}, u_2 \right) + t \left( x_0, \dots, x_{d-1}, u_1 \right) \right) .
	\end{align*}
	
\end{defin}

If we have two $\mcC$-module maps $t : \mcM_1 \to \mcM_2$ and $t' : \mcM_1 \to \mcM_2'$, then we set 
\[\left[
\begin{tikzcd}
	& \mcM_1 \ar[ld, "t"] \ar[rd, "t'"] & \\
	\mcM_2 & & \mcM_2'
\end{tikzcd}
\right] := 
\left[
\begin{tikzcd}
	\mcM_1 \ar[d, "{\left( t, t' \right)}"] \\
	\mcM_2 \oplus \mcM_2'
\end{tikzcd}
\right] . \]

\begin{prop}\label{prop morphism induced by homotopy}
	
	Consider a diagram of $\mcC$-modules 
	\[ 
	\begin{tikzcd}
		\mcM_1 \ar[r, "t_1"] \ar[d, "t_1'"] & \mcM_2 \ar[d, "t_2"] \\
		\mcM_2' \ar[r, "t_2'"] & \mcM_3
	\end{tikzcd}
	\]
	where all the morphisms are of degree $0$ and closed. Then any homotopy $h : \mcM_1 \to \mcM_3$ between
	\[t := \mu_{\mathrm{Mod}_{\mcC}}^2 \left( t_1, t_2 \right) \text{ and } t' := \mu_{\mathrm{Mod}_{\mcC}}^2 \left( t_1', t_2' \right) \]
	induces a degree $0$ closed $\mcC$-module map 
	\[t_h : \left[
	\begin{tikzcd}
		& \mcM_1 \ar[ld, "t_1"] \ar[rd, "t_1'"] & \\
		\mcM_2 & & \mcM_2'
	\end{tikzcd}
	\right] \to \mcM_3 \]
	defined by 
	\begin{align*}
		t_h \left( x_0 , \dots , x_{d-1} , u_1 \oplus u_2 \oplus u_2' \right) & = \\
		h \left( x_0 , \dots , x_{d-1} , u_1 \right) & + t_2 \left( x_0 , \dots , x_{d-1} , u_2 \right) + t_2' \left( x_0 , \dots , x_{d-1} , u_2' \right) .
	\end{align*}
	
\end{prop}

\begin{proof}
	
	The only thing to check is that $\mu_{\mathrm{Mod}_{\mcC}}^1 \left( t_h \right) = 0$, which is straightforward.
	
\end{proof}

\begin{rmk}\label{rmk cone and localization commute}
	
	If $t : \mcM_1 \to \mcM_2$ is a degree $0$ closed $\mcC$-module map, then 
	\[_{\mcA \backslash} \mathrm{Cone} \left( \mcM_1 \xrightarrow{t} \mcM_2 \right) = \mathrm{Cone} \left( _{\mcA \backslash} \mcM_1 \xrightarrow{_{\mcA \backslash} t} \, _{\mcA \backslash} \mcM_2 \right) . \]
	
\end{rmk}

\subsection{Grothendieck construction and homotopy pushout}\label{Grothendieck subsection}

An exposition on Grothendieck constructions and homotopy colimits in the context of $A_{\infty}$-categories can be found in \cite[appendix A]{GPS19}. We recall here definitions and basic facts that will serve us.  In this section, $A_{\infty}$-categories are always assumed to be \emph{strictly unital} (see \cite[paragraph (2a)]{Sei08}).

\begin{defin}\label{definition Grothendieck construction}
	
	Consider a diagram of $A_{\infty}$-categories 
	\[\begin{tikzcd}
	\mcC \ar[r, "\Phi_1"] \ar[d, "\Phi_2" left] & \mcD_1 \\
	\mcD_2 
	\end{tikzcd}. \]
	The Grothendieck construction of this diagram is the $A_{\infty}$-category $\mcG$ such that 
	\begin{enumerate}
		
		\item the set of objects is $\mathrm{ob} \left( \mcC \right) \sqcup \mathrm{ob} \left( \mcD_1 \right) \sqcup \mathrm{ob} \left( \mcD_2 \right)$,
		
		\item the space of morphisms between two objects $X$ and $Y$ is given by
		\[\mcG \left( X, Y \right) = \left\{
		\begin{array}{ll}
		\mcC \left( X, Y \right) & \text{if } X, Y \in \mathrm{ob} \left( \mcC \right) \\
		\mcD_i \left( X, Y \right) & \text{if } X, Y \in \mathrm{ob} \left( \mcD_i \right) \\
		\mcD_i \left( \Phi_i X, Y \right) & \text{if } X \in \mathrm{ob} \left( \mcC \right) \text{ and } Y \in \mathrm{ob} \left( \mcD_i \right) \\
		0 & \text{otherwise} .
		\end{array}
		\right. \]
		
		\item the operations involving only objects of $\mcC$, respectively of $\mcD_i$, are the same as in $\mcC$, respectively in $\mcD_i$, and for every sequence 
		\begin{align*}
		\left( x_0 , \dots , x_{p-1} , y , z_0 , \dots, z_{q-1} \right) \in \mcC \left( X_0 , X_1 \right) \otimes \cdots \otimes \mcC \left( X_{p-1} , X_p \right) \\ \otimes \mcG \left( X_p , Y_0 \right) \otimes \mcD_i \left( Y_0 , Y_1 \right) \otimes \cdots \otimes \mcD_i \left( Y_{q-1} , Y_q \right), 
		\end{align*}
		we have 
		\begin{align*}
		\mu_{\mcG} & \left( x_0 , \dots , x_{p-1} , y , z_0 , \dots, z_{q-1} \right) = \\
		& \sum \mu_{\mcD_i} \left( \Phi_i \left( x_0 , \dots , x_{i_1-1} \right) , \dots , \Phi_i \left( x_{p-i_r} , \dots , x_{p-1} \right) , y , z_0 , \dots, z_{q-1} \right) .
		\end{align*}
		
	\end{enumerate}
	
	We will call \emph{adjacent unit} of $\mcG$ any morphism in $\mcG \left( X , \Phi_i \left( X \right) \right)$ which corresponds to the unit in $\mcD_i \left( \Phi_i \left( X \right) , \Phi_i \left( X \right) \right)$.
	The \emph{homotopy colimit} $\mcH$ of the above diagram is the localization of $\mcG$ at its adjacent units.
	
\end{defin}

\begin{prop}\label{prop induced functor from Grothendieck construction}
	
	Let $\mcG$ be the Grothendieck construction of a diagram 
	\[\begin{tikzcd}
	\mcC \ar[r, "\Phi_1"] \ar[d, "\Phi_2" left] & \mcD_1 \\
	\mcD_2 .
	\end{tikzcd} \] 
	Then any strictly commutative square 
	\[\begin{tikzcd}
	\mcC \ar[r, "\Phi_1"] \ar[d, "\Phi_2" left] & \mcD_1 \ar[d, "\Psi_1"] \\
	\mcD_2 \ar[r, "\Psi_2"] & \mcE 
	\end{tikzcd} \]
	induces a functor $\sigma : \mcG \to \mcE$ defined as follows. On the objects, $\sigma$ acts on $\mcD_i$ as $\Psi_i$, and on $\mcC$ as $\Psi_1 \circ \Phi_1 = \Psi_2 \circ \Phi_2$ ; on the morphisms, $\sigma$ acts on $\mcD_i$ as $\Psi_i$, on $\mcC$ as $\Psi_1 \circ \Phi_1 = \Psi_2 \circ \Phi_2$, and it sends any sequence  
	\begin{align*}
	\left( x_0 , \dots , x_{p-1} , y , z_0 , \dots, z_{q-1} \right) \in \mcC \left( X_0 , X_1 \right) \otimes \cdots \otimes \mcC \left( X_{p-1} , X_p \right) \\ \otimes \mcG \left( X_p , Y_0 \right) \otimes \mcD_i \left( Y_0 , Y_1 \right) \otimes \cdots \otimes \mcD_i \left( Y_{q-1} , Y_q \right), 
	\end{align*} 
	to 
	\begin{align*}
	\sigma & \left( x_0 , \dots , x_{p-1} , y , z_0 , \dots, z_{q-1} \right) = \\
	& \sum \Psi_i \left( \Phi_i \left( x_0 , \dots , x_{i_1-1} \right) , \dots , \Phi_i \left( x_{p-i_r} , \dots , x_{p-1} \right) , y , z_0 , \dots, z_{q-1} \right) .
	\end{align*}  
	
\end{prop}

\begin{proof}
	
	This is a straightforward verification. 
	
\end{proof}

\begin{prop}[\cite{GPS19} Lemma A.5]\label{prop invariance of homotopy colimits}
	
	A strictly commutative diagram of $A_{\infty}$-categories
	\[\begin{tikzcd}
	\mcB_1 \ar[d] & \mcA \ar[l] \ar[r] \ar[d] & \mcB_2 \ar[d] \\
	\mcD_1 & \mcC \ar[l] \ar[r] & \mcD_2
	\end{tikzcd} \] 
	induces an $A_{\infty}$-functor from the Grothendieck construction of the top line to the Grothendieck construction of the bottom line which preserves adjacent units. 
	If moreover each vertical arrow is a quasi-equivalence, then the induced functor
	\[\mathrm{hocolim} \left( 
	\begin{tikzcd}
	\mcA \ar[r] \ar[d] & \mcB_1 \\
	\mcB_2 &
	\end{tikzcd} \right)
	\to 
	\mathrm{hocolim} \left( 
	\begin{tikzcd}
	\mcC \ar[r] \ar[d] & \mcD_1 \\
	\mcD_2 &
	\end{tikzcd} \right) \]
	is a quasi-equivalence.
	
\end{prop}

\begin{prop}\label{prop homotopic functors induce quasi-isomorphic homotopy colimits}
	
	Consider two diagrams of $A_{\infty}$-categories 
	\[\begin{tikzcd}
	\mcC \ar[r, "\Phi_1"] \ar[d, "\Phi_2" left] & \mcD_1 \\
	\mcD_2 &
	\end{tikzcd}
	\text{ and } 
	\begin{tikzcd}
	\mcC \ar[r, "\Psi_1"] \ar[d, "\Psi_2" left] & \mcD_1 \\
	\mcD_2 . &
	\end{tikzcd} \]
	If $\Phi_i$ and $\Psi_i$ (for $i \in \{1,2\}$) are homotopic (see \cite[paragraph (1h)]{Sei08}), then the homotopy colimits of the diagrams above are quasi-equivalent.
	
\end{prop}

\begin{proof}
	
	Let $\mcG_0$ and $\mcG_1$ the Grothendieck constructions of the above diagrams.
	
	Let $T_i$ be an homotopy from $\Phi_i$ to $\Psi_i$. This means that 
	\begin{align*}
	\Phi_i + \Psi_i & = \sum T_i \left( \dots , \mu_{\mcC} \left( \dots \right) , \dots \right) \\ 
	& + \sum \mu_{\mcD_i} \left( \Psi_i \left( \dots \right) , \dots , \Psi_i \left( \dots \right) , T_i \left( \dots \right) , \Phi_i \left( \dots \right), \dots, \Phi_i \left( \dots \right) \right) .
	\end{align*}
	
	We consider the functor $\kappa : \mcG_0 \to \mcG_1$ such that 
	\[ \kappa_{|\mcC} = \mathrm{id}_{\mcC}, \quad \kappa_{|\mcD_i} = \mathrm{id}_{\mcD_i}, \]
	and which sends every sequence 
	\begin{align*}
	\left( \dots , y , \dots \right) & \in \mcC \left( X_0,X_1 \right) \times \cdots \times \mcC \left( X_{p-1},X_p \right) \times \mcG_0 \left( X_p , Y_0 \right) \\
	& \times \mcD_i \left( Y_0,Y_1 \right) \times \cdots \times \mcD_i \left( Y_{q-1},Y_q \right),
	\end{align*}
	to
	\begin{align*}
	& \kappa \left( \dots , y , \dots \right) = \\ 
	& \sum \mu_{\mcD_i} \left( \Psi_i \left( \dots \right) , \dots , \Psi_i \left( \dots \right) , T_i \left( \dots \right), \Phi_i \left( \dots \right) , \dots , \Phi_i \left( \dots \right) , y , \dots \right)
	\end{align*}
	if $p$ is positive, and to
	\[\Phi \left( y , \dots \right) = \mathrm{id}_{\mcD_i} \left( y , \dots \right) \]
	otherwise.
	Using the facts that $\Phi_i , \Psi_i$ are $A_{\infty}$-functors, that $T_i$ is an homotopy from $\Phi_i$ to $\Psi_i$, and gathering the terms depending on if they contain $T_i^k \left( \dots \right)$ or $y$, we conclude that $\kappa$ satisfies the $A_{\infty}$-relations. This proves the result because $\kappa$ is a quasi-equivalence sending the adjacent units of $\mcG_0$ onto those of $\mcG_1$. 
	
\end{proof}

\subsection{Cylinder object and homotopy}

Let $\mcA_{\bot}$, $\mcA_I$ and $\mcA_{\top}$ be three copies of an $A_{\infty}$-category $\mcA$. 
We denote by $\mcC$ the Grothendieck construction of the following diagram
\[\begin{tikzcd}
\mcA_I \ar[r, "\mathrm{id}"] \ar[d, "\mathrm{id}" left] & \mcA_{\top} \\
\mcA_{\bot} &
\end{tikzcd}, \]
and we let $\iota_{\bot}, \iota_I, \iota_{\top} : \mcA \to \mcC$ be the strict inclusions with images $\mcA_{\bot}$, $\mcA_I$, $\mcA_{\top}$ respectively. Finally, we denote by $W_{\mcC}$ the set of adjacent units in $\mcC$, and we let $\mcC yl_{\mcA} = \mcC \left[ W_{\mcC}^{-1} \right]$ be the homotopy colimit of the diagram above. We say that $\mcC yl_{\mcA}$ is a cylinder object for $\mcA$. 

We denote by $\pi : \mcC \to \mcA$ the $A_{\infty}$-functor induced by the following commutative square 
\[\begin{tikzcd}
\mcA \ar[r, "\mathrm{id}"] \ar[d, "\mathrm{id}" left] & \mcA \ar[d, "\mathrm{id}"] \\
\mcA \ar[r, "\mathrm{id}"] & \mcA 
\end{tikzcd} \]
(see Proposition \ref{prop induced functor from Grothendieck construction}). 

\begin{prop}\label{prop cylinder object}
	
	The following diagram of $A_{\infty}$-categories commutes 
	\[\begin{tikzcd}
	\mcA \sqcup \mcA \ar[r, "\iota_{\bot} \sqcup \iota_{\top}"] \ar[rr, "\mathrm{id} \sqcup \mathrm{id}" below, bend right] & \mcC \ar[r, "\pi"] & \mcA .
	\end{tikzcd} \]
	Moreover, $\pi$ sends $W_{\mcC}$ to the set of units in $\mcA$, and the induced $A_{\infty}$-functor $\widetilde{\pi} : \mcC yl_{\mcA} \to \mcA \left[ \mathrm{units}^{-1} \right]$ is a quasi-equivalence.  
	
\end{prop}

\begin{proof}
	
	The facts that $\pi \circ \left( \iota_{\bot} \sqcup \iota_{\top} \right) = \mathrm{id} \sqcup \mathrm{id}$ and that $\pi$ sends $W_{\mcC}$ to the set of units in $\mcA$ are clear. We now show that $\widetilde{\pi} : \mcC yl_{\mcA} \to \mcA \left[ \mathrm{units}^{-1} \right]$ is a quasi-equivalence. 
	
	First observe that it is enough to show that the map  
	\[\widetilde{\pi} : \mcC yl_{\mcA} \left( X, Y \right) \to \mcA \left[ \{ \mathrm{units} \}^{-1} \right] \left( \pi X, \pi Y \right) \]
	is a quasi-isomorphism for every objects $X$, $Y$ in $\mcA_{\bot}$  because every object of $\mcC$ can be related to one of $\mcA_{\bot}$ by a zigzag of morphisms in $W_{\mcC}$, which are quasi-isomorphisms in $\mcC yl_{\mcA}$ (see \cite[Lemma 3.12]{GPS20}). 
	Our strategy is to apply Proposition \ref{prop quasi-isomorphism between localizations}.
	Let $Y = \iota_{\bot} (Z)$ be an object in $\mcA_{\bot}$.
	For the $\mcC$-module we take 
	\[\mcM_{\mcC} = \left[ 
	\begin{tikzcd}
	& \mcC \left( -, \iota_I (Z) \right) \ar[ld, "t_{I \bot}"] \ar[rd, "t_{I \top}"] & \\
	\mcC \left( -, \iota_{\bot} (Z) \right) & & \mcC \left( -, \iota_{\top} (Z) \right)
	\end{tikzcd} 
	\right], \]
	where 
	\[t_{I \triangle} : \mcC \left( -, \iota_I (Z) \right) \to \mcC \left( -, \iota_{\triangle} (Z) \right), \quad \triangle \in \{ \bot, \top \}, \]
	is the $\mcC$-module map induced by the adjacent unit in $\mcC \left( \iota_I (Z), \iota_{\triangle} (Z) \right)$ (see Definition \ref{definition Yoneda functor})).
	For the $\mcA$-module we simply take $\mcA \left( -, Z \right)$. Besides, we let $t_{\mcC} : \mcC \left( -, \iota_{\bot} (Z) \right) \to \mcM_{\mcC}$ be the $\mcC$-module inclusion, and we let $t_{\mcA} : \mcA \left( -, Z \right) \to \mcA \left( -, Z \right)$ be the identity map. 
	We now define the morphism $t_0 : \mcM_{\mcC} \to \pi^* \mcA \left( -, Z \right)$. Consider the following diagram of $\mcC$-modules
	\[\begin{tikzcd}
	\mcC \left( -, \iota_I (Z) \right) \ar[r, "t_{I \bot}"] \ar[d, "t_{I \top}"] & \mcC \left( -, \iota_{\top} (Z) \right) \ar[d, "t_{\pi}"] \\
	\mcC \left( -, \iota_{\bot} (Z) \right) \ar[r, "t_{\pi}"] & \pi^* \mcA \left( -, Z \right) . 
	\end{tikzcd} \]
	Observe that this diagram is commutative, and thus it induces a strict $\mcC$-module map $t_0 : \mcM_{\mcC} \to \pi^* \mcA \left( -, Z \right)$ according to Proposition \ref{prop morphism induced by homotopy}. It is then easy to see that the following diagram commutes 
	\[\begin{tikzcd}
	\mcC \left( - , \iota_{\triangle} (Z) \right) \ar[d, "t_{\mcC}"] \ar[r, "t_{\pi}"] & \pi^* \mcA \left( - , Z \right) \ar[d, "\pi^* t_{\mcA}"] \\
	\mcM_{\mcC} \ar[r, "t_0"] & \pi^* \mcA \left( -, Z \right) .
	\end{tikzcd} \]
	To conclude the proof, it suffices to check the three items of Proposition \ref{prop quasi-isomorphism between localizations}. 
	Observe that the pair $\left( \mcA \left( -, Z \right), t_{\mcA} \right)$ trivially satisfies the two first items.
	
	We check that $\mcM_{\mcC}$ satisfies the first item of Proposition \ref{prop quasi-isomorphism between localizations}. Let $Z'$ be an object in $\mcA$ and let $w$ be the adjacent unit in $\mcC \left( \iota_{I} (Z'), \iota_{\bot} (Z') \right)$ (the proof is the same for the adjacent unit in $\mcC \left( \iota_{I} (Z'), \iota_{\top} (Z') \right) \cap W_{\mcC}$). Then 
	\begin{align*}
	\mcM_{\mcC} \left( \mathrm{Cone} \, w \right) & = \mathrm{Cone} \left( \mcM_{\mcC} \left( \iota_{\bot} (Z') \right) \xrightarrow{\mu_{\mcC}^2 \left( w , - \right)} \mcM_{\mcC} \left( \iota_{I} (Z') \right) \right) \\
	& = \mathrm{Cone} \left( \mcC \left( \iota_{\bot} (Z') , \iota_{\bot} (Z) \right) \xrightarrow{\mu_{\mcC}^2 \left( w , - \right)} K \right).
	\end{align*}
	where 
	\[K = \left[ 
	\begin{tikzcd}
	& \mcC \left( \iota_{I} (Z') , \iota_I (Z) \right) \ar[ld, "t_{I \bot}"] \ar[rd, "t_{I \top}"] \\
	\mcC \left( \iota_{I} (Z') , \iota_{\bot} (Z) \right) & & \mcC \left( \iota_{I} (Z') , \iota_{\top} (Z) \right)
	\end{tikzcd}
	\right] . \]
	Observe that $\mu_{\mcC}^2 \left( w , - \right) : \mcC \left( \iota_{\bot} (Z') , \iota_{\bot} (Z) \right) \to K$ is injective so its cone is quasi-isomorphic to its cokernel, which is the cone of $t_{I \top} : \mcC \left( \iota_{I} (Z') , \iota_I (Z) \right) \to \mcC \left( \iota_{I} (Z') , \iota_{\top} (Z) \right)$. The latter map is a quasi-isomorphism, so $\mcM_{\mcC} \left( \mathrm{Cone} \, w \right)$ is acyclic.
	
	We now check that $\left( \mcM_{\mcC}, t_{\mcC} \right)$ satisfies the second item of Proposition \ref{prop quasi-isomorphism between localizations}. Observe that 
	\[_{W_{\mcC}^{-1}} \mcM_{\mcC} = \left[ 
	\begin{tikzcd}
	& _{W_{\mcC}^{-1}} \mcC \left( -, \iota_I (Z) \right) \ar[ld, "_{W_{\mcC}^{-1}} t_{I \bot}"] \ar[rd, "_{W_{\mcC}^{-1}} t_{I \top}"] & \\
	_{W_{\mcC}^{-1}} \mcC \left( -, \iota_{\bot} (Z) \right) & & _{W_{\mcC}^{-1}} \mcC \left( -, \iota_{\top} (Z) \right)
	\end{tikzcd} 
	\right] \]
	and $_{W_{\mcC}^{-1}} t_{\mcC} : \, _{W_{\mcC}^{-1}} \mcC \left( -, \iota_{\bot} (Z) \right) \to \, _{W_{\mcC}^{-1}} \mcM_{\mcC}$ is the inclusion. Thus if $X$ is some object of $\mcC$, the cone of $_{W_{\mcC}^{-1}} t_{\mcC} : \mcC yl_{\mcA} \left( X, \iota_{\bot} (Z) \right) \to \, _{W_{\mcC}^{-1}} \mcM_{\mcC} \left( X \right)$ is quasi-isomorphic to the cone of the multiplication in $\mcC yl_{\mcA}$ by an element of $W_{\mcC}$, which is a quasi-isomorphism. Thus the map $_{W_{\mcC}^{-1}} t_{\mcC} : \, _{W_{\mcC}^{-1}} \mcC \left( -, \iota_{\bot} (Z) \right) \to \, _{W_{\mcC}^{-1}} \mcM_{\mcC}$ indeed is a quasi-isomorphism. 
	
	It remains to check the third item of Proposition \ref{prop quasi-isomorphism between localizations}, which is that the map $t_0 : \mcM_{\mcC} \left( X \right) \to \pi^* \mcA \left( -, Z \right) \left( X \right)$ is a quasi-isomorphism when $X$ is in $\mcA_{\bot}$. This is true because $\mcM_{\mcC} \left( \iota_{\bot} (Z') \right) = \mcC ( \iota_{\bot} (Z'), \iota_{\bot} (Z) ) = \mcA ( Z', Z )$, and 
	\[t_0 : \mcA ( Z', Z ) = \mcM_{\mcC} \left( \iota_{\bot} (Z') \right) \to \pi^* \mcA \left( -, Z \right) \left( \iota_{\bot} (Z') \right) = \mcA ( Z', Z ) \]
	is the identity.
	This concludes the proof. 
	
\end{proof}

\begin{rmk}
	
	Proposition \ref{prop cylinder object} can be thought as saying that $\mcC yl_{\mcA}$ is a cylinder object for $\mcA$.
	
\end{rmk}

\begin{prop}\label{prop homotopy is generalized homotopy}
	
	If two $A_{\infty}$-functors $\Phi, \Psi : \mcA \to \mcB$ are homotopic (see \cite[paragraph (1h)]{Sei08}), then there is an $A_{\infty}$-functor $\eta : \mcC \to \mcB$ which sends the adjacent units of $\mcC$ to the units in $\mcB$ and such that the following diagram commutes
	\[\begin{tikzcd}
	\mcA \ar[d, "\iota_{\top}" left] \ar[rd, bend left, "\Phi"] & \\
	\mcC \ar[r, "\eta"] & \mcB \\
	\mcA \ar[u, "\iota_{\bot}"] \ar[ru, bend right, "\Psi"] . & 
	\end{tikzcd} \]
	
\end{prop}

\begin{proof}
	
	On the objects, we set $\eta \left( X_{\triangle} \right) = \Phi \left( X \right) = \Psi \left( X \right)$ for every object $X$ of $\mcA$ and $\triangle \in \{ \bot, I, \top \}$. 
	On the morphisms, we set 
	\[\eta_{| \mcA_{\bot}} = \eta_{| \mcA_I} = \Psi, \, \eta_{| \mcA_{\top}} = \Phi \]
	and ask for the restriction of $\eta$ to
	\[\mcA_I \left( X_0 , X_1 \right) \otimes \cdots \otimes \mcA_I \left( X_{p-1} , X_p \right) \otimes \mcC \left( X_p , X_{p+1} \right) \otimes \mcA_{\bot} \left( X_{p+1} , X_{p+2} \right) \otimes \cdots \otimes \mcA_{\bot} \left( X_{p+q} , X_{p+q+1} \right) \] 
	to be $\Psi$.
	It remains to define $\eta$ for 
	\begin{align*}
	\left( \dots, x, \dots \right) \in & \mcA_I \left( X_0 , X_1 \right) \otimes \cdots \otimes \mcA_I \left( X_{p-1} , X_p \right) \\ 
	& \otimes \mcC \left( X_p , X_{p+1} \right) \otimes \mcA_{\top} \left( X_{p+1} , X_{p+2} \right) \otimes \cdots \otimes \mcA_{\top} \left( X_{p+q} , X_{p+q+1} \right) . 
	\end{align*}
	For this we take a homotopy $T$ between $\Phi$ and $\Psi$, which means that
	\begin{align*}
	\Phi + \Psi & = \sum T \left( \dots , \mu_{\mcA} \left( \dots \right) , \dots \right) \\ 
	& + \sum \mu_{\mcB} \left( \Phi \left( \dots \right) , \dots , \Phi \left( \dots \right) , T \left( \dots \right) , \Psi \left( \dots \right), \dots, \Psi \left( \dots \right) \right) .
	\end{align*}
	Then we let
	\begin{align*}
	\eta & \left( \dots, x, \dots \right) = \\
	& \sum \mu_{\mcB} \left( \Phi (\dots), \dots, \Phi (\dots), T (\dots), \Psi (\dots), \dots, \Psi (\dots), \Psi ( \dots, x, \dots ) \Psi (\dots), \dots, \Psi (\dots) \right)
	\end{align*}
	if $p$ is positive, and $\eta \left( x, \dots \right) = \Psi \left( x, \dots \right)$ otherwise.
	
\end{proof}

\subsection{Adjunctions between Adams-graded and non Adams-graded}

We end this this section by describing specific adjunctions between the category of Adams-graded $A_{\infty}$-categories concentrated in non-negative Adams-degree and the category of (non Adams-graded) $A_{\infty}$-categories.  

\begin{defin}\label{definition change of grading}
	
	If $V$ is an Adams-graded vector space and $m$ is an integer, we denote by $V_m$ the graded vector space whose degree $n$ component is the direct sum of the bidegree $\left( p,k \right)$ components of $V$, where the sum is over the set of couples $\left( p,k \right) \in \mathbf{Z} \times \mathbf{Z}$ such that $p-mk=n$.
	
\end{defin}

\begin{defin}\label{definition forgetful functor}
	
	If $\mcC$ is an Adams-graded $A_{\infty}$-category, we denote by $\mcC_m$ the (non Adams-graded) $A_{\infty}$-category obtained from $\mcC$ by changing the grading so that
	\[\mcC_m \left(X_0, X_1 \right) = \mcC \left( X_0, X_1 \right)_m \] 	
	Observe that any $A_{\infty}$-functor $\Phi : \mcC_1 \to \mcC_2$ between two Adams-graded $A_{\infty}$-categories induces an $A_{\infty}$-functor from $(\mcC_1)_m$ to $(\mcC_2)_m$ (that we still denote by $\Phi$) which acts exactly as $\Phi$ on objects and morphisms. 
	This defines a functor $\mcC \mapsto \mcC_m$ from the category of Adams-graded $A_{\infty}$-categories to the category of (non Adams-graded) $A_{\infty}$-categories. 
	
\end{defin}

We denote by $\mathbf{F} \left[ t_m \right]$ the augmented Adams-graded associative algebra generated by a variable $t_m$ of bidegree $(m, 1)$, and by $t_m \mathbf{F} \left[ t_m \right]$ its augmentation ideal (or equivalently, the ideal generated by $t_m$).

\begin{defin}\label{definition cofree functor}
	
	If $\mcD$ is a (non Adams-graded) $A_{\infty}$-category, we denote by $\mathbf{F} \left[ t_m \right] \otimes \mcD$ the Adams-graded $A_{\infty}$-category such that 
	\begin{enumerate}
		
		\item the objects of $\mathbf{F} \left[ t_m \right] \otimes \mcD$ are those of $\mcD$,
		
		\item the space of morphisms from $Y_1$ to $Y_2$ is $\mathbf{F} \left[ t_m \right] \otimes \mcD \left( Y_1, Y_2 \right)$, and if $y \in \mcD \left( Y_1, Y_2 \right)$ is of degree $j$,  $t_m^k \otimes y$ is of bidegree $\left( j + mk, k \right)$,
		
		\item the operations send any sequence $\left( t_m^{k_0} \otimes y_0, \dots, t_m^{k_{d-1}} \otimes y_{d-1} \right)$ of morphisms to 
		\[\mu_{\mathbf{F} \left[ t_m\right] \otimes \mcD} \left( t_m^{k_0} \otimes y_0, \dots, t_m^{k_{d-1}} \otimes y_{d-1} \right) = t_m^{k_0 + \dots + k_{d-1}} \otimes \mu_{\mcD} \left( y_0, \dots, y_{d-1} \right) . \]
		
	\end{enumerate}
	Observe that any $A_{\infty}$-functor $\Psi : \mcD_1 \to \mcD_2$ between (non Adams-graded) $A_{\infty}$-categories induces an $A_{\infty}$-functor $\mathbf{F} \left[ t_m \right] \otimes \mcD_1 \to \mathbf{F} \left[ t_m \right] \otimes \mcD_2$ which acts as $\Psi$ on objects, and which sends any sequence $\left( t_m^{k_0} \otimes y_0, \dots, t_m^{k_{d-1}} \otimes y_{d-1} \right)$ of morphisms to $t_m^{k_0 + \dots k_{d-1}} \otimes \Psi \left( y_0, \dots, y_{d-1} \right)$.
	This defines a functor $\mcD \mapsto \mathbf{F} \left[ t_m \right] \otimes \mcD$ from the category of (non Adams-graded) $A_{\infty}$-categories to the category of Adams-graded $A_{\infty}$-categories. 
	
\end{defin}

\begin{defin}\label{definition adjunction}
	
	Let $\mcC$ be an Adams-graded $A_{\infty}$-category concentrated in non-negative Adams-degree, and let $\mcD$ be a (non Adams-graded) $A_{\infty}$-category. To any $A_{\infty}$-functor $\Psi_m : \mcC_m \to \mcD$, we associate an $A_{\infty}$-functor $\Psi : \mcC \to \mathbf{F} \left[ t_m \right] \otimes \mcD$ which sends a sequence $\left( x_0, \dots, x_{d-1} \right)$, where $x_j$ is of bidegree $\left( i_j, k_j \right)$, to
	\[\Psi \left( x_0, \dots, x_{d-1} \right) = t_m^{k_0 + \dots + k_{d-1}} \otimes \Psi_m \left( x_0, \dots, x_{d-1} \right) . \]
	This defines an adjunction between the category of Adams-graded $A_{\infty}$-categories concentrated in non-negative Adams-degree and the category of (non Adams-graded) $A_{\infty}$-categories.  
	
\end{defin}

% !TeX spellcheck = en_US
\section{Mapping torus of an $A_{\infty}$-autoequivalence}\label{subsection mapping torus}

In this section, we introduce the notion of mapping torus for a quasi-autoequivalence of an $A_{\infty}$-category, by analogy with the mapping torus associated to an automorphism of a topological space. This terminology was also used in \cite{Kar21}, but we do not know if the two notions coincide. The two main theorems of this section allow us to compute this mapping torus under different hypotheses.  

\begin{rmk}
	
	In this section, $A_{\infty}$-categories are always assumed to be \emph{strictly unital} (see \cite[paragraph (2a)]{Sei08}).
	
\end{rmk}

\subsection{Definitions and main results}

\subsubsection{Definitions}

\begin{defin}\label{definition mapping torus}
	
	Let $\tau$ be a quasi-autoequivalence of an Adams-graded $A_{\infty}$-category $\mcA$. 
	The mapping torus of $\tau$ is the $A_{\infty}$-category
	\[\mathrm{MT} (\tau) := \mathrm{hocolim} \left( 
	\begin{tikzcd}
	\mcA \sqcup \mcA \ar[r, "\mathrm{id} \sqcup \tau"] \ar[d, "\mathrm{id} \sqcup \mathrm{id}" left] & \mcA \\
	\mcA
	\end{tikzcd} \right) \]
	(see Definition \ref{definition Grothendieck construction}).
	
\end{defin}

\begin{rmks}
	
	\begin{enumerate}
		
		\item We use the terminology ``mapping torus'' by analogy with the analogous situation in the category of topological spaces. Indeed, if $f$ is an automorphism of some topological space $X$, then the mapping torus of $f$
		\[M_f = \left( X \times \left[ 0,1 \right] \right) / \left( \left( x, 0 \right) \sim \left( f(x), 1 \right) \right) \]
		is the homotopy colimit of the following diagram
		\[\begin{tikzcd}
			X \sqcup X \ar[r, "\mathrm{id} \sqcup f"] \ar[d, "\mathrm{id} \sqcup \mathrm{id}" left] & X \\
			X .
		\end{tikzcd} \]
	
		\item The terminology ``mapping torus of an autoequivalence of $A_{\infty}$-categories'' also appears in \cite{Kar21}, and it is used in \cite{Kar21bis} in order to distinguish open symplectic mapping tori. We do not know if the two notions (the one of \cite{Kar21} and the one of Definition \ref{definition mapping torus}) coincide.
		
		\item The mapping torus of a quasi-autoequivalence is also Adams-graded, because it is the localization of an Adams-graded $A_{\infty}$-category at morphisms of Adams-degree $0$.  
		
	\end{enumerate}
	
\end{rmks}

\begin{defin}\label{definition group-action}
	
	Let $\mcA$ be an $A_{\infty}$-category. A $\mathbf{Z}$-splitting of $\mathrm{ob} \left( \mcA \right)$ is a bijection 
	\[\mathbf{Z} \times \mcE \xrightarrow{\sim} \mathrm{ob} \left( \mcA \right), \quad \left( n, E \right) \mapsto X^n \left( E \right) . \]
	If such a splitting has been chosen, we define the Adams-grading of an homogeneous element $x \in \mcA \left( X^i (E), X^j (E) \right)$ to be $(j-i)$. This turns $\mcA$ into an Adams-graded $A_{\infty}$-category.
	
	Let $\tau$ be a quasi-autoequivalence of $\mcA$.
	We say that a $\mathbf{Z}$-splitting of $\mathrm{ob} \left( \mcA \right)$ is compatible with $\tau$ if 
	\[\tau \left( X^n \left( E \right) \right) = X^{n+1} \left( E \right) \]
	for every $n \in \mathbf{Z}$ and $E \in \mcE$. 
	
	We say that $\mcA$ is weakly directed with respect to a $\mathbf{Z}$-splitting of $\mathrm{ob} \left( \mcA \right)$ if 
	\[\mcA \left( X^i (E), X^j (E') \right) = 0 \]
	for every $i > j$ and $E, E' \in \mcE$ (we use the term ``weakly directed'' $A_{\infty}$-category because the notion is slightly more general than that of directed $A_{\infty}$-category defined by Seidel in \cite[paragraph (5m)]{Sei08}).
	
\end{defin}

\begin{rmk}
	
	Compatible $\mathbf{Z}$-splittings naturally arise in the context of $\mathbf{Z}$-actions. 
	A strict $\mathbf{Z}$-action on an $A_{\infty}$-category $\mcA$ is a family of $A_{\infty}$-endofunctors $\left( \tau_n \right)_{n \in \mathbf{Z}}$ such that $\tau_0 = \mathrm{id}_{\mcA}$ and $\tau_{i+j} = \tau_i \circ \tau_j$ (see \cite[paragraph (10b)]{Sei08}). 
	If the induced $\mathbf{Z}$-action on $\mathrm{ob} \left( \mcA \right)$ is free, then any section $\sigma$ of the projection $\mathrm{ob} ( \mcA ) \to \mcE$, where $\mcE$ is the set of equivalence classes of objects in $\mcA$ under the $\mathbf{Z}$-action, gives a $\mathbf{Z}$-splitting 
	\[\mathbf{Z} \times \mcE \xrightarrow{\sim} \mathrm{ob} \left( \mcA \right), \quad \left( n, E \right) \mapsto \tau_n \left( \sigma \left( E \right) \right) \]
	which is compatible with the automorphism $\tau_1$. 
	
\end{rmk}

\subsubsection{Main results}

\paragraph{First result.}

\begin{defin}\label{definition category of coinvariants}
	
	Let $\tau$ be a quasi-autoequivalence of an $A_{\infty}$-category $\mcA$ equipped with a compatible $\mathbf{Z}$-splitting of $\mathrm{ob} \left( \mcA \right)$.
	Assume that $\tau$ is strict, i.e $\tau^d=0$ for $d \geq 2$, and acts bijectively on hom-sets. 
	In this case, we define an Adams-graded $A_{\infty}$-category $\mcA_{\tau}$ as follows:
	\begin{enumerate}
		
		\item the set of objects of $\mcA_{\tau}$ is $\mcE$,
		
		\item the space of morphisms $\mcA_{\tau} \left( E, E' \right)$ is the Adams-graded vector space given by
		\[\mcA_{\tau} \left( E, E' \right) = \left( \bigoplus \limits_{i,j \in \mathbf{Z}} \mcA \left( X^i \left( E \right), X^j \left( E' \right) \right) \right) / \left( \tau (x) \sim x \right) \]
		
		\item the operations are the unique linear maps such that for every sequence 
		\[ \left( x_0, \dots, x_{d-1} \right) \in \mcA \left( X^{i_0} \left( E_0 \right), X^{i_1} \left( E_1 \right) \right) \times \dots \times \mcA \left( X^{i_{d-1}} \left( E_{d-1} \right), X^{i_d} \left( E_d \right) \right), \]
		we have 
		\[\mu_{\mcA_{\tau}} \left( [ x_0 ], \dots, [ x_{d-1} ] \right) = \left[ \mu_{\mcA} \left( x_0, \dots, x_{d-1} \right) \right], \]
		where $[\cdot] : \mcA \left( X^i \left( E \right), X^j \left( E' \right) \right) \to \mcA_{\tau} \left( E, E' \right)$ denotes the projection. (It is not hard to see that such operations exist and satisfy the $A_{\infty}$-relations.) 
		
	\end{enumerate}
	
\end{defin}

\begin{thm}\label{thm mapping torus in strict situation}

	Let $\tau$ be a quasi-autoequivalence of an $A_{\infty}$-category $\mcA$ equipped with a compatible $\mathbf{Z}$-splitting of $\mathrm{ob} \left( \mcA \right)$.
	Assume that $\tau$ is strict and acts bijectively on hom-sets.
	Then there is a quasi-equivalence of Adams-graded $A_{\infty}$-categories 
	\[\mathrm{MT} (\tau) \simeq \mcA_{\tau}. \] 
	
\end{thm}

\begin{rmk}
	
	The $A_{\infty}$-category $\mcA_{\tau}$ is the (ordinary) colimit of the diagram used to define $\mathrm{MT} (\tau)$. Thus, Theorem \ref{thm mapping torus in strict situation} can be thought of as a ``homotopy colimit equals colimit" result.
	
\end{rmk}

\paragraph{Second result.}

We denote by $\mathbf{F} \left[ t_m \right]$ the augmented Adams-graded associative algebra generated by a variable $t_m$ of bidegree $(m, 1)$.
Observe that if $\mcC$ is a subcategory of an $A_{\infty}$-category $\mcD$ with $\mathrm{ob} (\mcC) = \mathrm{ob} (\mcD)$, then $\mcC \oplus (t \mathbf{F} [t] \otimes \mcD)$ is naturally an Adams-graded $A_{\infty}$-category, where the Adams degree of $t^k \otimes x$ equals $k$.
Besides, if $\mcC$ is an $A_{\infty}$-category equipped with a $\mathbf{Z}$-splitting of $\mathrm{ob} \left( \mcC \right)$, we denote by $\mcC^0$ the full $A_{\infty}$-subcategory of $\mcC$ whose set of objects corresponds to $\{ 0 \} \times \mcE$.
Finally, we use the functor $\mcC \mapsto \mcC_m$ of Definition \ref{definition forgetful functor}. 

\begin{thm}[Theorem \ref{thm mapping torus in weak situation introduction} in the Introduction]\label{thm mapping torus in weak situation}
	 
	Let $\tau$ be a quasi-autoequivalence of an $A_{\infty}$-category $\mcA$, weakly directed with respect to some compatible $\mathbf{Z}$-splitting of $\mathrm{ob} \left( \mcA \right)$.
	Assume that there exists a closed degree $0$ bimodule map $f : \mcA_m \left( -, - \right) \to \mcA_m \left( -, \tau(-) \right)$ such that $f : \mcA_m \left( X^i(E) , X^j(E') \right) \to \mcA_m \left( X^i(E), X^{j+1} (E') \right)$ is a quasi-somorphism for every $i < j$ and $E, E' \in \mcE$.
	Then there is a quasi-equivalence of Adams-graded $A_{\infty}$-categories 
	\[\mathrm{MT} (\tau) \simeq \mcA_m^0 \oplus \left( t_m \mathbf{F} \left[ t_m \right] \otimes \mcA_m \left[ f \left( \mathrm{units} \right)^{-1} \right]^0 \right) . \]
	
\end{thm}

\paragraph{Outline of the section} 

In section \ref{subsection resuts about a specific module}, we consider an $A_{\infty}$-category $\mcA$ equipped with a $\mathbf{Z}$-splitting of $\mathrm{ob} \left( \mcA \right)$ and a choice of a closed degree $0$ morphism $c_n (E) \in \mcA \left( X^n (E), X^{n+1} (E) \right)$ for every $n \in \mathbf{Z}$ and every $E \in \mcE$. We give technical results about specific modules associated to this data. This will be used in the proof of Theorem \ref{thm mapping torus in weak situation} with $c_n \left( E \right) = f \left( e_{X^n(E)} \right)$.

In section \ref{subsection category and modules for the mapping torus}, we consider the Grothendieck construction $\mcG$ of a slightly different diagram than the one in Definition \ref{definition mapping torus}, together with a set $W_{\mcG}$ of closed degree 0 morphisms. The idea is that the localization $\mcH = \mcG \left[ W_{\mcG}^{-1} \right]$ is the homotopy colimit of a diagram obtained from the one in Definition \ref{definition mapping torus} by a cofibrant replacement of the diagonal functor $\mcA \sqcup \mcA \to \mcA$. Thus it is not surprising that $\mcH$ is quasi-equivalent to the mapping torus of $\tau$.
Moreover, we prove technical results about specific modules over $\mcG$ that will be used in the proofs of Theorems \ref{thm mapping torus in strict situation} and \ref{thm mapping torus in weak situation}. 
  
In section \ref{subsection proof of the first result}, we prove Theorem \ref{thm mapping torus in strict situation}. We first define an $A_{\infty}$-functor $\Phi : \mcG \to \mcA_{\tau}$ which sends $W_{\mcG}$ to the set of units in $\mcA_{\tau}$. Then we prove that the induced $A_{\infty}$-functor $\widetilde{\Phi} : \mcH \to \mcA_{\tau} \left[ \{ \text{units} \}^{-1} \right]$ is a quasi-equivalence. To do that, our strategy is to apply Proposition \ref{prop quasi-isomorphism between localizations} using the results of section \ref{subsection category and modules for the mapping torus} about the specific $\mcG$-modules. 

In section \ref{subsection proof of the second result}, we prove Theorem \ref{thm mapping torus in weak situation}. We use the fact that $\mcG$ is ``big enough'' (there are more objects and morphisms than in the Grothendieck construction of the diagram in Definition \ref{definition mapping torus}) in order to define an $A_{\infty}$-functor $\Psi_m : \mcG_m \to \mcA_m$ (see Definition \ref{definition forgetful functor}). This induces an $A_{\infty}$-functor $\widetilde{\Psi} : \mcH \to \mathbf{F} \left[ t_m \right] \otimes \mcA_m \left[ f \left( \{ \text{units} \} \right)^{-1} \right]$. Then we prove that for every Adams degree $j \geq 1$, and for every objects $X,Y$ in $\mcH$, the map
\[\widetilde{\Psi} : \mcH \left( X,Y \right)^{*,j} \to \left( \mathbf{F} \left[ t_m \right] \otimes \mcA_m \left[ f \left( \{ \text{units} \} \right)^{-1} \right] \right) \left( \Psi X , \Psi Y \right)^{*,j} \]
is a quasi-isomorphism (if $V$ is an Adams-graded vector space, $V^{*,j}$ denotes the subspace of Adams degree $j$ elements).
To do that, we apply once again Proposition \ref{prop quasi-isomorphism between localizations} using the results of sections \ref{subsection resuts about a specific module} and \ref{subsection category and modules for the mapping torus} about the specific modules over $\mcA_m$ and $\mcG$ respectively.
This allows us to finish the proof of Theorem \ref{thm mapping torus in weak situation}.

\subsection{Results about specific modules}\label{subsection resuts about a specific module}

In this section, we give technical results that will allow us to apply Proposition \ref{prop quasi-isomorphism between localizations} in the proof of Theorem \ref{thm mapping torus in weak situation}.

Let $\mcA$ be an $A_{\infty}$-category equipped with a $\mathbf{Z}$-splitting of $\mathrm{ob} \left( \mcA \right)$.
Assume that we chose, for every $n \in \mathbf{Z}$ and every $E \in \mcE$, a closed degree $0$ morphism $c_n (E) \in \mcA \left( X^n (E), X^{n+1} (E) \right)$. Moreover, assume that we chose a set $W_{\mcA}$ of closed degree $0$ morphisms which contains the morphisms $c_n(E)$.

\begin{rmk}
	
	According to Definition \ref{definition group-action}, the $\mathbf{Z}$-splitting of $\mathrm{ob} \left( \mcA \right)$ naturally induces an Adams-grading on $\mcA$. However in this section, we do not consider $\mcA$ as being Adams-graded. 
	
\end{rmk}

In the following, we fix some element $E_{\diamond} \in \mcE$. When we write an object $X^n$ or a morphism $c_n$ without specifying the element of $\mcE$, we mean $X^n (E_{\diamond})$ or $c_n (E_{\diamond})$ respectively. 
Recall that $t_{c_n} : \mcA \left( -, X^n \right) \to \mcA \left( -, X^{n+1} \right)$ denotes the $\mcA$-module map induced by $c_n \in \mcA \left( X^n, X^{n+1} \right)$ (see Definition \ref{definition Yoneda functor}). 

\begin{defin}\label{definition specific module}
	
	We set $\mcM_{\mcA}$ to be the $\mcA$-module
	\[ \mcM_{\mcA} := \left[
	\begin{tikzcd}
	\dots \ar[rd, "t_{c_{-1}}"] & \mcA \left( - , X^0 \right) \ar[d, "\mathrm{id}"] \ar[rd, "t_{c_0}"] & \mcA \left( - , X^1 \right) \ar[d, "\mathrm{id}"] \ar[rd, "t_{c_1}"] & \dots \\
	\dots & \mcA \left( - , X^0 \right) & \mcA \left( - , X^1 \right) & \dots
	\end{tikzcd} 
	\right] \]
	(see Definition \ref{definition cone of a morphism between modules}).
	Besides, we set $t_{\mcA}^n : \mcA \left( - , X^n \right) \to \mcM_{\mcA}$ to be the $\mcA$-module inclusion for every $n \in \mathbf{Z}$.
	
\end{defin}

The first result highlights a key property of the module $\mcM_{\mcA}$.

\begin{lemma}\label{lemma multiplication by continuation element becomes homotopic to inclusion}
	
	For every $n \in \mathbf{Z}$, the closed $\mcA$-module map $t_{\mcA}^{n+1} \circ t_{c_n} : \mcA \left( -, X^n \right) \to \mcM_{\mcA}$ is homotopic to $t_{\mcA}^n : \mcA \left( -, X^n \right) \to \mcM_{\mcA}$. 
	
\end{lemma}

\begin{proof}
	
	Consider the degree $(-1)$ strict $\mcA$-module map $s : \mcA \left( -, X^n \right) \to \mcM_{\mcA}$ which sends a morphism in $\mcA \left( X, X^n \right)$ to the corresponding shifted element in $\mcA \left( X, X^n \right) \left[ 1 \right]$.
	Then an easy computation gives
	\[\mu_{\mathrm{Mod}_{\mcA}}^1 \left( s \right) = t_{\mcA}^{n+1} \circ t_{c_n} + t_{\mcA}^n . \]
	This concludes the proof. 
	
\end{proof}

In the proof of the two results below, we will use specific $\mcA$-modules. If $p$ is a fixed non negative integer, we set 
\[K_p = \left[
\begin{tikzcd}
\dots \ar[rd, "t_{c_{p-2}}"] & \mcA \left( - , X^{p-1} \right) \ar[d, "\mathrm{id}"] \ar[rd, "t_{c_{p-1}}"] & \mcA \left( -  , X^p \right) \ar[d, "\mathrm{id}"] \\
\dots & \mcA \left( - , X^{p-1} \right) & \mcA \left( -  , X^p \right)
\end{tikzcd} 
\right] \]
and
\[\widetilde{K}_p = \left[
\begin{tikzcd}
\mcA \left( -  , X^p \right) \ar[rd, "t_{c_p}"] & \mcA \left( -  , X^{p+1} \right) \ar[d, "\mathrm{id}"] \ar[rd, "t_{c_{p+1}}"] & \dots \\
& \mcA \left( -  , X^{p+1} \right) & \dots 
\end{tikzcd} 
\right] . \]
Moreover, we will consider the sequences of $\mcA$-modules $\left( F_p^q \right)_{q \geq 0}$, $\left( \widetilde{F}_p^q \right)_{q \geq 0}$ starting at $F_p^0 = \widetilde{F}_p^0 = 0$ and with
\[ F_p^q = \left[
\begin{tikzcd}
\mcA \left( - , X^{p-q+1} \right) \ar[d, "\mathrm{id}"] \ar[rd, "t_{c_{p-q+1}}"] & \dots \ar[rd, "t_{c_{p-1}}"] & \mcA \left( - , X^p \right) \ar[d, "\mathrm{id}"] \\
\mcA \left( - , X^{p-q+1} \right) & \dots & \mcA \left( - , X^p \right)
\end{tikzcd}
\right] \]
and 
\[\widetilde{F}_p^q = \left[
\begin{tikzcd}
\mcA \left( - , X^p \right) \ar[rd, "t_{c_p}"] & \dots  \ar[rd, "t_{c_{p+q-1}}"] & \\
& \dots & \mcA \left( - , X^{p+q} \right)
\end{tikzcd} 
\right] \]
for $q \geq 1$.

The following Lemma is mostly technical. It will be used in the proofs of Lemmas \ref{lemma second result for specific module} and \ref{lemma G-module map for relation G-A}.

\begin{lemma}\label{lemma first result for specific module}
	
	Assume that for every $i < j$, for every $E \in \mcE$, the chain map
	\[\mu_{\mcA}^2 \left( - , c_j \right) : \mcA \left( X^i (E) , X^j \right) \to \mcA \left( X^i (E) , X^{j+1} \right) \]
	is a quasi-isomorphism. 
	Then for every $k < n$, for every $E \in \mcE$, the inclusion $\mcA \left( X^k (E), X^n \right) \hookrightarrow \mcM_{\mcA} \left( X^k (E) \right)$ is a quasi-isomorphism.
	
\end{lemma}

\begin{proof}
	
	The cone of the inclusion $\mcA \left( X^k (E)  , X^n \right) \hookrightarrow  \mcM_{\mcA} \left( X^k (E)  \right)$ is quasi-isomorphic to its cokernel, which is $K_{n-1} \left( X^k (E) \right) \oplus \widetilde{K}_n \left( X^k (E) \right)$.
	
	We have to show that these complexes are acyclic. 
	Observe that $\left( F_{n-1}^q \left( X^k (E) \right) \right)_{q \geq 0}$ and $\left( \widetilde{F}_n^q \left( X^k (E) \right) \right)_{q \geq 0}$ are increasing, exhaustive, and bounded from below filtrations of $K_{n-1} \left( X^k (E) \right)$ and $\widetilde{K}_n \left( X^k (E) \right)$ respectively.
	For every $q \geq 1$, we have 
	\[F_{n-1}^q \left( X^k (E) \right) / F_{n-1}^{q-1} \left( X^k (E) \right) = \left[
	\begin{tikzcd}
	\mcA \left( X^k (E) , X^{n-q} \right) \ar[d, "\mathrm{id}"] \\
	\mcA \left( X^k (E) , X^{n-q} \right)
	\end{tikzcd} \right] \]
	and
	\[\widetilde{F}_n^q \left( X^k (E) \right) / \widetilde{F}_n^{q-1} \left( X^k (E) \right) = \left[
	\begin{tikzcd}
	\mcA \left( X^k (E) , X^{n+q-1} \right) \ar[rd, "t_{c_{n+q-1}}"] & \\
	& \mcA \left( X^k (E) , X^{n+q} \right) 
	\end{tikzcd} \right] . \]
	The first of the two latter complexes is clearly acyclic, and the second one is acyclic by assumption on the morphisms $c_j$. Thus the entire complex $K_{n-1} \left( X^k (E) \right) \oplus \widetilde{K}_n \left( X^k (E) \right)$ is acyclic, which is what we needed to prove. 
	
\end{proof}

The following two lemmas will be used later in order to apply Proposition \ref{prop quasi-isomorphism between localizations}. 

\begin{lemma}\label{lemma second result for specific module}
	
	Assume that for every $i < j < k$, for every $E \in \mcE$, the chain maps
	\[\left\{
	\begin{array}{ccccl}
	\mu_{\mcA}^2 \left( - , c_j \right) & : & \mcA \left( X^i (E) , X^j \right) & \to & \mcA \left( X^i (E) , X^{j+1} \right) \\
	\mu_{\mcA}^2 \left( c_j (E) , - \right) & : & \mcA \left( X^{j+1} (E), X^{k+1} \right) & \to & \mcA \left( X^j (E), X^{k+1} \right)
	\end{array}
	\right. \]
	are quasi-isomorphisms. 
	Then for every $(n, E) \in \mathbf{Z} \times \mcE$, the complex $\mcM_{\mcA} \left( \mathrm{Cone} \, c_n (E) \right)$ is acyclic.
	
\end{lemma}

\begin{proof}
	
	We have
	\[\mcM_{\mcA} \left( \mathrm{Cone} \, c_n (E) \right) = \mathrm{Cone} \left( \mcM_{\mcA} \left( X^{n+1} (E)  \right) \xrightarrow{\mu_{\mcM_{\mcA}}^2 \left( c_n (E) , - \right)} \mcM_{\mcA} \left( X^n (E)  \right) \right), \]
	so we have to prove that $\mu_{\mcM_{\mcA}}^2 \left( c_n (E) , - \right) : \mcM_{\mcA} \left( X^{n+1} (E)  \right) \to \mcM_{\mcA} \left( X^n (E)  \right)$ is a quasi-isomorphism. 
	Observe that we have the following commutative diagram
	\[\begin{tikzcd}[column sep = 2.5cm]
	\mcM_{\mcA} \left( X^{n+1} (E)  \right) \ar[r, "\mu_{\mcM_{\mcA}}^2 {(c_n (E),-)}"] & \mcM_{\mcA} \left( X^n (E)  \right) \\
	\mcA \left( X^{n+1} (E)  , X^{n+2} \right) \ar[u, hook] \ar[r, "\mu_{\mcA}^2 {\left( c_n (E) , - \right)}"] & \mcA \left( X^n (E)  , X^{n+2} \right) \ar[u, hook] .
	\end{tikzcd} \]
	The bottom horizontal map is a quasi-isomorphism by assumption on the morphisms $c_j (E)$. Moreover, the vertical maps are quasi-isomorphisms according to Lemma \ref{lemma first result for specific module}. This implies that $\mu_{\mcM_{\mcA}}^2 \left( c_n (E) , - \right)$ is indeed a quasi-isomorphism. 
	
\end{proof}

\begin{lemma}\label{lemma third result for specific module}
	
	The $\mcA$-module map $_{W_{\mcA}^{-1}} t_{\mcA}^n : \, _{W_{\mcA}^{-1}} \mcA \left( - , X^n \right) \to \, _{W_{\mcA}^{-1}} \mcM_{\mcA}$ is a quasi-isomorphism for every $n \in \mathbf{Z}$.
	
\end{lemma}

\begin{proof}
	
	Let $X$ be some object of $\mcA$. We want to prove that the chain map $_{W_{\mcA}^{-1}} t_{\mcA}^n : \, _{W_{\mcA}^{-1}} \mcA \left( X , X^n \right) \to \, _{W_{\mcA}^{-1}} \mcM_{\mcA} \left( X \right)$ is a quasi-isomorphism. Observe that 
	\[_{W_{\mcA}^{-1}} \mcM_{\mcA} \left( X \right) = \left[
	\begin{tikzcd}
		\dots \ar[rd] & \mcA \left[ W_{\mcA}^{-1} \right] \left( X , X^0 \right) \ar[d, "\mathrm{id}"] \ar[rd, "_{W_{\mcA}^{-1}} t_{c_0}" below = 2.5] & \mcA \left[ W_{\mcA}^{-1} \right] \left( X , X^1 \right) \ar[d, "\mathrm{id}"] \ar[rd, "_{W_{\mcA}^{-1}} t_{c_1}"] & \dots \\
		\dots & \mcA \left[ W_{\mcA}^{-1} \right] \left( X , X^0 \right) & \mcA \left[ W_{\mcA}^{-1} \right] \left( X , X^1 \right) & \dots
	\end{tikzcd} 
	\right] \]
	and the chain map $_{W_{\mcA}^{-1}} t_{\mcA}^n : \, _{W_{\mcA}^{-1}} \mcA \left( X , X^n \right) \to \, _{W_{\mcA}^{-1}} \mcM_{\mcA} \left( X \right)$ is the inclusion. The cone of the latter is then quasi-isomorphic to its cokernel, which is $_{W_{\mcA}^{-1}} K_{n-1} \left( X \right) \oplus \, _{W_{\mcA}^{-1}} \widetilde{K}_n \left( X \right)$.
	Observe that $\left( _{W_{\mcA}^{-1}} F_{n-1}^q \left( X \right) \right)_{q \geq 0}$, $\left( _{W_{\mcA}^{-1}} \widetilde{F}_n^q \left( X \right) \right)_{q \geq 0}$ are increasing, exhaustive, and bounded from below filtrations of $_{W_{\mcA}^{-1}} K_{n-1} \left( X \right)$, $_{W_{\mcA}^{-1}} \widetilde{K}_n \left( X \right)$ respectively. For every $q \geq 1$, we have 
	\[_{W_{\mcA}^{-1}} F_{n-1}^q \left( X \right) / \, _{W_{\mcA}^{-1}} F_{n-1}^{q-1} \left( X \right) = \left[
	\begin{tikzcd}
		\mcA \left[ W_{\mcA}^{-1} \right] \left( X , X^{n-q} \right) \ar[d, "\mathrm{id}"] \\
		\mcA \left[ W_{\mcA}^{-1} \right] \left( X , X^{n-q} \right)
	\end{tikzcd} \right] \]
	and
	\[_{W_{\mcA}^{-1}} \widetilde{F}_n^q \left( X \right) / \, _{W_{\mcA}^{-1}} \widetilde{F}_n^{q-1} \left( X \right) =
	\left[
	\begin{tikzcd}
		\mcA \left[ W_{\mcA}^{-1} \right] \left( X , X^{n-1+q} \right) \ar[rd, "_{W_{\mcA}^{-1}} t_{c_{n-1+q}}"] & \\
		& \mcA \left[ W_{\mcA}^{-1} \right] \left( X , X^{n+q} \right) 
	\end{tikzcd} \right] . \]
	The first of the two latter complexes is clearly acyclic, and the second one is acyclic because $c_{n-1+q}$ belongs to the set $W_{\mcA}$ by which we localized (see \cite[Lemma 3.12]{GPS20}). Thus the entire complex $_{W_{\mcA}^{-1}} K_{n-1} \left( X \right) \oplus \, _{W_{\mcA}^{-1}} \widetilde{K}_n \left( X \right)$ is acyclic, which is what we needed to prove.
	
\end{proof}

\subsection{The $A_{\infty}$-category and modules for the mapping torus}\label{subsection category and modules for the mapping torus}

In this section, we consider an $A_{\infty}$-category $\mcG$, together with a set $W_{\mcG}$ of closed degree 0 morphisms. We prove that $\mcH = \mcG \left[ W_{\mcG}^{-1} \right]$ is quasi-equivalent to the mapping torus of $\tau$, and we prove technical results about specific $\mcG$-modules that will allow us to apply Proposition \ref{prop quasi-isomorphism between localizations} in the proofs of Theorems \ref{thm mapping torus in strict situation} and \ref{thm mapping torus in weak situation}. 

Let $\tau$ be a quasi-autoequivalence of an $A_{\infty}$-category $\mcA$ equipped with a compatible $\mathbf{Z}$-splitting of $\mathrm{ob} \left( \mcA \right)$.
If $\mcA_{\triangle}$ is a copy of $\mcA$, we denote by $X_{\triangle}^n \left( E \right)$ the object of $\mcA_{\triangle}$ corresponding to $\left( n, E \right) \in \mathbf{Z} \times \mcE$.

\subsubsection{The Grothendieck construction $\mcG$}

The $A_{\infty}$-category $\mcG$ will be the Grothendieck construction of a slightly different diagram than the one in Definition \ref{definition mapping torus}. The idea is to introduce an $A_{\infty}$-category $\mcC$ together with a set of closed degree $0$ morphisms $W_{\mcC}$ such that the localization $\mcC \left[ W_{\mcC}^{-1} \right]$ is a cylinder object for $\mcA$. 
Observe that this kind of cofibrant replacement is common in homotopy colimits computation, and indeed we need it to prove Theorem \ref{thm mapping torus in weak situation}.

\begin{defin}\label{definition cylinder object}
	
	Let $\mcA_{\bot}$, $\mcA_I$ and $\mcA_{\top}$ be three copies of $\mcA$. 
	We denote by $\mcC$ the Grothendieck construction (see Definition \ref{definition Grothendieck construction}) of the following diagram
	\[\begin{tikzcd}
	\mcA_I \ar[r, "\mathrm{id}"] \ar[d, "\mathrm{id}" left] & \mcA_{\top} \\
	\mcA_{\bot} &
	\end{tikzcd} \]
	and we let $\iota_{\bot}, \iota_I, \iota_{\top} : \mcA \to \mcC$ be the strict inclusions with images $\mcA_{\bot}$, $\mcA_I$, $\mcA_{\top}$ respectively. Finally, we denote by $W_{\mcC}$ the set of adjacent units in $\mcC$, and we let $\mcC yl_{\mcA} = \mcC \left[ W_{\mcC}^{-1} \right]$ be the homotopy colimit of the diagram above.
	
\end{defin}
 
\begin{defin}\label{definition category G}
	
	Let $\mcA_-$, $\mcA_{+}$, $\mcA_{\bullet}$ be three copies of $\mcA$.
	We denote by $\mcG$ the Grothendieck construction of the following diagram
	\[\begin{tikzcd}
	\mcA_- \sqcup \mcA_+ \ar[r, "\mathrm{id} \sqcup \tau"] \ar[d, "\iota_{\bot} \sqcup \iota_{\top}" left] & \mcA_{\bullet} \\
	\mcC .
	\end{tikzcd} \]
	Besides, we denote by $W_{\mcG}$ the union of $W_{\mcC}$ and the set of adjacent units in $\mcG$, and we set
	\[\mcH := \mcG \left[ W_{\mcG}^{-1} \right]. \]
	
\end{defin}

According to Proposition \ref{prop cylinder object}, $\mcC yl_{\mcA}$ can be thought as a cylinder object for $\mcA$.
Therefore, the following result should not be surprising. 

\begin{lemma}\label{lemma relation G - mapping torus}
	
	The mapping torus of $\tau$ is quasi-equivalent to $\mcH$. 
	
\end{lemma}

\begin{proof}
	
	Let $\pi : \mcC \to \mcA$ be the $A_{\infty}$-functor induced by the following commutative diagram 
	\[\begin{tikzcd}
	\mcA_I \ar[r, "\mathrm{id}"] \ar[d, "\mathrm{id}"] & \mcA_{\top} \ar[d, "\mathrm{id}"] \\
	\mcA_{\bot} \ar[r, "\mathrm{id}"] & \mcA
	\end{tikzcd} \]
	(see Proposition \ref{prop induced functor from Grothendieck construction}).
	We get a commutative diagram 
	\[\begin{tikzcd}
	\mcC \ar[d, "\pi"] & \mcA_- \sqcup \mcA_+ \ar[l, "\iota_{\bot} \sqcup \iota_{\top}" above] \ar[r, "\mathrm{id} \sqcup \tau"] \ar[d, "\mathrm{id}"] & \mcA_{\bullet} \ar[d, "\mathrm{id}"] \\
	\mcA & \mcA_- \sqcup \mcA_+ \ar[l, "\mathrm{id} \sqcup \mathrm{id}"] \ar[r, "\mathrm{id} \sqcup \tau" below] & \mcA_{\bullet} 
	\end{tikzcd} \]
	which induces an $A_{\infty}$-functor $\chi$ from $\mcG$ to the Grothendieck construction of the bottom line (see Proposition \ref{prop invariance of homotopy colimits}).
	Observe that $\chi$ sends $W_{\mcC}$ to the set $U$ of units in $\mcA$. 
	Now, according to Proposition \ref{prop cylinder object}, the $A_{\infty}$-functor $\widetilde{\pi} : \mcC yl_{\mcA} = \mcC \left[ W_{\mcC}^{-1} \right] \to \mcA \left[ U^{-1} \right]$ is a quasi-equivalence. According to Lemma A.6 in \cite{GPS19} (called ``localization and homotopy colimits commute''), this implies that the $A_{\infty}$-functor induced by $\chi$
	\[\mathrm{hocolim} \left(  
	\begin{tikzcd}[column sep = 0.5cm]
	\mcA_- \sqcup \mcA_+ \ar[d, "\iota_{\bot} \sqcup \iota_{\top}"] \ar[r, "\mathrm{id} \sqcup \tau"] & \mcA_{\bullet} \\
	\mcC
	\end{tikzcd}
	\right) \left[ W_{\mcC}^{-1} \right] \xrightarrow{\widetilde{\chi}}
	\mathrm{hocolim} \left(  
	\begin{tikzcd}[column sep = 0.5cm]
	\mcA_- \sqcup \mcA_+ \ar[d, "\mathrm{id} \sqcup \mathrm{id}"] \ar[r, "\mathrm{id} \sqcup \tau"] & \mcA_{\bullet} \\
	\mcA
	\end{tikzcd}
	\right) \left[ U^{-1} \right] \]
	is a quasi-equivalence.
	This completes the proof because the source of $\widetilde{\chi}$ is exactly $\mcH$, and its target is quasi-equivalent to the mapping torus of $\tau$. 
	
\end{proof}

\subsubsection{Modules over $\mcG$}

In the following, we fix some element $E_{\diamond} \in \mcE$. When we write an object $X_{\triangle}^n$ without specifying the element of $\mcE$, we mean $X_{\triangle}^n (E_{\diamond})$.
Moreover, we denote by 
\[t_{\triangle \square}^n : \mcG \left( -, X_{\triangle}^n \right) \to \mcG \left( -, X_{\square}^{n+\delta_{\triangle \square}} \right) \]
the $\mcG$-module map induced by the adjacent unit in $\mcG \left( X_{\triangle}^n, X_{\square}^{n+\delta_{\triangle \square}} \right)$ (see Definition \ref{definition Yoneda functor}), where 
\[\delta_{\triangle \square} = \left\{
\begin{array}{ll}
1 & \text{if } \left( \triangle, \square \right) = \left( +, \bullet \right) \\
0 & \text{otherwise} .
\end{array}
\right. \]

\begin{defin}\label{definition module for G}
	
	We denote by $\mcM_{\mcG}$ the $\mcG$-module defined by
	\[\mcM_{\mcG} = 
	\left[ 
	\begin{tikzcd}[column sep = scriptsize]
	\dots \ar[rd, "t_{+ \bullet}^{-1}"] & \mcG \left( - , X_-^{0} \right) \ar[d, "t_{- \bullet}^0"] \ar[rd, "t_{- \bot}^0"] & \mcG \left( - , X_I^{0} \right) \ar[d, "t_{I \bot}^0"] \ar[rd, "t_{I \top}^0"] & \mcG \left( - , X_+^{0} \right) \ar[d, "t_{+ \top}^0"] \ar[rd, "t_{+ \bullet}^0"] & \mcG \left( - , X_{-}^{1} \right) \ar[d, "t_{- \bullet}^1"] \ar[rd, "t_{- \bot}^1"] & \dots \\
	\dots & \mcG \left( - , X_{\bullet}^{0} \right) & \mcG \left( - , X_{\bot}^0 \right) & \mcG \left( - , X_{\top}^0 \right) & \mcG \left( - , X_{\bullet}^{1} \right) & \dots 
	\end{tikzcd}
	\right] \]
	(see Definition \ref{definition cone of a morphism between modules}). 
	For practical reasons, we also consider the $\mcG$-modules
	\[ \mcM_{\star}^n := 
	\left[
	\begin{tikzcd}
	& \mcG \left( - , X_I^n \right) \ar[ld, "t_{I \bot}^n"] \ar[rd, "t_{I \top}^n"] & \\
	\mcG \left( - , X_{\bot}^n \right) & & \mcG \left( - , X_{\top}^n \right)
	\end{tikzcd}
	\right], \quad n \in \mathbf{Z}. \]
	Besides, we denote by $t_{\mcG} : \mcG \left( - , X_{\bullet}^0 \right) \to \mcM_{\mcG}$ the $\mcG$-module inclusion.
	
\end{defin}

\begin{rmk}
	
	We can write 
	\[ \mcM_{\mcG} = 
	\left[
	\begin{tikzcd}
	& \bigoplus \limits_{n \in \mathbf{Z}} \left( \mcG \left( - , X_-^n \right) \oplus \mcG \left( - , X_+^n \right) \right) \ar[ld, "\bigoplus \limits_{n \in \mathbf{Z}} \left( t_{- \bot}^n \oplus t_{+ \top}^n \right)"] \ar[rd, "\bigoplus \limits_{n \in \mathbf{Z}} \left( t_{- \bullet}^n \oplus t_{+ \bullet}^n \right)"] & \\
	\bigoplus \limits_{n \in \mathbf{Z}} \mcM_{\star}^n & & \bigoplus \limits_{n \in \mathbf{Z}} \mcG \left( - , X_{\bullet}^n \right)
	\end{tikzcd}
	\right] . \]
	
\end{rmk}

The following two lemmas are analogs of Lemmas \ref{lemma second result for specific module} and \ref{lemma third result for specific module} respectively. They will be used later in order to apply Proposition \ref{prop quasi-isomorphism between localizations}.

\begin{lemma}\label{lemma properties of modules}
	
	For every $w$ in $W_{\mcG}$, the complex $\mcM_{\mcG} \left( \mathrm{Cone} \, w \right)$ is acyclic.
	
\end{lemma}

\begin{proof}
	
	Let $w$ be the morphism in $W_{\mcG} \cap \mcG \left( X_I^k (E) , X_{\top}^k (E) \right)$ (the proof is analogous for the morphism in $W_{\mcG} \cap \mcG \left( X_I^k (E) , X_{\bot}^k (E) \right)$). Then 
	\begin{align*}
	\mcM_{\mcG} \left( \mathrm{Cone} \, w \right) & = \mathrm{Cone} \left( \mcM_{\mcG} \left( X_{\top}^k (E) \right) \xrightarrow{\mu_{\mcM_{\mcG}}^2 \left( w , - \right)} \mcM_{\mcG} \left( X_I^k (E) \right) \right) \\
	& = \bigoplus\limits_{n} \mathrm{Cone} \left( \mcG \left( X_{\top}^k (E) , X_{\top}^n \right) \xrightarrow{\mu_{\mcM_{\mcG}}^2 \left( w , - \right)} \mcM_{\star}^n \left( X_I^k (E) \right) \right) .
	\end{align*}
	We want to prove that $\mu_{\mcM_{\mcG}}^2 \left( w , - \right) : \mcG \left( X_{\top}^k (E) , X_{\top}^n \right) \to \mcM_{\star}^n \left( X_I^k (E) \right)$ is a quasi-isomorphism for every $n$. Observe that the following diagram of chain complexes is commutative 
	\[\begin{tikzcd}
	\mcG \left( X_{\top}^k (E) , X_{\top}^n \right) \ar[rr, "{ \mu_{\mcM_{\mcG}}^2 \left( w , - \right) }"] & & \mcM_{\star}^n \left( X_I^k (E) \right) \\
	\mcG \left( X_{\top}^k (E) , X_{\top}^n \right) \ar[u, equal] \ar[rr, "{\mu_{\mcG}^2 \left( w, - \right)}"] & & \mcG \left( X_I^k (E) , X_{\top}^n \right) \ar[u, hook] .
	\end{tikzcd} \]
	The rightmost vertical arrow is injective, so its cone is quasi-isomorphic to its cokernel, which is the cone of $t_{I \bot}^n : \mcG \left( X_I^k (E), X_I^n \right) \to \mcG \left( X_I^k (E), X_{\bot}^n \right)$. Since the latter map is a quasi-isomorphism, the cone of $\mu_{\mcM_{\mcG}}^2 \left( w , - \right) : \mcG \left( X_{\top}^k (E) , X_{\top}^n \right) \to \mcM_{\star}^n \left( X_I^k (E) \right)$ is quasi-isomorphic to the cone of $\mu_{\mcG}^2 \left( w, - \right) : \mcG \left( X_{\top}^k (E) , X_{\top}^n \right) \to \mcG \left( X_I^k (E) , X_{\top}^n \right)$. The latter map is a quasi-isomorphism, so we conclude that $\mu_{\mcM_{\mcG}}^2 \left( w , - \right) : \mcG \left( X_{\top}^k (E) , X_{\top}^n \right) \to \mcM_{\star}^n \left( X_I^k (E) \right)$ is a quasi-isomorphism for every $n$, and thus $\mcM_{\mcG} \left( \mathrm{Cone} \, w \right)$ is acyclic. 
	
	Now let $w$ be the morphism in $W_{\mcG} \cap \mcG \left( X_+^k (E) , X_{\top}^k (E)  \right)$ (the proof is analogous for the morphism in $W_{\mcG} \cap \mcG \left( X_-^k (E) , X_{\bot}^k (E) \right)$). Then 
	\begin{align*}
	\mcM_{\mcG} \left( \mathrm{Cone} \, w \right) & = \mathrm{Cone} \left( \mcM_{\mcG} \left( X_{\top}^k (E) \right) \xrightarrow{\mu_{\mcM_{\mcG}}^2 \left( w , - \right)} \mcM_{\mcG} \left( X_+^k (E) \right) \right) \\
	& = \bigoplus\limits_{n} \mathrm{Cone} \left( \mcG \left( X_{\top}^k (E) , X_{\top}^n \right) \xrightarrow{\mu_{\mcM_{\mcG}}^2 \left( w , - \right)} K^n \right).
	\end{align*}
	where 
	\[K^n = \left[ 
	\begin{tikzcd}
	& \mcG \left( X_+^k (E) , X_+^n \right) \ar[ld, "t_{+ \top}^n"] \ar[rd, "t_{+ \bullet}^n"] \\
	\mcG \left( X_+^k (E) , X_{\top}^n \right) & & \mcG \left( X_+^k (E) , X_{\bullet}^{n+1} \right)
	\end{tikzcd}
	\right] . \]
	Observe that $\mu_{\mcM_{\mcG}}^2 \left( w , - \right) : \mcG \left( X_{\top}^k (E) , X_{\top}^n \right) \to K^n$ is injective, so its cone is quasi-isomorphic to its cokernel, which is the cone of $t_{+ \bullet}^n : \mcG \left( X_+^k (E) , X_+^n \right) \to \mcG \left( X_+^k (E) , X_{\bullet}^{n+1} \right)$. The latter map is a quasi-isomorphism because $\tau$ is a quasi-equivalence. This implies that the cone of $\mu_{\mcM_{\mcG}}^2 \left( w , - \right) : \mcG \left( X_{\top}^k (E) , X_{\top}^n \right) \to K^n$ is acyclic for every $n$, and thus $\mcM_{\mcG} \left( \mathrm{Cone} \, w \right)$ is acyclic.
	
	It remains to consider a morphism $w$ in $W_{\mcG} \cap \mcG \left( X_+^k (E) , X_{\bullet}^{k+1} (E) \right)$ (the proof is analogous for the morphism in $W_{\mcG} \cap \mcG \left( X_-^k (E) , X_{\bullet}^k (E) \right)$). Then 
	\begin{align*}
	\mcM_{\mcG} \left( \mathrm{Cone} \, w \right) & = \mathrm{Cone} \left( \mcM_{\mcG} \left( X_{\bullet}^{k+1} (E) \right) \xrightarrow{\mu_{\mcM_{\mcG}}^2 \left( w , - \right)} \mcM_{\mcG} \left( X_+^k (E) \right) \right) \\
	& = \bigoplus\limits_{n} \mathrm{Cone} \left( \mcG \left( X_{\bullet}^{k+1} (E) , X_{\bullet}^{n} \right) \xrightarrow{\mu_{\mcM_{\mcG}}^2 \left( w , - \right)} K^n \right).
	\end{align*}
	where 
	\[K^n = \left[ 
	\begin{tikzcd}
	& \mcG \left( X_+^k (E) , X_+^{n-1} \right) \ar[ld, "t_{+ \top}^{n-1}"] \ar[rd, "t_{+ \bullet}^{n-1}"] \\
	\mcG \left( X_+^k (E) , X_{\top}^{n-1} \right) & & \mcG \left( X_+^k (E) , X_{\bullet}^{n} \right)
	\end{tikzcd}
	\right] . \]
	Observe that $\mu_{\mcM_{\mcG}}^2 \left( w , - \right) : \mcG \left( X_{\bullet}^{k+1} (E) , X_{\bullet}^{n} \right) \to K^n$ is injective, so its cone is quasi-isomorphic to its cokernel, which is the cone of $t_{+ \top}^{n-1} : \mcG \left( X_+^k (E) , X_+^{n-1} \right) \to \mcG \left( X_+^k (E) , X_{\top}^{n-1} \right)$. The latter map is a quasi-isomorphism, so we conclude that the cone of $\mu_{\mcM_{\mcG}}^2 \left( w , - \right) : \mcG \left( X_{\bullet}^{k+1} (E) , X_{\bullet}^{n} \right) \to K^n$ is acyclic for every $n$, and thus $\mcM_{\mcG} \left( \mathrm{Cone} \, w \right)$ is acyclic.
	
\end{proof}

\begin{lemma}\label{lemma properties of modules bis}
	
	The $\mcH$-module map $_{W_{\mcG}^{-1}} t_{\mcG} : \, _{W_{\mcG}^{-1}} \mcG \left( - , X_{\bullet}^0 \right) \to \, _{W_{\mcG}^{-1}} \mcM_{\mcG}$ is a quasi-isomorphism.
	
\end{lemma}

\begin{proof}
	
	We fix an object $X$ in $\mcG$, and we want to prove that $_{W_{\mcG}^{-1}} t_{\mcG} : \, _{W_{\mcG}^{-1}} \mcG \left( X , X_{\bullet}^0 \right) \to \, _{W_{\mcG}^{-1}} \mcM_{\mcG} \left( X \right)$ is a quasi-isomorphism. 
	Observe that
	\[_{W_{\mcG}^{-1}} \mcM_{\mcG} := \left[ 
	\begin{tikzcd}
		\dots \ar[rd] & \mcG \left[ W_{\mcG}^{-1} \right] \left( - , X_+^{-1} \right) \ar[d, "_{W_{\mcG}^{-1}} t_{+ \top}^{-1}"] \ar[rd, "_{W_{\mcG}^{-1}} t_{+ \bullet}^{-1}" below = 2] & \mcG \left[ W_{\mcG}^{-1} \right] \left( - , X_{-}^{0} \right) \ar[d, "_{W_{\mcG}^{-1}} t_{- \bullet}^0"] \ar[rd, "_{W_{\mcG}^{-1}} t_{- \bot}^0"] & \dots \\
		\dots & \mcG \left[ W_{\mcG}^{-1} \right] \left( - , X_{\top}^{-1} \right) & \mcG \left[ W_{\mcG}^{-1} \right] \left( - , X_{\bullet}^0 \right) & \dots 
	\end{tikzcd}
	\right] \]
	and that the chain map $_{W_{\mcG}^{-1}} t_{\mcG} : \, _{W_{\mcG}^{-1}} \mcG \left( X , X_{\bullet}^0 \right) \to \, _{W_{\mcG}^{-1}} \mcM_{\mcG} \left( X \right)$ is the inclusion. The cone of the latter is then quasi-isomorphic to its cokernel, which can be written $K' \oplus K''$ with
	\[K' = \left[ 
	\begin{tikzcd}
		\dots \ar[rd] & \mcG \left[ W_{\mcG}^{-1} \right] \left( X , X_I^{-1} \right) \ar[d, "_{W_{\mcG}^{-1}} t_{I \bot}^{-1}"] \ar[rd, "_{W_{\mcG}^{-1}} t_{I \top}^{-1}" below = 2] & \mcG \left[ W_{\mcG}^{-1} \right] \left( X , X_+^{-1} \right) \ar[d, "_{W_{\mcG}^{-1}} t_{+ \top}^{-1}"] \\
		\dots & \mcG \left[ W_{\mcG}^{-1} \right] \left( X , X_{\bot}^{-1} \right) & \mcG \left[ W_{\mcG}^{-1} \right] \left( X , X_{\top}^{-1} \right) & 
	\end{tikzcd}
	\right] \]
	and 
	\[K'' = \left[ 
	\begin{tikzcd}	
		\mcG \left[ W_{\mcG}^{-1} \right] \left( X , X_-^0 \right) \ar[rd, "_{W_{\mcG}^{-1}} t_{- \bot}^0" below=2] & \mcG \left[ W_{\mcG}^{-1} \right] \left( X , X_I^0 \right) \ar[d, "_{W_{\mcG}^{-1}} t_{I \bot}^0" left=1] \ar[rd, "_{W_{\mcG}^{-1}} t_{I \top}^0"] & \dots \\
		& \mcG \left[ W_{\mcG}^{-1} \right] \left( X , X_{\bot}^0 \right) & \dots
	\end{tikzcd}
	\right] . \]
	Observe that the maps defining the chain complexes structures in $K'$ and $K''$ are all quasi-isomorphisms (see \cite[Lemma 3.12]{GPS20}). Thus it is not difficult to show using an increasing exhaustive and bounded from below filtration of $K'$ and $K''$ that these complexes are acyclic (compare the proof of Lemma \ref{lemma third result for specific module}). This implies that the map $_{W_{\mcG}^{-1}} t_{\mcG} : \, _{W_{\mcG}^{-1}} \mcG \left( X , X_{\bullet}^0 \right) \to \, _{W_{\mcG}^{-1}} \mcM_{\mcG} \left( X \right)$ is a quasi-isomorphism. 
	
\end{proof}

\subsection{Proof of the first result}\label{subsection proof of the first result}

Let $\tau$ be a quasi-autoequivalence of an $A_{\infty}$-category $\mcA$ equipped with a compatible $\mathbf{Z}$-splitting of $\mathrm{ob} \left( \mcA \right)$.
Assume that $\tau$ is strict and acts bijectively on hom-sets.

Observe that there is a strict $A_{\infty}$-functor $\sigma : \mcA \to \mcA_{\tau}$ which sends $X^n (E)$ to $E$, and which sends $x \in \mcA \left( X^i (E_1) , X^j (E_2) \right)$ to $[x] \in \mcA_{\tau} \left( E_1, E_2 \right)$. 
Besides, let $\pi : \mcC \to \mcA$ be the $A_{\infty}$-functor induced by the following commutative diagram 
\[\begin{tikzcd}
\mcA \ar[r, "\mathrm{id}"] \ar[d, "\mathrm{id}"] & \mcA \ar[d, "\mathrm{id}"] \\
\mcA \ar[r, "\mathrm{id}"] & \mcA
\end{tikzcd} \]
(see Proposition \ref{prop induced functor from Grothendieck construction}).
Then the diagram of Adams-graded $A_{\infty}$-categories
\[\begin{tikzcd}
\mcA_- \sqcup \mcA_+ \ar[r, "\mathrm{id} \sqcup \tau"] \ar[d, "\iota_{\bot} \sqcup \iota_{\top}" left] & \mcA_{\bullet} \ar[d, "\sigma"] \\
\mcC \ar[r, "\sigma \circ \pi"] & \mcA_{\tau}
\end{tikzcd} \]
is commutative because $\sigma \circ \tau = \sigma$. Moreover, the induced $A_{\infty}$-functor $\Phi : \mcG \to \mcA_{\tau}$ is strict, and it sends $W_{\mcG}$ to the set of units in $\mcA_{\tau}$.
Let 
\[ \widetilde{\Phi} : \mcH \to \mcA_{\tau} \left[ \{ \mathrm{units} \}^{-1} \right] \]
be the $A_{\infty}$-functor induced by $\Phi$.

According to Lemma \ref{lemma relation G - mapping torus}, $\mcH$ is quasi-equivalent to the mapping torus of $\tau$. Moreover, $\mcA_{\tau} \left[ \{ \mathrm{units} \}^{-1} \right]$ is quasi-equivalent to $\mcA_{\tau}$. Thus, Theorem \ref{thm mapping torus in strict situation} will follow if we prove that $\widetilde{\Phi}$ is a quasi-equivalence. 
Our strategy is to apply Proposition \ref{prop quasi-isomorphism between localizations}. 
Observe that it suffices to prove that 
\[ \widetilde{\Phi} : \mcH (X, Y) \to \mcA_{\tau} \left[ \{ \mathrm{units} \}^{-1} \right] (\Phi X, \Phi Y) \]
is a quasi-isomorphism for every object $X$, $Y$ in $\mcA_{\bullet}^0$ (recall that $\mcA_{\bullet}^0$ denotes the subcategory of $\mcA_{\bullet}$ generated by the objects $X_{\bullet}^0 (E)$, $E \in \mcE$) because every object of $\mcG$ can be related to one of $\mcA_{\bullet}^0$ by a zigzag of morphisms in $W_{\mcG}$, which are quasi-isomorphisms in $\mcH$ (see \cite[Lemma 3.12]{GPS20}).  

In the following, we fix some element $E_{\diamond} \in \mcE$. When we write an object $X_{\triangle}^n$ without specifying the element of $\mcE$, we mean $X_{\triangle}^n (E_{\diamond})$.
We consider the corresponding $\mcG$-module $\mcM_{\mcG}$ and the $\mcG$-module map $t_{\mcG} : \mcG \left( -, X_{\bullet}^0 \right) \to \mcG$ of Definition \ref{definition module for G}. 
Moreover, we set 
\[\mcM_{\mcA_{\tau}} := \mcA_{\tau} \left( -, E_{\diamond} \right) \text{ and } t_{\mcA_{\tau}} := \mathrm{id} : \mcA_{\tau} \left( -, E_{\diamond} \right) \to \mcM_{\mcA_{\tau}}. \]

\begin{lemma}\label{lemma G-module map for relation G - coinvariants}
	
	There exists a $\mcG$-module map $t_0 : \mcM_{\mcG} \to \Phi^* \mcM_{\mcA_{\tau}}$ (see Definition \ref{definition pullback functor} for the pullback functor) such that
	\begin{enumerate}
		
		\item the following diagram of $\mcG$-modules commutes
		\[\begin{tikzcd}
			\mcG \left( - , X_{\bullet}^0 \right) \ar[d, "t_{\mcG}"] \ar[r, "t_{\Phi}"] & \Phi^* \mcA_{\tau} \left( - , E_{\diamond} \right) \ar[d, "\Phi^* t_{\mcA_{\tau}} = \mathrm{id}"] \\
			\mcM_{\mcG} \ar[r, "t_0"] & \Phi^* \mcM_{\mcA_{\tau}}
		\end{tikzcd} \]
		(see Definition \ref{definition modules morphism induced by functor} for the map $t_{\Phi}$),
		
		\item for every $E \in \mcE$, the map $t_0 : \mcM_{\mcG} (X^0_{\bullet} (E)) \to \Phi^* \mcM_{\mcA_{\tau}} (X^0_{\bullet} (E))$ is a quasi-isomorphism.
		
	\end{enumerate}
	
\end{lemma}

\begin{proof}
	
	Observe that the diagram of $\mcG$-modules 
	\[\begin{tikzcd}
	\mcG \left( - , X_I^n \right) \ar[r, "t_{I \top}^n"] \ar[d, "t_{I \bot}^n"] & \mcG \left( - , X_{\top}^n \right) \ar[d, "t_{\Phi}"] \\
	\mcG \left( - , X_{\bot}^n \right) \ar[r, "t_{\Phi}"] & \Phi^* \mcM_{\mcA_{\tau}}
	\end{tikzcd} \]
	is commutative, so that it induces a $\mcG$-module map $\mcM_{\star}^n \to \Phi^* \mcM_{\mcA_{\tau}}$ (see Proposition \ref{prop morphism induced by homotopy}). 
	Now observe that the following diagram of $\mcG$-modules commutes
	\[\begin{tikzcd}[column sep = huge]
	\bigoplus \limits_{n \in \mathbf{Z}} \left( \mcG \left( - , X_-^n \right) \oplus \mcG \left( - , X_+^n \right) \right) \ar[r, "\bigoplus \limits_{n \in \mathbf{Z}} \left( t_{- \bullet}^n \oplus t_{+ \bullet}^n \right)"] \ar[d, "\bigoplus \limits_{n \in \mathbf{Z}} \left( t_{- \bot}^n \oplus t_{+ \top}^n \right)"] & \bigoplus \limits_{n \in \mathbf{Z}} \mcG \left( - , X_{\bullet}^n \right) \ar[d, "t_{\Phi}"] \\
	\bigoplus \limits_{n \in \mathbf{Z}} \mcM_{\star}^n \ar[r] & \Phi^* \mcM_{\mcA_{\tau}}
	\end{tikzcd} . \]
	We let $t_0 : \mcM_{\mcG} \to \Phi^* \mcM_{\mcA_{\tau}}$ be the induced $\mcG$-module map. It is then easy to verify that the following diagram of $\mcG$-modules is commutative
	\[\begin{tikzcd}
	\mcG \left( - , X_{\bullet}^0 \right) \ar[d, "t_{\mcG}"] \ar[r, "t_{\Phi}"] & \Phi^* \mcA_{\tau} \left( - , E_{\diamond} \right) \ar[d, "\Phi^* t_{\mcA_{\tau}} = \mathrm{id}"] \\
	\mcM_{\mcG} \ar[r, "t_0"] & \Phi^* \mcM_{\mcA_{\tau}} .
	\end{tikzcd} \]
	
	We now prove the second part of the lemma.
	We have 
	\[\mcM_{\mcG} \left( X_{\bullet}^0 (E) \right) = \bigoplus\limits_n \mcG \left( X_{\bullet}^0 (E), X_{\bullet}^n \right) = \bigoplus\limits_n \mcA \left( X^k (E) , X^n \right) \]	
	and 
	\[t_0 : \bigoplus\limits_n \mcA \left( X^k (E) , X^n \right) = \mcM_{\mcG} \left( X_{\bullet}^0 (E) \right) \to \Phi^* \mcM_{\mcA_{\tau}} \left( X_{\bullet}^0 (E) \right) = \mcA_{\tau} \left( E, E_{\diamond} \right) \]
	is the sum of the projections, which is an isomorphism. This concludes the proof
	
\end{proof}

\begin{lemma}\label{lemma relation G - coinvariants}
	
	For every $E \in \mcE$, the chain map 
	\[ \widetilde{\Phi} : \mcH \left( X_{\bullet}^0 (E), X_{\bullet}^0 \right) \to \mcA_{\tau} \left[ \{ \mathrm{units} \}^{-1} \right] \left( E, E_{\diamond} \right) \]
	is a quasi-isomorphism.
	
\end{lemma}

\begin{proof}
	
	According to Lemmas \ref{lemma properties of modules}, \ref{lemma properties of modules bis} and \ref{lemma G-module map for relation G - coinvariants}, the assumptions of Proposition \ref{prop quasi-isomorphism between localizations} are satisfied. This concludes the proof. 
	
\end{proof}

As explained above, Theorem \ref{thm mapping torus in strict situation} follows from Lemma \ref{lemma relation G - coinvariants} since $\mcH$ is quasi-equivalent to the mapping torus of $\tau$ (see Lemma \ref{lemma relation G - mapping torus}) and $\mcA_{\tau} \left[ \{ \mathrm{units} \}^{-1} \right]$ is quasi-equivalent to $\mcA_{\tau}$.

\subsection{Proof of the second result}\label{subsection proof of the second result}

Let $\tau$ be a quasi-autoequivalence of an $A_{\infty}$-category $\mcA$ equipped with a compatible $\mathbf{Z}$-splitting of $\mathrm{ob} \left( \mcA \right)$.
Assume that the following holds: 
\begin{enumerate}
	
	\item $\mcA$ is weakly directed with respect to the $\mathbf{Z}$-splitting of $\mathrm{ob} \left( \mcA \right)$ (see Definition \ref{definition group-action}),
	
	\item there exists a closed degree $0$ bimodule map $f : \mcA_m \left( -, - \right) \to \mcA_m \left( -, \tau(-) \right)$ (see Definitions \ref{definition bimodule} and \ref{definition pullback bimodule}) such that $f : \mcA_m \left( X^i(E) , X^j(E') \right) \to \mcA_m \left( X^i(E), X^{j+1} (E') \right)$ is a quasi-somorphism for every $i < j$ and $E, E' \in \mcE$.
	
\end{enumerate}

\begin{rmk}\label{rmk multiplication by f(unit) is quasi-iso}
	
	It follows from Lemma \ref{coro closed module map homotopic to Yoneda module map} and $\tau$ being a quasi-equivalence that the chain maps
	\[\left\{
	\begin{array}{ccccl}
		\mu_{\mcA_m}^2 \left( - , f(e_{X^j(E)}) \right) & : & \mcA_m \left( X^i(E') , X^j(E) \right) & \to & \mcA_m \left( X^i(E') , X^{j+1} (E) \right) \\
		\mu_{\mcA_m}^2 \left( f(e_{X^j (E)}) , - \right) & : & \mcA_m \left( X^{j+1} (E) , X^{k+1} (E') \right) & \to & \mcA_m \left( X^j (E) , X^{k+1} (E') \right)
	\end{array}
	\right. \]
	are quasi-isomorphisms for every $i < j < k$ and $E, E' \in \mcE$. 
	
\end{rmk}

In the following, we set 
\[c_n (E) := f \left( e_{X^n (E)} \right) \]
for every $n \in \mathbf{Z}$, $E \in \mcE$, and
\[W_{\mcA_m} := \left\{ c_n (E) \mid n \in \mathbf{Z}, E \in \mcE \right\} \cup \left\{ \text{units of } \mcA_m \right\} . \]

\subsubsection{Generalized homotopy}\label{subsection generalized homotopy}

Recall that we introduced a functor $\mcB \mapsto \mcB_m$ from the category of Adams-graded $A_{\infty}$-categories to the category of (non Adams-graded) $A_{\infty}$-categories. 
Also, recall that we introduced Adams-graded $A_{\infty}$-categories $\mcC$ and $\mcG$ in Definitions \ref{definition cylinder object} and \ref{definition category G} respectively.
Observe that $\mcC_m$ and $\mcG_m$ are the Grothendieck constructions of the diagrams
\[\begin{tikzcd}
(\mcA_I)_m \ar[r, "\mathrm{id}"] \ar[d, "\mathrm{id}" left] & (\mcA_{\top})_m \\
(\mcA_{\bot})_m &
\end{tikzcd} \text{ and }
\begin{tikzcd}
(\mcA_-)_m \sqcup (\mcA_+)_m \ar[r, "\mathrm{id} \sqcup \tau"] \ar[d, "\iota_{\bot} \sqcup \iota_{\top}" left] & (\mcA_{\bullet})_m \\
\mcC_m
\end{tikzcd}
\]
respectively. 
We denote by $W_{\mcC_m}$ the set of adjacent units in $\mcC_m$, and by $W_{\mcG_m}$ the union of $W_{\mcC_m}$ and the set of adjacent units in $\mcG_m$. 

We would like to define an $A_{\infty}$-functor $\Psi_m : \mcG_m \to \mcA_m$ which sends $W_{\mcG_m}$ to $W_{\mcA_m}$. According to Proposition \ref{prop induced functor from Grothendieck construction}, it is enough to prove the following result.

\begin{lemma}\label{lemma relation tau continuation}
	
	There exists an $A_{\infty}$-functor $\eta : \mcC_m \to \mcA_m$ which sends $W_{\mcC_m}$ to $W_{\mcA_m}$, and such that
	\[\eta \circ \iota_I = \eta \circ \iota_{\bot} = \mathrm{id}, \quad \eta \circ \iota_{\top} = \tau . \]
	
\end{lemma}

\begin{proof}
	
	We first define $\eta$ to be $\mathrm{id}$ on $(\mcA_{\bot})_m$, $(\mcA_I)_m$. and to be $\tau$ on $(\mcA_{\top})_m$. Observe that this completely defines $\eta$ on the objects.
	Besides, we ask for $\eta$ to act as the identity on the sequences involving an adjacent morphism from $(\mcA_I)_m$ to $(\mcA_{\bot})_m$.
	
	It remains to define $\eta$ on the sequences involving an adjacent morphism from $(\mcA_I)_m$ to $(\mcA_{\top})_m$.
	Consider a sequence of morphisms
	\begin{align*}
	\left( x_0, \dots, x_{p+q} \right) & \in \mcC_m \left( X_I^{i_0} (E_0), X_I^{i_1} (E_1) \right) \times \dots \times \mcC_m \left( X_I^{i_{p-1}} (E_{p-1}), X_I^{i_p} (E_p) \right) \\ 
	& \times \mcC_m \left( X_I^{i_p} (E_p), X_{\top}^{i_{p+1}} (E_{p+1}) \right) \\
	& \times \mcC_m \left( X_{\top}^{i_{p+1}} (E_{p+1}), X_{\top}^{i_{p+2}} (E_{p+2}) \right) \times \dots \times \mcC_m \left( X_{\top}^{i_{p+q}} (E_{p+q}), X_{\top}^{i_{p+q+1}} (E_{p+q+1}) \right) .
	\end{align*} 
	Observe that 
	\[\mcC_m \left( X_I^i (E), X_I^j (E') \right) = \mcC_m \left( X_I^i (E), X_ {\top}^j (E') \right) = \mcC_m \left( X_{\top}^i (E), X_{\top}^j (E') \right) = \mcA_m \left( X^i (E), X^j (E) \right) . \]
	Then we set 
	\[\eta \left( x_0, \dots, x_{p+q} \right) := f \left( x_0, \dots, x_{p-1}, x_p, x_{p+1}, \dots, x_{p+q+1} \right) \in \mcA_m \left( X^{i_0} (E_0), \tau X^{i_{p+q+1}} (E_{p+q+1}) \right) . \]
	The functor $\eta$ we defined satisfies the $A_{\infty}$-relations because $f : \mcA_m \left( -, - \right) \to \mcA_m \left( -, \tau (-) \right)$ is a closed degree $0$ bimodule map. Moreover, $\eta$ sends $W_{\mcC_m}$ to $W_{\mcA_m}$ by construction. 
	
\end{proof}

\begin{rmk}
	
	First observe that
	\[\mcC yl_{\mcA_m} = \mcC_m \left[ W_{\mcC_m}^{-1} \right] = \left( \mcC \left[ W_{\mcC}^{-1} \right] \right)_m = \left( \mcC yl_{\mcA} \right)_m . \]
	According to Lemma \ref{lemma relation tau continuation}, the functor $\eta$ induces an $A_{\infty}$-functor $\widetilde{\eta} : \mcC yl_{\mcA_m} \to \mcA_m \left[ W_{\mcA_m}^{-1} \right]$. Moreover, Lemma \ref{lemma relation tau continuation} implies that the following diagram commutes
	\[\begin{tikzcd}
	(\mcA_+)_m \ar[d, "\lambda_{\mcC_m} \circ \iota_{\top}" left] \ar[rd, bend left, "\lambda_{\mcA_m} \circ \tau"] & \\
	\mcC yl_{\mcA_m} \ar[r, "\widetilde{\eta}"] & \mcA_m \left[ W_{\mcA_m}^{-1} \right] \\
	(\mcA_-)_m \ar[u, "\lambda_{\mcC_m} \circ \iota_{\bot}"] \ar[ru, bend right, "\lambda_{\mcA_m}"] & 
	\end{tikzcd}
	\] 
	($\lambda_{\mcA_m} : \mcA_m \to \mcA_m \left[ W_{\mcA_m}^{-1} \right]$ and $\lambda_{\mcC_m} : \mcC_m \to \mcC_m \left[ W_{\mcC_m}^{-1} \right]$ denote the localization functors). Since $\mcC yl_{\mcA_m}$ should be thought as a cylinder object for $\mcA_m$ (see Proposition \ref{prop cylinder object}), we should think that the functors $\lambda_{\mcA_m}$ and $\lambda_{\mcA_m} \circ \tau$ are homotopic (even if they do not act the same way on objects) and that $\widetilde{\eta}$ is a generalized homotopy between them (see Proposition \ref{prop homotopy is generalized homotopy} for a justification of this terminology).
	
\end{rmk}

\subsubsection{Relation between $\mcG$ and $\mcA_m$}

Using the $A_{\infty}$-functor $\eta : \mcC_m \to \mcA_m$ of Lemma \ref{lemma relation tau continuation}, we get a commutative diagram of (non Adams-graded) $A_{\infty}$-categories 
\[\begin{tikzcd}
(\mcA_-)_m \sqcup (\mcA_+)_m \ar[r, "\mathrm{id} \sqcup \tau"] \ar[d, "\iota_{\bot} \sqcup \iota_{\top}" left] & (\mcA_{\bullet})_m \ar[d, "\mathrm{id}"] \\
\mcC_m \ar[r, "\eta"] & \mcA_m
\end{tikzcd} \]
and the induced $A_{\infty}$-functor $\Psi_m : \mcG_m \to \mcA_m$ (see Proposition \ref{prop induced functor from Grothendieck construction}) sends $W_{\mcG_m}$ to $W_{\mcA_m}$ (see Lemma \ref{lemma relation tau continuation}). Let 
\[ \widetilde{\Psi_m} : \mcH_m = \mcG_m \left[ W_{\mcG_m}^{-1} \right] \to \mcA_m \left[ W_{\mcA_m}^{-1} \right] \]
be the $A_{\infty}$-functor induced by $\Psi_m$. 
Observe that, since $\mcA$ is assumed to be weakly directed and since the Adams-degree of $\mcA$ comes from the $\mathbf{Z}$-splitting of $\mathrm{ob} \left( \mcA \right)$ (see Definition \ref{definition group-action}), $\mcH$ is concentrated in non-negative Adams-degree.
In particular, we can apply the adjunction of Definition \ref{definition adjunction} to $\widetilde{\Psi_m}$, which gives an $A_{\infty}$-functor
\[ \widetilde{\Psi} : \mcH = \mcG \left[ W_{\mcG}^{-1} \right] \to \mathbf{F} \left[ t_m \right] \otimes \mcA_m \left[ W_{\mcA_m}^{-1} \right] . \]
We would like to prove that for every Adams degree $j \geq 1$, and for every objects $X,Y$ in $\mcA_{\bullet}^0$ (recall that $\mcA_{\bullet}^0$ denotes the subcategory of $\mcA_{\bullet}$ generated by the objects $X_{\bullet}^0 (E)$, $E \in \mcE$), the map
\[\widetilde{\Psi} : \mcH \left( X,Y \right)^{*,j} \to \left( \mathbf{F} \left[ t_m \right] \otimes \mcA_m \left[ W_{\mcA_m}^{-1} \right] \right) \left( \Psi X , \Psi Y \right)^{*,j} = \mathbf{F}_m^j \otimes \mcA_m \left[ W_{\mcA_m}^{-1} \right] \left( \Psi X , \Psi Y \right) \]
is a quasi-isomorphism, ($\mathbf{F}_m^j$ is the vector space generated by $t_m^j$ ; also recall that if $V$ is an Adams-graded vector space, $V^{*,j}$ denotes the subspace of Adams degree $j$ elements). Our strategy is once again to apply Proposition \ref{prop quasi-isomorphism between localizations}.

In the following we fix some element $E_{\diamond} \in \mcE$. When we write $X_{\triangle}^n$ or $c_n$ without specifying the element of $\mcE$, we mean $X_{\triangle}^n (E_{\diamond})$ or $c_n (E_{\diamond})$ respectively.
We consider the corresponding $\mcG$-module $\mcM_{\mcG}$, and the $\mcG$-module map $t_{\mcG} : \mcG \left( -, X_{\triangle}^n \right) \to \mcG$ of Definition \ref{definition module for G}. 
Moreover, we consider the $\mcA_m$-module $\mcM_{\mcA_m}$ and the $\mcA_m$-module maps
\[t_{\mcA_m}^n : \mcA_m \left( - , X^n \right) \to \mcM_{\mcA_m}, \quad n \in \mathbf{Z}, \]
associated to the morphisms $\left( c_n \right)_{n \in \mathbf{Z}}$ as in Definition \ref{definition specific module}.

The following result is a first step in order to define a $\mcG_m$-module map $(t_0)_m : (\mcM_{\mcG})_m \to \Psi_m^* \mcM_{\mcA_m}$ as in Proposition \ref{prop quasi-isomorphism between localizations}.

\begin{lemma}\label{lemma first step to define the module map}
	
	For every $n \in \mathbf{Z}$, the diagram of $\mcG_m$-modules 
	\[\begin{tikzcd}[column sep = huge]
	\mcG_m \left( - , X_I^n \right) \ar[r, "t_{I \top}^n"] \ar[d, "t_{I \bot}^n"] & \mcG_m \left( - , X_{\top}^n \right) \ar[d, "\Psi_m^* t_{\mcA_m}^{n+1} \circ t_{\Psi_m}"] \\
	\mcG_m \left( - , X_{\bot}^n \right) \ar[r, "\Psi_m^* t_{\mcA_m}^n \circ t_{\Psi_m}"] & \Psi_m^* \mcM_{\mcA_m}
	\end{tikzcd} \]
	commutes up to homotopy.
	
\end{lemma}

\begin{proof}
	
	First observe that $\iota_I^* \mcG_m \left( - , X_I^n \right) = \mcA_m \left( -, X^n \right)$, and $\iota_I^* \Psi_m^* \mcM_{\mcA_m} = \mcM_{\mcA_m}$	because $\Psi_m \circ \iota_I = \mathrm{id}$ (see Remark \ref{rmk composition of pullback functors}).
	Moreover, it suffices to show that the $\mcA_m$-module maps 
	\[\iota_I^* \left( \Psi_m^* t_{\mcA_m}^n \circ t_{\Psi_m} \circ t_{I \bot}^n \right) = t_{\mcA_m}^n \circ \iota_I^* \left( t_{\Psi_m} \circ t_{I \bot}^n \right) : \mcA_m \left( -, X^n \right) \to \mcM_{\mcA_m} \]
	and 
	\[\iota_I^* \left( \Psi_m^* t_{\mcA_m}^{n+1} \circ t_{\Psi_m} \circ t_{I \top}^n \right) = t_{\mcA_m}^{n+1} \circ \iota_I^* \left( t_{\Psi_m} \circ t_{I \top}^n \right) : \mcA_m \left( -, X^n \right) \to \mcM_{\mcA_m} \]
	are homotopic because 
	\[\mcG_m \left( X_{\triangle}^k, X_I^n \right) = 0 \text{ if } \triangle \ne I . \]
	On the one hand, 
	\[t_{\mcA_m}^n \circ \iota_I^* \left( t_{\Psi_m} \circ t_{I \bot}^n \right) = t_{\mcA_m}^n .\]
	On the other hand, $\iota_I^* \left( t_{\Psi_m} \circ t_{I \top}^n \right) : \mcA_m \left( -, X^n \right) \to \mcA_m \left( -, X^{n+1} \right)$ is closed (as composition and pullback of closed module maps), and 
	\[\iota_I^* \left( t_{\Psi_m} \circ t_{I \top}^n \right) \left( e_{X^n} \right) = \eta \left( e_{X^n} \right) = c_n \]
	according to Lemma \ref{lemma relation tau continuation}. 
	Therefore, $\iota_I^* \left( t_{\Psi_m} \circ t_{I \top}^n \right)$ is homotopic to $t_{c_n}$ according to Corollary \ref{coro closed module map homotopic to Yoneda module map}, and thus $t_{\mcA_m}^{n+1} \circ \iota_I^* \left( t_{\Psi_m} \circ t_{I \top}^n \right)$ is homotopic to $t_{\mcA_m}^{n+1} \circ t_{c_n}$. Now according to Lemma \ref{lemma multiplication by continuation element becomes homotopic to inclusion}, $t_{\mcA_m}^{n+1} \circ t_{c_n}$ is homotopic to $t_{\mcA_m}^n$. This concludes the proof.
	
\end{proof}

We can now state the result establishing the existence of a $\mcG_m$-module map $(t_0)_m : (\mcM_{\mcG})_m \to \Psi_m^* \mcM_{\mcA_m}$ as in Proposition \ref{prop quasi-isomorphism between localizations}.

\begin{lemma}\label{lemma G-module map for relation G-A}
	
	There exists a $\mcG_m$-module map $(t_0)_m : (\mcM_{\mcG})_m \to \Psi_m^* \mcM_{\mcA_m}$ such that the following holds. 
	\begin{enumerate}
		
		\item The following diagram of $\mcG_m$-modules commutes
		\[\begin{tikzcd}
			\mcG_m \left( - , X_{\bullet}^0 \right) \ar[d, "t_{\mcG_m}"] \ar[r, "t_{\Psi_m}"] & \Psi_m^* \mcA_m \left( - , X^0 \right) \ar[d, "\Psi_m^* t_{\mcA_m}^0"] \\
			(\mcM_{\mcG})_m \ar[r, "(t_0)_m"] & \Psi_m^* \mcM_{\mcA_m}
		\end{tikzcd} \]
	
		\item for every $E \in \mcE$ and $j \geq 1$, the induced map $t_0 : \mcM_{\mcG} (X_{\bullet}^0 (E)) \to \mathbf{F} \left[ t_m \right] \otimes \Psi_m^* \mcM_{\mcA_m} (X_{\bullet}^0 (E))$ (see Definition \ref{definition adjunction}) is a quasi-isomorphism in each positive Adams-degree.
		
	\end{enumerate}
	
\end{lemma}

\begin{proof}
	
	Using Lemma \ref{lemma first step to define the module map} and Proposition \ref{prop morphism induced by homotopy}, we get a $\mcG_m$-module map $t_{\star}^n : \left( \mcM_{\star}^n \right)_m \to \Psi_m^* \mcM_{\mcA_m}$ for every $n \in \mathbf{Z}$ (see Definition \ref{definition module for G} for the $\mcG$-modules $\mcM_{\star}^n$). 
	Observe that the diagram of $\mcG_m$-modules 
	\[\begin{tikzcd}[column sep = huge]
	\bigoplus \limits_{n \in \mathbf{Z}} \left( \mcG_m \left( - , X_-^n \right) \oplus \mcG_m \left( - , X_+^n \right) \right) \ar[r, "\bigoplus \limits_{n \in \mathbf{Z}} \left( t_{- \bullet}^n \oplus t_{+ \bullet}^n \right)"] \ar[d, "\bigoplus \limits_{n \in \mathbf{Z}} \left( t_{- \bot}^n \oplus t_{+ \top}^n \right)"] & \bigoplus \limits_{n \in \mathbf{Z}} \mcG_m \left( - , X_{\bullet}^n \right) \ar[d, "\bigoplus \limits_{n \in \mathbf{Z}} \Psi_m^* t_{\mcA_m}^n \circ t_{\Psi_m}"] \\
	\bigoplus \limits_{n \in \mathbf{Z}} \left( \mcM_{\star}^n \right)_m \ar[r, "\bigoplus \limits_{n \in \mathbf{Z}} t_{\star}^n"] & \Psi_m^* \mcM_{\mcA_m}
	\end{tikzcd} \]
	is commutative (the composition is $\mathrm{id}$ for $-$-terms and $\tau$ for $+$-terms), so that it induces a $\mcG_m$-module map $(t_0)_m : (\mcM_{\mcG})_m \to \Psi_m^* \mcM_{\mcA_m}$. It is then easy to verify that the following diagram of $\mcG_m$-modules is commutative
	\[\begin{tikzcd}
	\mcG_m \left( - , X_{\bullet}^0 \right) \ar[d, "t_{\mcG_m}"] \ar[r, "t_{\Psi_m}"] & \Psi_m^* \mcA_m \left( - , X^0 \right) \ar[d, "\Psi_m^* t_{\mcA_m}^0"] \\
	(\mcM_{\mcG})_m \ar[r, "(t_0)_m"] & \Psi_m^* \mcM_{\mcA_m} .
	\end{tikzcd} \]
	
	It remains to show that the map $t_0 : \mcM_{\mcG} \left( X_{\bullet}^0 (E) \right)^{*,j} \to \mathbf{F}_m^j \otimes \mcM_{\mcA_m} \left( X^0 (E) \right)$ is a quasi-isomorphism for every $E \in \mcE$ and $j \geq 1$.
	Note that 
	\[\mcM_{\mcG} \left( X_{\bullet}^0 (E) \right)^{*,j} = \mcG \left( X_{\bullet}^0 (E), X_{\bullet}^j \right) = \mcA \left( X^0 , X^j \right) \\  \]
	and the map
	\[\mcA \left( X^0 (E) , X^j \right) = \mcM_{\mcG} \left( X_{\bullet}^0 (E) \right)^{*,j} \xrightarrow{t_0} \mathbf{F}_m^j \otimes \mcM_{\mcA_m} \left( X^0 (E) \right) \]
	is the inclusion. 
	Now observe that the following diagram of chain complexes commutes 
	\[\begin{tikzcd}
	\mcA \left( X^0 (E), X^j \right) \ar[r, "t_0"] & \mathbf{F}_m^j \otimes \mcM_{\mcA_m} \left( X^0 (E) \right) \\
	\mathbf{F}_m^j \otimes \mcA_m \left( X^0 (E), X^j \right) \ar[u, equal] \ar[r, equal] & \mathbf{F}_m^j \otimes \mcA_m \left( X^0 (E), X^j \right) \ar[u, hook] .
	\end{tikzcd} \]
	The inclusion $\mcA_m \left( X^0 (E) , X^j \right) \hookrightarrow \mcM_{\mcA_m} \left( X^0 (E) \right)$ is a quasi-isomorphism according to Lemma \ref{lemma first result for specific module} (observe that it is important here that $j$ is \emph{strictly} greater than $0$). Therefore the map $t_0 : \mcA \left( X^0 (E), X^j \right) \to \mathbf{F}_m^j \otimes \mcM_{\mcA_m} \left( X^0 (E) \right)$ is a quasi-isomorphism, which is what we needed to prove. 
	
\end{proof}

\begin{lemma}\label{lemma relation G-A}
	
	For every $E \in \mcE$ and $j \geq 1$, the map
	\[\widetilde{\Psi} : \mcH \left( X_{\bullet}^0 (E), X_{\bullet}^0 \right)^{*,j} \to \left( \mathbf{F} \left[ t_m \right] \otimes \mcA_m \left[ W_{\mcA_m}^{-1} \right] \right) \left( X^0 (E) , X^0 \right)^{*,j} = \mathbf{F}_m^j \otimes \mcA_m \left[ W_{\mcA_m}^{-1} \right] \left( X^0 (E) , X^0 \right) \]
	is a quasi-isomorphism.
	
\end{lemma}

\begin{proof}
	
	Using the first part of Lemma \ref{lemma G-module map for relation G-A} and Proposition \ref{prop quasi-isomorphism between localizations}, we know that there exists a chain map $u_m : \, _{W_{\mcG_m}^{-1}} (\mcM_{\mcG})_m \left( X_{\bullet}^0 (E) \right) \to \, _{W_{\mcA_m}^{-1}} \mcM_{\mcA_m} \left( X^0 \right)$ such that the following diagram of chain complexes commutes
	\[\begin{tikzcd}
	\mcH_m \left( X_{\bullet}^0 (E) , X_{\bullet}^0 \right) \ar[d, "_{W_{\mcG_m}^{-1}} t_{\mcG_m}"] \ar[r, "\widetilde{\Psi}_m"] & \mcA_m \left[ W_{\mcA_m}^{-1} \right] \left( X^0 (E) , X^0 \right) \ar[d, "_{W_{\mcA_m}^{-1}} t_{\mcA_m}^0"] \\
	_{W_{\mcG_m}^{-1}} (\mcM_{\mcG})_m \left( X_{\bullet}^0 (E) \right) \ar[r, "u_m"] & _{W_{\mcA_m}^{-1}} \mcM_{\mcA_m} \left( \Psi_m X \right) \\
	(\mcM_{\mcG})_m \left( X_{\bullet}^0 (E) \right) \ar[u, hook] \ar[r, "(t_0)_m"] & \mcM_{\mcA_m} \left( \Psi_m X \right) \ar[u, hook] .
	\end{tikzcd}  \]
	Observe that 
	\[\left\{
	\begin{array}{lll}
	\mcH_m \left( X_{\bullet}^0 (E) , X_{\bullet}^0 \right) & = & \mcH \left( X_{\bullet}^0 (E) , X_{\bullet}^0 \right)_m \\
	_{W_{\mcG_m}^{-1}} (\mcM_{\mcG})_m \left( X_{\bullet}^0 (E) \right) & = & _{W_{\mcG}^{-1}} \mcM_{\mcG} \left( X_{\bullet}^0 (E) \right)_m \\
	(\mcM_{\mcG})_m \left( X_{\bullet}^0 (E) \right) & = & \mcM_{\mcG} \left( X_{\bullet}^0 (E) \right)_m .
	\end{array}
	\right. \]
	Applying the adjunction of Definition \ref{definition adjunction} to the last diagram, we get the following commutative diagram of Adams-graded chain complexes
	\[\begin{tikzcd}
	\mcH \left( X_{\bullet}^0 (E) , X_{\bullet}^0 \right) \ar[d, "_{W_{\mcG}^{-1}} t_{\mcG}"] \ar[r, "\widetilde{\Psi}"] & \mathbf{F} \left[ t_m \right] \otimes \mcA_m \left[ W_{\mcA_m}^{-1} \right] \left( X^0 (E) , X^0 \right) \ar[d, "\mathrm{id} \otimes _{W_{\mcA_m}^{-1}} t_{\mcA_m}^0"] \\
	_{W_{\mcG}^{-1}} \mcM_{\mcG} \left( X_{\bullet}^0 (E) \right) \ar[r, "u"] & \mathbf{F} \left[ t_m \right] \otimes _{W_{\mcA_m}^{-1}} \mcM_{\mcA_m} \left( \Psi_m X \right) \\
	\mcM_{\mcG} \left( X_{\bullet}^0 (E) \right) \ar[u, hook] \ar[r, "t_0"] & \mathbf{F} \left[ t_m \right] \otimes \mcM_{\mcA_m} \left( \Psi_m X \right) \ar[u, hook] .
	\end{tikzcd} \]
	Specializing to the components of fixed Adams degree $j \geq 1$, we get the following commutative diagram of chain complexes 
	\[\begin{tikzcd}
	\mcH \left( X_{\bullet}^0 (E) , X_{\bullet}^0 \right)^{*,j} \ar[d, "_{W_{\mcG}^{-1}} t_{\mcG}"] \ar[r, "\widetilde{\Psi}"] & \mathbf{F}_m^j \otimes \mcA_m \left[ W_{\mcA_m}^{-1} \right] \left( X^0 (E) , X^0 \right) \ar[d, "\mathrm{id} \otimes _{W_{\mcA_m}^{-1}} t_{\mcA_m}^0"] \\
	_{W_{\mcG}^{-1}} \mcM_{\mcG} \left( X_{\bullet}^0 (E) \right)^{*,j} \ar[r, "u"] & \mathbf{F}_m^j \otimes _{W_{\mcA_m}^{-1}} \mcM_{\mcA_m} \left( \Psi_m X \right) \\
	\mcM_{\mcG} \left( X_{\bullet}^0 (E) \right)^{*,j} \ar[u, hook] \ar[r, "t_0"] & \mathbf{F}_m^j \otimes \mcM_{\mcA_m} \left( \Psi_m X \right) \ar[u, hook] .
	\end{tikzcd} \]
	Using Lemmas \ref{lemma properties of modules}, \ref{lemma properties of modules bis}, and \cite[Lemma 3.13]{GPS20}, we know that all the vertical maps on the left are quasi-isomorphisms. 
	Similarly, using Lemmas \ref{lemma second result for specific module}, \ref{lemma third result for specific module}, and \cite[Lemma 3.13]{GPS20}, we know that all the vertical maps on the right are quasi-isomorphisms. 
	Moreover, the second part of Lemma \ref{lemma G-module map for relation G-A} states that the lowest horizontal map is a quasi-isomorphism. 
	Thus, the chain map $\widetilde{\Psi} : \mcH \left( X_{\bullet}^0 (E) , X_{\bullet}^0  \right)^{*,j} \to \mathbf{F}_m^j \otimes \mcA_m [ W_{\mcA_m}^{-1} ] \left( X^0 (E) , X^0 \right)$ is a quasi-isomorphism. 
	
\end{proof}

\subsubsection{End of the proof}

We end the section with the proof of Theorem \ref{thm mapping torus in weak situation}.
Now that we have proved Lemma \ref{lemma relation G-A} which takes care of the \emph{positive} Adams-degrees, we have to treat the zero Adams-degree part (recall that $\mcH$ is concentrated in non-negative Adams-degree because $\mcA$ is assumed to be weakly directed).
 
Let $\mcI$ be the (non full) $A_{\infty}$-subcategory of $\mcH$ with 
\[\mathrm{ob} \left( \mcI \right) = \left\{ X_{\bullet}^0 (E) \mid E \in \mcE \right\} \text{ and } \mcI \left( X, Y \right) = \mcG \left( X, Y \right) \oplus \left( \bigoplus \limits_{j \geq 1} \mcH \left( X, Y \right)^{*,j} \right) \]
(recall that if $V$ is an Adams-graded vector space, we denote by $V^{*,j}$ its component of Adams-degree $j$). 

\begin{lemma}
	
	The inclusion $\mcI \hookrightarrow \mcH$ is a quasi-equivalence.
	
\end{lemma}

\begin{proof}
	
	Observe that the inclusion $\mcI \hookrightarrow \mcH$ is cohomologically essentially surjective because every object of $\mcH$ can be related to one of $\mcI$ by a zigzag of morphisms in $W_{\mcG}$, which are quasi-isomorphisms in $\mcH$ (see \cite[Lemma 3.12]{GPS20}). 
	Therefore, it suffices to show that the inclusion 
	\[\mcG \left( X_{\bullet}^0 \left( E \right), X_{\bullet}^0 \left( E_{\diamond} \right) \right) \hookrightarrow \mcH \left( X_{\bullet}^0 \left( E \right), X_{\bullet}^0 \left( E_{\diamond} \right) \right)^{*,0} \]
	is a quasi-isomorphism for every $E, E_{\diamond} \in \mcE$.
	
	Let $E_{\diamond}$ be an element of $\mcE$. When we write an object $X_{\bullet}^n$ without specifying the element of $\mcE$, we mean $X_{\bullet}^n (E_{\diamond})$.
	Recall that we introduced a pair $\left( \mcM_{\mcG}, t_{\mcG} \right)$ in Definition \ref{definition module for G}.
	According to Lemmas \ref{lemma properties of modules}, \ref{lemma properties of modules bis} and \cite[Lemma 3.13]{GPS20}, the inclusion $\mcM_{\mcG} \left( X^0_{\bullet} (E) \right) \hookrightarrow \, _{W_{\mcG}^{-1}} \mcM_{\mcG} \left( X^0_{\bullet} (E) \right)$ and the map $_{W_{\mcG}^{-1}} t_{\mcG} : \, \mcH \left( X^0_{\bullet} (E) , X^0_{\bullet} \right) \to \, _{W_{\mcG}^{-1}} \mcM_{\mcG} \left( X^0_{\bullet} (E) \right)$ are quasi-isomorphisms for every $E \in \mcE$. Besides, observe that 
	\[\mcM_{\mcG} \left( X^0_{\bullet} (E) \right)^{*,0} = \mcG \left( X^0_{\bullet} (E),X^0_{\bullet} \right) . \]
	The result then follows from the commutativity of the following diagram
	\[\begin{tikzcd}
	\mcG \left( X_{\bullet}^0 \left( E \right), X_{\bullet}^0 \right) \ar[d, equal] \ar[r, equal] & \mcG \left( X_{\bullet}^0 \left( E \right), X_{\bullet}^0 \right) \ar[d, hook] \ar[r, equal] & \mcG \left( X_{\bullet}^0 \left( E \right), X_{\bullet}^0 \right) \ar[d, hook] \\
	\mcM_{\mcG} \left( X^0_{\bullet} (E) \right)^{*,0} \ar[r, hook, "\sim"] & _{W_{\mcG}^{-1}} \mcM_{\mcG} \left( X^0_{\bullet} (E) \right)^{*,0} & \mcH \left( X^0_{\bullet} (E) , X^0_{\bullet} \right)^{*,0} \ar[l, "\sim" above, "_{W_{\mcG}^{-1}} t_{\mcG}" below] .
	\end{tikzcd} \]  
	
\end{proof}

The following diagram of Adams-graded $A_{\infty}$-categories is commutative
\[ \begin{tikzcd}
\mcH \ar[r, "\widetilde{\Psi}"] & \mathbf{F} \left[ t_m \right] \otimes \mcA_m \left[ W_{\mcA_m}^{-1} \right] \\
\mcI \ar[u, hook, "\sim" right] \ar[r, "\widetilde{\Psi}"] & \mcA_m^0 \oplus \left( t \mathbf{F} \left[ t_m \right] \otimes \mcA_m \left[ W_{\mcA_m}^{-1} \right]^0 \right) \ar[u, hook]
\end{tikzcd} \]
(recall that if $\mcC$ is an $A_{\infty}$-category equipped with a splitting $\mathrm{ob} \left( \mcC \right) \simeq \mathbf{Z} \times \mcE$, then we denote by $\mcC^0$ the full $A_{\infty}$-subcategory of $\mcC$ whose set of objects corresponds to $\{ 0 \} \times \mcE$).
Moreover, since $\mcA$ is assumed to be weakly directed with respect to the $\mathbf{Z}$-splitting of $\mathrm{ob} \left( \mcA \right)$, Lemma \ref{lemma relation G-A} implies that the bottom horizontal $A_{\infty}$-functor is a quasi-equivalence. 
Therefore we have
\[\mcH \simeq \mcA_m^0 \oplus \left( t \mathbf{F} \left[ t_m \right] \otimes \mcA_m \left[ W_{\mcA_m}^{-1} \right]^0 \right) . \]
Recall that $W_{\mcA_m} = f \left( \mathrm{units} \right) \cup \{ \mathrm{units} \}$, so that 
\[\mcA_m \left[ W_{\mcA_m}^{-1} \right] \simeq \mcA_m \left[ f \left( \mathrm{units} \right)^{-1} \right] . \]
This concludes the proof of Theorem \ref{thm mapping torus in weak situation}, since $\mcH$ is quasi-equivalent to the mapping torus of $\tau$ (see Lemma \ref{lemma relation G - mapping torus}).

% !TeX spellcheck = en_US
\section{Chekanov-Eliashberg DG-category}\label{section Legendrian invariants}

In this section, we recall the definition of the Chekanov-Eliashberg DG-category associated to a family of Legendrians in a contact manifold equipped with a hypertight contact form $\alpha$, which means that the Reeb vector field of $\alpha$ has no contractible periodic orbits. . We also describe the behavior of the Chekanov-Eliashberg DG-category under change of data. 

In the following, $(V, \xi)$ is a contact manifold of dimension $(2n+1)$.
In order to have well-defined gradings in $\mathbf{Z}$, we assume that $H_1 \left( V \right)$ is free and that the first Chern class of $\xi$ (equipped with any compatible almost complex structure) is $2$-torsion.
We will need the following definition.

\begin{defin}\label{definition chord generic}
	
	We say that a Legendrian submanifold $\Lambda$ in $\left( V, \xi \right)$ is \emph{chord generic} with respect to a contact form $\alpha$ if the following holds:
	\begin{enumerate}
		
		\item for every Reeb chord $c : \left[ 0,T \right] \to V$ of $\Lambda$, the space $D \varphi_{R_{\alpha}}^T \left( T_{c(0)} \Lambda \right)$ is transverse to $T_{c(T)} \Lambda$ in $\xi$,
		
		\item different Reeb chords belong to different Reeb trajectories. 
		
	\end{enumerate}
	
\end{defin}

\subsection{Conley-Zehnder index}\label{subsection Conley-Zehnder index}

Let $\alpha$ be a hypertight contact form on $(V, \xi)$ and let $\Lambda$ be a chord generic Legendrian submanifold of $(V, \alpha)$.
In the following, we define the Conley-Zehnder index of a Reeb chord of $\Lambda$ starting and ending on the same connected component (such chords are called \emph{pure}).

We briefly recall what is the Maslov index of a loop in the Grassmanian of Lagrangian subspaces in $\mathbf{C}^n$. We refer to \cite{RS93} for a precise exposition. Fix a Lagrangian subspace $K$, and denote by $\Sigma_k \left( K \right)$ the set of Lagrangian subspaces in $\mathbf{C}^n$ whose intersection with $K$ is $k$ dimensional. Consider the \emph{Maslov cycle} 
\[\Sigma = \Sigma_1 \left( K \right) \cup \dots \cup \Sigma_n \left( K \right) . \]
This is an algebraic variety of codimension one in the Lagrangian Grassmanian. Now if $\Gamma$ is a loop in the Lagrangian Grassmanian, its Maslov index $\mu \left( \Gamma \right) \in \mathbf{Z}$ is the intersection number of $\Gamma$ with $\Sigma$. The contribution of an intersection instant $t_0$ is computed as follows. Choose a Lagrangian complement $W$ of $K$ in $\mathbf{C}^n$. Then for each $v$ in $\Gamma (t_0) \cap K$, there exists a vector $w(t)$ in $W$ such that $v + w(t)$ is in $\Gamma (t)$ for every $t$ near $t_0$. Consider the quadratic form
\[Q(v) = \frac{d}{dt} \omega \left( v, w(t) \right)_{|t=t_0} \]
on $\Gamma (t_0) \cap K$. Without loss of generality, $Q$ can be assumed non-singular and the contribution of $t_0$ to $\mu (\Gamma)$ is the signature of $Q$.

Recall that $H_1 \left( V \right)$ is assumed to be free. We choose a family $(h_1, \dots, h_r)$ of embedded circles in $V$ which represent a basis of $H_1 ( V )$, and a symplectic trivialization of $\xi$ over each $h_i$. If $\gamma$ is some loop in $\Lambda$, there is a unique family $(a_1, \dots, a_r)$ of integers such that $\left[ \gamma_c - \sum_i a_i h_i \right]$ is zero in $H_1 (V)$. Choose a surface $\Sigma_{\gamma}$ in $V$ such that 
\[ \partial \Sigma_{\gamma} = \gamma - \sum_i a_i h_i . \]
There is a unique trivialization of $\xi$ over $\Sigma_{\gamma}$ which extends the chosen trivializations over $h_i$.
Thus we get a trivialization $\gamma^{-1} \xi \simeq S^1 \times \mathbf{C}^n$ (where $n$ is the dimension of $\Lambda$). We denote by $\Gamma$ the loop of Lagrangian planes in $\mathbf{C}^n$ corresponding, via the latter trivialization, to the loop $t \mapsto T_{\gamma (t)} \Lambda$. The Maslov index of $\Gamma$ does not depend on the choice of the surface $\Sigma_{\gamma}$ because we assumed $2 c_1 (\xi) = 0$. This construction defines a morphism $H_1 \left( \Lambda, \mathbf{Z} \right) \to \mathbf{Z}$, and the \emph{Maslov number} $m(\Lambda)$ of $\Lambda$ is the generator of its image. 
In the following, we assume that the Maslov number of $\Lambda$ is zero.  

Now, let $c$ be a pure Reeb chord of $\Lambda$. We choose a path $\gamma_c : \left[ 0,1 \right] \to \Lambda$ which starts at the endpoint of $c$, and ends at its starting point ($\gamma_c$ is called a \emph{capping path} of $c$). We denote by $\overline{\gamma}_c$ the loop obtained by concatenating $\gamma$ and $c$. Let $(a_1, \dots, a_r)$ be the unique family of integers such that $\left[ \overline{\gamma}_c - \sum_i a_i h_i \right]$ is zero in $H_1 (V)$, and choose a surface $\Sigma_c$ in $V$ such that 
\[ \partial \Sigma_c = \overline{\gamma}_c - \sum_i a_i h_i . \]
There is a unique trivialization of $\xi$ over $\Sigma_c$ which extends the chosen trivializations over $h_i$.
Thus we get a trivialization $\overline{\gamma}_c^{-1} \xi \simeq S^1 \times \mathbf{C}^n$ (where $n$ is the dimension of $\Lambda$). We denote by $\Gamma_c$ the path of Lagrangian planes in $\mathbf{C}^n$ corresponding, via the latter trivialization, to the concatenation of $t \mapsto T_{\gamma (t)} \Lambda$ and $t \mapsto D \varphi_{R_{\alpha}}^t \left( T_{c(0)} \Lambda \right)$. Since $\Lambda$ is chord generic, $\Gamma_c$ is not a loop: we close it in the following way. Let $I$ be a complex structure on $\mathbf{C}^n$ which is compatible with the standard symplectic form on $\mathbf{C}^n$ and such that $I \left( \Gamma_c (1) \right) = \Gamma_c (0)$. Then we let $\overline{\Gamma}_c$ be the loop of Lagrangian subspaces obtained by concatenating $\Gamma_c$ and the path $t \mapsto e^{tI} \Gamma_c (1)$. The Conley-Zehnder index of $c$ is the Maslov index of $\overline{\Gamma}_c$: 
\[CZ (c) := \mu \left( \overline{\Gamma}_c \right). \]
The Conley-Zehnder index of a Reeb chord does not depend on the choice of $\Sigma_c$ because the first Chern class of $\xi$ is $2$-torsion, and it does not depend on the choice of $\gamma_c$ because the Maslov number of $\Lambda$ vanishes.  

\subsection{Moduli spaces}\label{subsection moduli spaces}

Recall that $(V, \xi)$ is a contact manifold such that $H_1 (V)$ is free and the first Chern class of $\xi$ (equipped with any compatible almost complex structure) is 2-torsion. 
Let $\alpha$ be a hypertight contact form on $(V, \xi)$ and let $\Lambda$ be a chord generic Legendrian submanifold of $(V, \alpha)$ with vanishing Maslov number.
In the following, we introduce the moduli spaces needed to define the Chekanov-Eliashberg category of $\Lambda$. 

\begin{defin}\label{definition Riemann disk}
	
	A Riemann $(d+1)$-pointed disk is a triple $\left( D, \boldsymbol{\zeta}, j \right)$ such that 
	\begin{enumerate}
		
		\item $D$ is a smooth oriented manifold-with-boundary diffeomorphic to the closed unit disk in $\mathbf{C}$,
		
		\item a family $\boldsymbol{\zeta} = \left( \zeta_d, \dots, \zeta_1, \zeta_0 \right)$ is a cyclically ordered family of distinct points on $\partial D$,
		
		\item $j$ is an integrable almost complex structure on $D$ which induces the given orientation on $D$.
		
	\end{enumerate}
	If $\left( D, \boldsymbol{\zeta}, j \right)$ is a Riemann pointed disk, we denote by $\Delta := D \setminus \{ \zeta_d, \dots, \zeta_1, \zeta_0 \}$ the corresponding punctured disk.
	
\end{defin}

\begin{defin}
	
	A family of Riemann $(d+1)$-pointed discs is a bundle $\mcS \to \mcR$ with 
	\begin{enumerate}
		
		\item a family $\boldsymbol{\zeta} = \left( \zeta_d, \dots, \zeta_1, \zeta_0 \right)$ of non-intersecting sections $\zeta_k : \mcR \to \mcS$ and
		
		\item a section $j : \mcR \to \mathrm{End} (T \mcS )$
		
	\end{enumerate}
	such that $(\mcS_r, \boldsymbol{\zeta} (r), j(r))$ is a Riemann $(d+1)$-pointed disk for every $r \in \mcR$.
	
\end{defin}

\begin{defin}\label{definition choice of strip-like ends}
	
	Let $\mcS \to \mcR$ be a family of Riemann $(d+1)$-pointed discs.
	A choice of strip-like ends for $\mcS \to \mcR$ is a family of sections
	\[\epsilon_d, \dots, \epsilon_1 : \mcR \times \mathbf{R}_{\leq 0} \times \left[ 0,1 \right] \to \Delta_r, \quad \epsilon_0 : \mcR \times \mathbf{R}_{\geq 0} \times \left[ 0,1 \right] \to \Delta_r \]
	such that
	\begin{enumerate}
		
		\item $\epsilon_d (r), \dots, \epsilon_1 (r), \epsilon_0 (r)$ are proper embeddings with
		\[\epsilon_k (r) \left( \mathbf{R}_{\leq 0} \times \{ 0, 1 \} \right) \subset \partial \Delta_r \text{ and } \epsilon_0 (r) \left( \mathbf{R}_{\geq 0} \times \{ 0, 1 \} \right) \subset \partial \Delta_r,  \]
		
		\item $\epsilon_d (r), \dots, \epsilon_1 (r), \epsilon_0 (r)$ satisfy the asymptotic conditions
		\[\epsilon_k (r) \left( s, t \right) \underset{s \to - \infty}{\longrightarrow} \zeta_k (r) \text{ and } \epsilon_0 (r) \left( s, t \right) \underset{s \to + \infty}{\longrightarrow} \zeta_0 (r) \]
		
		\item $\epsilon_d (r), \dots, \epsilon_1 (r), \epsilon_0 (r)$ are $\left( i, j(r) \right)$-holomorphic, where $i$ is the standard complex structure on $\mathbf{C}$.
		
	\end{enumerate}
	
\end{defin}

As explained in \cite[section (9c)]{Sei08}, there is a universal family $\mcS^{d+1} \to \mcR^{d+1}$ of Riemann $(d+1)$-pointed discs when $d \geq 2$, which means that any other family $\mcS \to \mcR$ is isomorphic to the pullback of $\mcS^{d+1} \to \mcR^{d+1}$ by a map $\mcR \to \mcR^{d+1}$.
In the following, we fix a choice of strip-like ends for the unversal family $\mcS^{d+1} \to \mcR^{d+1}$.

\begin{defin}\label{definition moduli spaces}
	
	Let $J$ be an almost complex structure on $\xi$ compatible with $\left( d \alpha \right)_{| \xi}$. We denote by $J^{\alpha}$ the unique almost complex structure on $\mathbf{R}_{\sigma} \times V$ which sends $\partial_{\sigma}$ to $R_{\alpha}$ and which restricts to $J$ on $\xi$.
	Let $c_d, \dots, c_1, c_0$ be a Reeb chords of $\Lambda$, where $c_k : [0, T_k] \to V$. 
	\begin{enumerate}
		
		\item If $d=1$, we denote by $\widetilde{\mcM}_{c_1, c_0} \left( \mathbf{R} \times \Lambda, J, \alpha \right)$ the set of equivalence classes of maps $u : \mathbf{R} \times \left[ 0, 1 \right] \to \mathbf{R} \times V$ such that
		\begin{itemize}                  
			\item $u$ maps the boundary of $\mathbf{R} \times \left[ 0, 1 \right]$ to $\mathbf{R} \times \Lambda$,
			\item $u$ satisfies the asymptotic conditions
			\[u \left( s, t \right) \underset{s \to - \infty}{\longrightarrow} \left( - \infty, c_1 (T_1 t) \right) \text{ and } u \left( s, t \right) \underset{s \to + \infty}{\longrightarrow} \left( + \infty, c_0 (T_0 t) \right), \]
			\item $u$ is $(i, J^{\alpha})$-holomorphic,
		\end{itemize}
		where two maps $u$ and $u'$ are identified if there exists $s_0 \in \mathbf{R}$ such that $u' (\cdot, \cdot) = u (\cdot + s_0, \cdot)$.
		
		\item If $d \geq 2$, we denote by $\widetilde{\mcM}_{c_d, \dots, c_1, c_0} \left( \mathbf{R} \times \Lambda, J, \alpha \right)$ the set of of pairs $\left( r, u \right)$ such that
		\begin{itemize}
			\item $r \in \mcR^{d+1}$ and $u : \Delta_r \to \mathbf{R} \times V$ maps the boundary of $\Delta_r$ to $\mathbf{R} \times \Lambda$,
			\item $u$ satisfies the asymptotic conditions
			\[\left( u \circ \epsilon_k (r) \right) \left( s, t \right) \underset{s \to - \infty}{\longrightarrow} \left( -\infty, c_k (T_k t) \right) \text{ and } \left( u \circ \epsilon_0 (r) \right) \left( s, t \right) \underset{s \to + \infty}{\longrightarrow} \left( + \infty, c_0 (T_0 t) \right), \]
			\item $u$ is $(i, J^{\alpha})$-holomorphic.
		\end{itemize}
		
	\end{enumerate}
	Observe that $\mathbf{R}$ acts on $\widetilde{\mcM}_{c_d, \dots, c_1, c_0} \left( \mathbf{R} \times \Lambda, J, \alpha \right)$ by translation in the $\mathbf{R}_{\sigma}$-coordinate. 
	We set
	\[\mcM_{c_d, \dots, c_1, c_0} \left( \mathbf{R} \times \Lambda, J, \alpha \right) := \widetilde{\mcM}_{c_d, \dots, c_1, c_0} \left( \mathbf{R} \times \Lambda, J, \alpha \right) / \mathbf{R} . \] 
	
\end{defin}

The moduli space $\widetilde{\mcM}_{c_d, \dots, c_1, c_0} \left( \mathbf{R} \times \Lambda, J, \alpha \right)$ can be realized as the zero-set of a section $\overline{\partial} : \mcB \to \mcE$ of a Banach bundle $\mcE \to \mcB$ (see for example \cite{EES07}). We say that $\widetilde{\mcM}_{c_d, \dots, c_1, c_0} \left( \mathbf{R} \times \Lambda, J \right)$ is transversely cut out if $\overline{\partial}$ is transverse to the 0-section. 

\begin{defin}\label{definition regular almost complex structure}
	
	We say that $J$ is \emph{regular} (with repect to $\alpha$ and $\Lambda$) if the moduli spaces $\widetilde{\mcM}_{c_d, \dots, c_1, c_0} \left( \mathbf{R} \times \Lambda, J, \alpha \right)$ are all transversely cut out.
	
\end{defin}

\begin{prop}[\cite{DR16bis} Proposition 3.13]\label{prop transverslity moduli spaces}

	The set of regular almost complex structures on $\xi$ is Baire.  
	Moreover, the dimension of a transversely cut out moduli space is 
	\[\mathrm{dim} \, \widetilde{\mcM}_{c_d, \dots, c_1, c_0} \left( \mathbf{R} \times \Lambda, J, \alpha \right) = CZ \left( a \right) - \left( \sum_{k=1}^{d} CZ \left( b_k \right) \right) + d-1 . \]

\end{prop}

\subsection{Chekanov-Eliashberg DG-category}\label{subsection Chekanov-Eliashberg algebra}

Recall that $(V, \xi)$ is a contact manifold such that $H_1 (V)$ is free and the first Chern class of $\xi$ (equipped with any compatible almost complex structure) is 2-torsion. 
Let $\alpha$ be a hypertight contact form on $(V, \xi)$ and let $\mathbf{\Lambda} = (\Lambda (E))_{E \in \mcE}$ be a family of Legendrian submanifolds of $(V, \xi)$. We set $\Lambda := \bigcup_{E \in \mcE} \Lambda(E)$ and we assume that $\Lambda$ is chord generic with vanishing Maslov number. 
Moreover, we denote by $\mcC (\Lambda(E), \Lambda(E'))$ the graded vector space generated by the words of Reeb chords $c_1 \cdots c_d$ where $c_1$ starts on $\Lambda(E)$, $c_d$ ends on $\Lambda(E')$, and the ending component of $c_i$ is the starting component of $c_{i+1}$ for every $1 \leq i \leq d-1$, with grading 
\[\vert c_1 \cdots c_d \vert := \sum_{i=1}^{d} \left( CZ \left( c_i \right) - 1 \right). \]
Finally, let $J$ be a regular almost complex structure on $\xi$.

\begin{defin}\label{definition Chekanov-Eliashberg}
	
	We denote by $CE_* \left( \mathbf{\Lambda} \right) = CE_* \left( \mathbf{\Lambda}, J, \alpha \right)$ the graded category defined as follows 
	\begin{enumerate}
		
		\item the objects are the Legendrians $\Lambda (E)$, $E \in \mcE$,
		
		\item the space of morphisms from $\Lambda (E)$ to $\Lambda (E')$ is either $\mcC (\Lambda(E), \Lambda(E'))$ if $E \ne E'$, or $\mathbf{F} \oplus \mcC (\Lambda(E), \Lambda(E'))$ if $E = E'$,
		
		\item the composition is given by concatenation of words.
		
	\end{enumerate}
	If $c_0$ is a Reeb chord in $CE_* \left( \mathbf{\Lambda} \right)$, we set 
	\[\partial (c_0) := \sum \limits_{c_d, \dots,c_1} \# \mcM_{c_d, \dots, c_1, c_0} \left( \mathbf{R} \times \Lambda, J, \alpha \right) c_d \cdots c_1, \]
	where $\# \mcM \in \mathbf{F}$ denotes the number of elements modulo $2$ in $\mcM$ if $\mcM$ is finite, and $0$ otherwise. Finally, we extend $\partial$ to $CE_* \left( \mathbf{\Lambda} \right)$ so that it is linear and satisfies the Leibniz rule with respect to the concatenation product.
	
\end{defin}

\begin{thm}\label{thm d rond d egal 0}
	
	$\partial : CE_* \left( \mathbf{\Lambda} \right) \to CE_* \left( \mathbf{\Lambda} \right)$ decreases the grading by $1$ and satisfies $\partial \circ \partial = 0$.
	As a result, $(CE_{-*} \left( \mathbf{\Lambda} \right), \partial)$ is a DG-category. 
	
\end{thm}

\begin{proof}
	
	This follows from Proposition  \ref{prop transverslity moduli spaces}, SFT compactness (see \cite{BEHW03} and \cite{Abb14}, in particular \cite[Theorem 3.20]{Abb14}) and pseudo-holomorphic gluing. See \cite{Ekh08}, \cite{EES05} and \cite{EES07} for details.
	
\end{proof}

\paragraph{Augmentations and Legendrian $A_{\infty}$-(co)category.} 

Let $\mathbf{F}_{\mcE}$ be the category with $\mcE$ as set of objects, and morphism space from $E$ to $E'$ equal to $\mathbf{F}$ if $E = E'$, or $0$ if $E \ne E'$.
Assume that we have an augmentation of $CE_{-*} \left( \mathbf{\Lambda} \right)$, i.e. a DG-functor $\varepsilon : CE_{-*} \left( \mathbf{\Lambda} \right) \to \mathbf{F}_{\mcE}$. Denote by $\phi_{\varepsilon}$ the automorphism of $CE_{-*} \left( \mathbf{\Lambda} \right)$ defined by 
\[\phi_{\varepsilon} (c) = c + \varepsilon (c) \]
for every Reeb chord $c$ of $\Lambda$.  We denote by $CE_{-*}^{\varepsilon} \left( \mathbf{\Lambda} \right)$ the DG-category whose underlying graded category is the same as for $CE_{-*} \left( \mathbf{\Lambda} \right)$, but the differential is $\partial_{\varepsilon} = \phi_{\varepsilon} \circ \partial \circ \phi_{\varepsilon}^{-1}$. Now let $\overline{LC_*^{\varepsilon} \left( \mathbf{\Lambda} \right)}$ be the graded pre-category (no composition) with
\begin{enumerate}
	
	\item objects the Legendrians $\Lambda (E)$, $E \in \mcE$,
	
	\item space of morphisms from $\Lambda (E)$ to $\Lambda (E')$ generated by Reeb chords $c$ which starts on $\Lambda(E)$ and ends on $\Lambda(E')$, with grading 
	\[\vert c \vert := -CZ(c). \]
\end{enumerate} 
Observe that, as a graded pre-category, we have
\[CE_{-*}^{\varepsilon} \left( \mathbf{\Lambda} \right) = \mathbf{F}_{\mcE} \oplus \left( \bigoplus \limits_{d \geq 1} \overline{LC_*^{\varepsilon} \left( \mathbf{\Lambda} \right)} \left[ -1 \right]^{\otimes d} \right) . \] 
If we write
\[\left( \partial_{\varepsilon} \right)_{| \overline{LC_*^{\varepsilon} \left( \mathbf{\Lambda} \right)} } = \sum_{d \geq 0} \partial_{\varepsilon}^d \text{ with } \partial_{\varepsilon}^d : \overline{LC_*^{\varepsilon} \left( \mathbf{\Lambda} \right)} \to \overline{LC_*^{\varepsilon} \left( \mathbf{\Lambda} \right)}^{\otimes d}, \]
then $\partial_{\varepsilon}^0 = \varepsilon \circ \partial = 0$. Moreover, the operations $\left( \partial_{\varepsilon}^d \right)_{d \geq 1}$ make $\overline{LC_*^{\varepsilon} \left( \mathbf{\Lambda} \right)}$ a (non-counital) $A_{\infty}$-cocategory (see Definition \ref{definition coalgebra}). We define the coaugmented $A_{\infty}$-cocategory of $\left( \mathbf{\Lambda}, \varepsilon \right)$ to be
\[LC_*^{\varepsilon} \left( \mathbf{\Lambda} \right) := \mathbf{F}_{\mcE} \oplus \overline{LC_*^{\varepsilon} \left( \mathbf{\Lambda} \right)} \]
(the $A_{\infty}$-cooperations are naturally extended so that $1 \in \mathbf{F}_{\mcE} (E, E)$, $E \in \mcE$ are counits).
Now observe that, as a DG-category, 
\[CE_{-*}^{\varepsilon} \left( \mathbf{\Lambda} \right) = \Omega \left( LC_*^{\varepsilon} \left( \mathbf{\Lambda} \right) \right) \]
(see \cite[section 2.2]{EL21} for the cobar construction).
Finally, we define the augmented $A_{\infty}$-category of $\left( \mathbf{\Lambda}, \varepsilon \right)$ to be the graded dual (see \cite[section 2.1.3]{EL21}) of $LC_*^{\varepsilon} \left( \mathbf{\Lambda} \right)$:
\[LA^*_{\varepsilon} \left( \mathbf{\Lambda} \right) = LC_*^{\varepsilon} \left( \mathbf{\Lambda} \right)^{\#} . \]

\subsection{Functoriality}\label{subsection invariance}

Recall that $(V, \xi)$ is a contact manifold such that $H_1 (V)$ is free and the first Chern class of $\xi$ (equipped with any compatible almost complex structure) is 2-torsion. 
Let $\mathbf{M}=(M(E))_{E \in \mcE}$ be a family of $n$-dimensional manifolds.
When we write a map $\mathbf{\Lambda} : \mathbf{M} \to V$, we mean that $\mathbf{\Lambda}$ is a family of maps $\Lambda (E) : M(E) \to V$ indexed by $\mcE$, and we set
\[\Lambda = \bigsqcup_{E \in \mcE} \Lambda (E) : \bigsqcup_{E \in \mcE} M (E) \to V . \]

\begin{defin}\label{definition geometric category}
	
	Let $\alpha$ be a hypertight contact form on $(V, \xi)$.
	We denote by $\mcL_\mathbf{M} (\alpha)$ the bicategory where
	\begin{enumerate}
		
		\item objects are the pairs $(\mathbf{\Lambda}, J)$, where $\mathbf{\Lambda} : \mathbf{M} \to V$ is a family of Legendrian embedding such that $\Lambda$ is chord generic with vanishing Maslov number, and $J$ is a regular almost complex structure on $\xi$,
		\item morphisms from $(\mathbf{\Lambda}_0, J_0)$ to $(\mathbf{\Lambda}_1, J_1)$ are the smooth paths $\Phi = (\mathbf{\Lambda}_t, J_t)_{0 \leq t \leq 1}$ going from $(\mathbf{\Lambda}_0, J_0)$ to $(\mathbf{\Lambda}_1, J_1)$, where $\mathbf{\Lambda}_t : \mathbf{M} \to V$ is a family of Legendrian embeddings and $J_t$ is an almost complex structure on $\xi$, 
		\item homotopies from a morphism $\Phi = (\mathbf{\Lambda}_t, J_t)_{0 \leq t \leq 1} : (\mathbf{\Lambda}_0, J_0) \to (\mathbf{\Lambda}_1, J_1)$ to another morphism $\Phi' = (\mathbf{\Lambda}_t', J_t')_{0 \leq t \leq 1} : (\mathbf{\Lambda}_0, J_0) \to (\mathbf{\Lambda}_1, J_1)$ are the smooth families $(\mathbf{\Lambda}_{s,t}, J_{s,t})_{0 \leq s \leq S, 0 \leq t \leq 1}$, where $\mathbf{\Lambda}_{s,t} : \mathbf{M} \to V$ is a family of Legendrian embeddings, $J_{s,t}$ is an almost complex structure on $\xi$, and
		\[(\mathbf{\Lambda}_{s,0}, J_{s,0}) = (\mathbf{\Lambda}_0, J_0), \; (\mathbf{\Lambda}_{s,1}, J_{s,1}) = (\mathbf{\Lambda}_1, J_1), \quad (\mathbf{\Lambda}_{0,t}, J_{0,t}) = (\mathbf{\Lambda}_t, J_t), \; (\mathbf{\Lambda}_{S,t}, J_{S,t}) = (\mathbf{\Lambda}_t', J_t'). \]
		
	\end{enumerate}
		
\end{defin}

\begin{defin}\label{definition action of contact isotopy}
	
	Let $\alpha$, $\alpha'$ be hypertight contact forms on $(V, \xi)$, and let $\varphi$ be a contactomorphism of $(V, \xi)$ such that $\varphi^* \alpha  = \alpha'$.
	If $\Phi = (\mathbf{\Lambda}_t, J_t)_{0 \leq t \leq 1}$ is a morphism in $\mcL_\mathbf{M} (\alpha)$, we denote by 
	\[\varphi^* \Phi = (\varphi^{-1} (\mathbf{\Lambda}_t), \varphi^* J_t)_{0 \leq t \leq 1} \]
	the corresponding morphism in $\mcL_\mathbf{M} (\alpha')$, and by 
	\[f_{(\mathbf{\Lambda}_t, J_t)}^{\varphi} : CE_{-*} (\mathbf{\Lambda}_t, J_t, \alpha) \to CE_{-*} (\varphi^{-1} (\mathbf{\Lambda}_t), \varphi^* J_t, \alpha') \]
	the DG-functor which sends a Reeb chord $c$ to $\varphi^{-1} (c)$.
	
\end{defin}

\begin{defin}\label{definition handle slide instant}
	
	Let $\alpha$ be a hypertight contact form on $(V, \xi)$, and let $\Phi = (\mathbf{\Lambda}_t, J_t)_{0 \leq t \leq 1}$ be a morphism in $\mcL_\mathbf{M} (\alpha)$.
	A \emph{handle slide instant} in $\Phi$ is a time $t_0$ where $\mathbf{\Lambda}_{t_0}$ is chord generic and has Reeb chords $c_d, \dots, c_1, c_0$ such that the moduli space $\widetilde{\mcM}_{c_d, \dots, c_1, c_0} \left( \mathbf{R} \times \Lambda_{t_0}, J_{t_0}, \alpha \right)$ is not transversely cut out.
	
\end{defin}

\begin{conj}\label{conjecture functoriality}
	
	There exist functors $\mcF_{\alpha}$ from $\mcL_\mathbf{M} (\alpha)$ to the bicategory\footnote{Homotopies between DG-maps are DG-homotopies, see for example \cite[section 2.1]{PR21}.} of DG-categories such that 
	\begin{enumerate}
		
		\item $\mcF_{\alpha}$ sends an object $( \mathbf{\Lambda}, J )$ to $CE_{-*} ( \mathbf{\Lambda}, J, \alpha)$,
		
		\item $\mcF_{\alpha}$ sends a morphism to a homotopy equivalence,
		
		\item if $\varphi$ is a contactomorphism of $(V, \xi)$ such that $\varphi^* \alpha  = \alpha'$ and if $\Phi = (\mathbf{\Lambda}_t, J_t)_{0 \leq t \leq 1}$ is a morphism in $\mcL_\mathbf{M} (\alpha)$, then 
		\[\mcF_{\alpha'} (\varphi^* \Phi) = f_{(\mathbf{\Lambda}_1, J_1)}^{\varphi} \circ \mcF_{\alpha} (\Phi) \circ (f_{(\mathbf{\Lambda}_0, J_0)}^{\varphi})^{-1} \]
		
		\item if $(\varphi_t)_{0 \leq t \leq 1}$ is a contact isotopy of $(V, \xi)$ satisfying $\varphi_t^* \alpha  = \alpha'$ for every $t$, and if $(\mathbf{\Lambda}, J)$ is an object of $\mcL_\mathbf{M} (\alpha)$ such that there is neither birth/death of Reeb chords nor handle slide instants in the path $\Phi' = (\varphi_t^{-1} (\mathbf{\Lambda}), \varphi_t^* J)$, then 
		\[\mcF_{\alpha'} (\Phi') = f_{(\mathbf{\Lambda}, J)}^{\varphi_1} \circ (f_{(\mathbf{\Lambda}, J)}^{\varphi_0})^{-1} . \]
				
	\end{enumerate}
	
\end{conj}

We chose to state the latter result as a conjecture out of caution, even if some parts of it have already been proved, or at least detailed strategies of proofs have been given. The existence of such functors at the category level (without homotopies) has been completely established in the case $(V, \alpha) = (\mathbf{R} \times P, dz - \lambda)$, see \cite[section 2.4]{EES07}. Detailed strategies of proofs for the general case can be found in \cite[section 4]{Ekh08} and \cite[section 5]{EO17}.
 
Note that I proved a weaker version of this result in my thesis by generalizing methods of \cite{EES05} and \cite{PR21}.
The following is the only version of Conjecture \ref{conjecture functoriality} that we will use in this paper. 

\begin{thm}[\cite{Pet22} Theorem 3.8]\label{thm invariance}
	
	Conjecture \ref{conjecture functoriality} is true if we replace the categories $\mcL_\mathbf{M} (\alpha)$ by the subcategories $\mcL_\mathbf{M}^0 (\alpha)$ where
	\begin{enumerate}
		
		\item objects are the pairs $(\mathbf{\Lambda}, J)$ such that $\Lambda$ has finitely many Reeb chords,
		
		\item morphisms from $(\mathbf{\Lambda}_0, J_0)$ to $(\mathbf{\Lambda}_1, J_1)$ are the families $\Phi = (\mathbf{\Lambda}_t, J_t)_{0 \leq t \leq 1}$ such that $\Lambda_t$ is chord generic and has finitely many Reeb chords for every $t$,
		
		\item homotopies from a morphism $\Phi = (\mathbf{\Lambda}_t, J_t)_{0 \leq t \leq 1} : (\mathbf{\Lambda}_0, J_0) \to (\mathbf{\Lambda}_1, J_1)$ to another morphism $\Phi' = (\mathbf{\Lambda}_t', J_t')_{0 \leq t \leq 1} : (\mathbf{\Lambda}_0, J_0) \to (\mathbf{\Lambda}_1, J_1)$ are the families $(\mathbf{\Lambda}_{s,t}, J_{s,t})_{0 \leq s \leq S, 0 \leq t \leq 1}$ such that $\Lambda_{s,t}$ is chord generic and has finitely many Reeb chords for every $s, t$.
		
	\end{enumerate}
	
\end{thm}

\begin{rmk}

	We expect that the finiteness of Reeb chords condition in Theorem \ref{thm invariance} (which is very restrictive) can be easily dropped using (homotopy) colimits of DG-categories diagrams. 
	On the other hand, studying birth/death of Reeb chords phenomena is a more serious issue that we will address in a future work.

\end{rmk}

% !TeX spellcheck = en_US
\section{Legendrian lifts of exact Lagrangians in the circular contactization}\label{subsection proof of main thm}

In this section, we start with a family $\mathbf{L}$ of mutually transverse compact connected exact Lagrangian submanifolds in a Liouville manifold, and we study a Legendrian lift $\mathbf{\Lambda^{\circ}}$ of $\mathbf{L}$ in the circular contactization.  
For the standard contact form, each point on a Legendrian gives rise to a (countable) infinite set of Reeb chords, and thus the situation is degenerate. In section \ref{subsection setting and Legendrian invariants}, we explain how we perturb the contact form and we state our central result, which relates the Chekanov-Eliashberg DG-category of $\mathbf{\Lambda^{\circ}}$ and the Fukaya $A_{\infty}$-category of $\mathbf{L}$. 

\subsection{Setting}\label{subsection setting and Legendrian invariants}

Let $\left( P , \lambda \right)$ be a Liouville manifold, and let 
\[\mathbf{L} = \left( L(E) \right)_{E \in \mcE}, \quad \mcE = \left\{ 1, \dots, N \right\}, \]
be a family of mutually transverse compact connected exact Lagrangian submanifolds in $(P, \lambda)$ such that there are primitives $f_E : L(E) \to \mathbf{R}$ of $\lambda_{| L(E)}$ satisfying $0 \leq f_1 < \dots < f_N \leq 1/2$.
We consider the contact manifold 
\[ \left( V^{\circ}, \xi^{\circ} \right) = \left( S^1 \times P, \mathrm{ker} \, \alpha^{\circ} \right), \text{ where } S^1 = \mathbf{R}_{\theta} / \mathbf{Z}, \, \alpha^{\circ} = d \theta - \lambda, \]
and the family of Legendrian submanifolds 
\[\mathbf{\Lambda^{\circ}} := (\Lambda^{\circ}(E))_{E \in \mcE}, \text{ where } \Lambda^{\circ}(E) = \left\{ (f_E(x), x) \in (\mathbf{R} / \mathbf{Z}) \times P \mid x \in L(E) \right\} . \]
In order for the Chekanov-Eliashberg category of $\mathbf{\Lambda^{\circ}}$ and the Fukaya category of $\mathbf{L}$ to be $\mathbf{Z}$-graded, we assume that $H_1(P)$ is free, that the first Chern class of $P$ (equipped with any almost complex structure compatible with $(- d \lambda)$) is 2-torsion, and that the Maslov classes of the Lagrangians $L(E)$ vanish.

\subsubsection{Reeb chords}\label{subsection perturbed Reeb chords}

Observe that $\Lambda^{\circ} = \bigcup_{E \in \mcE} \Lambda (E)$ is not chord generic for $\alpha^{\circ}$ (see Definition \ref{definition chord generic}). We will choose a compactly supported function $H : P \to \mathbf{R}$, and consider the perturbed contact form 
\[ \alpha_H^{\circ} = e^H \alpha^{\circ} . \]
The Reeb vector field of $\alpha_H^{\circ}$ is then 
\[ R_{\alpha_H^{\circ}} = e^{- H} \left( 
\begin{array}{c}
1 + \lambda \left( X_H \right) \\
X_H
\end{array} \right), \]
where $X_H$ is the unique vector field on $P$ satisfying $\iota_{X_H} d \lambda = -d H$.

We fix a compact neighborhood $K$ of $L$ which is contained in a Weinstein neighborhood of $L$ in $P$. 
It is not hard to see that for every positive integer $N$, the space of smooth functions $H$ on $P$ supported in $K$, such that the $R_{\alpha_H^{\circ}}$-chords of $\Lambda^{\circ}$ with action less than $N$ are generic, is open and dense in $C_K^{\infty} \left( P \right)$. Therefore, the space of functions $H \in C_K^{\infty} \left( P \right)$ such that $\Lambda^{\circ}$ is chord generic with respect to $\alpha_H^{\circ}$ is a Baire subset of $C_K^{\infty} \left( P \right)$. In particular, the latter is dense in $C_K^{\infty} \left( P \right)$. In the following, we choose $H \in C_K^{\infty} \left( P \right)$ such that 
\begin{enumerate}
	
	\item $\Lambda^{\circ}$ is chord generic with respect to $\alpha_H^{\circ}$,
	
	\item $H$ is sufficiently close to $0$ so that
	\[d \theta \left( R_{\alpha_H^{\circ}} \right) = e^{-H} \left( 1 + \lambda \left( X_H \right) \right) \geq 1/2 . \]
	
\end{enumerate}

\begin{exa}\label{example cotangent case}
	
	Assume that we are in the case
	\[\left( P, \lambda \right) = \left( T^* M, p dq \right), \, L = 0_M, \text{ and } H \left( q, p \right) = h(q), \]
	where $h : M \to \mathbf{R}$ is a Morse function (we present this example in order to see what happens, even if $H$ is not compactly supported in $T^*M$).
	The Reeb vector field of $\alpha_H^{\circ}$ is 
	\[ R_{\alpha_H^{\circ}} = e^{- h} \left( 
	\begin{array}{c}
	1 \\
	0 \\
	-d h
	\end{array} \right), \]
	and therefore the Reeb flow satisfies
	\[ \varphi_{R_{\alpha_H^{\circ}}}^t \left( \theta , \left( q,p \right) \right) = \left( \theta + t e^{- h (q)}, \left( q , p - t e^{-h (q)} dh \left( q \right) \right) \right) . \]
	Thus, the $R_{\alpha_H^{\circ}}$-chords of $\Lambda^{\circ}$ are the paths $c : \left[ 0,T \right] \to S^1 \times T^* M$ of the form 
	\[ c(t) = \left( t e^{- h (q_0)} , \left( q_0 , 0 \right) \right), \text{ with } T e^{- h (q_0)} \in \mathbf{Z}_{\geq 1} \text{ and } q_0 \in \mathrm{Crit} \, h. \]
	Observe that these Reeb chords are transverse but lie on top of each others. See Figure \ref{figure perturbed Reeb chords}, where we illustrate this perturbation when $M = S^1$.
	
\end{exa}

\begin{figure}
	\def\svgwidth{1.2\textwidth}
	%% Creator: Inkscape inkscape 0.92.5, www.inkscape.org
%% Accompanies image file 'dessin10.pdf' (pdf, eps, ps)
\begingroup%
  \makeatletter%
  \providecommand\color[2][]{%
    \errmessage{(Inkscape) Color is used for the text in Inkscape, but the package 'color.sty' is not loaded}%
    \renewcommand\color[2][]{}%
  }%
  \providecommand\transparent[1]{%
    \errmessage{(Inkscape) Transparency is used (non-zero) for the text in Inkscape, but the package 'transparent.sty' is not loaded}%
    \renewcommand\transparent[1]{}%
  }%
  \providecommand\rotatebox[2]{#2}%
  \newcommand*\fsize{\dimexpr\f@size pt\relax}%
  \newcommand*\lineheight[1]{\fontsize{\fsize}{#1\fsize}\selectfont}%
  \ifx\svgwidth\undefined%
    \setlength{\unitlength}{718.73001012bp}%
    \ifx\svgscale\undefined%
      \relax%
    \else%
      \setlength{\unitlength}{\unitlength * \real{\svgscale}}%
    \fi%
  \else%
    \setlength{\unitlength}{\svgwidth}%
  \fi%
  \global\let\svgwidth\undefined%
  \global\let\svgscale\undefined%
  \makeatother%
  \begin{picture}(1,0.2679528)%
    \lineheight{1}%
    \setlength\tabcolsep{0pt}%
    \put(0,0){\includegraphics[width=\unitlength,page=1]{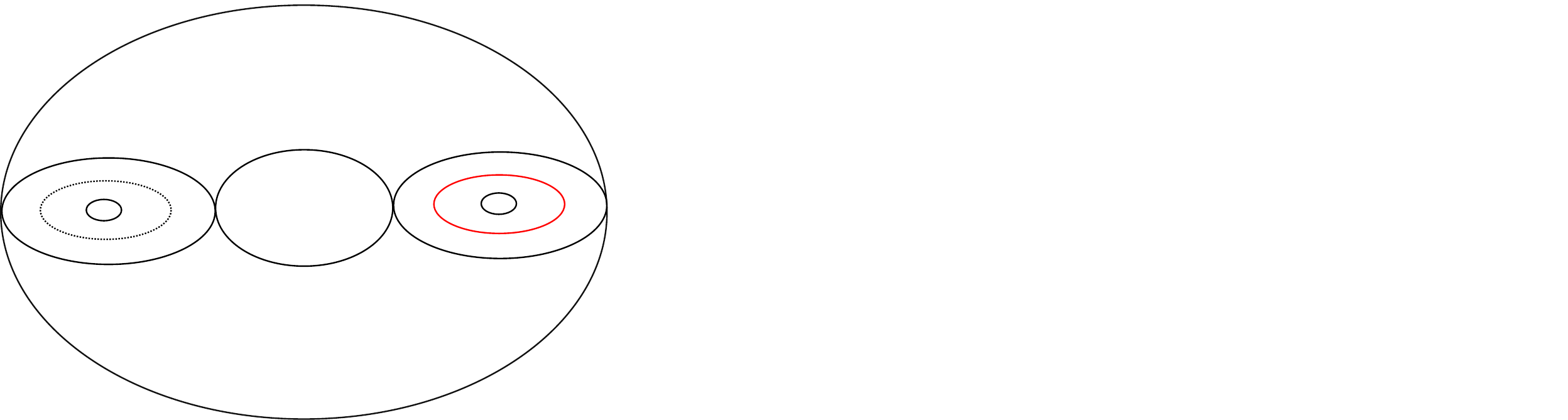}}%
    \put(0.36056249,0.12814535){\color[rgb]{0,0,0}\makebox(0,0)[lt]{\lineheight{1.25}\smash{\begin{tabular}[t]{l}\textcolor{red}{$\Lambda^{\circ}$}\end{tabular}}}}%
    \put(0.2577924,0.08072042){\color[rgb]{0,0,0}\makebox(0,0)[lt]{\lineheight{1.25}\smash{\begin{tabular}[t]{l}$T^*S^1$\end{tabular}}}}%
    \put(0.35444981,0.01167934){\color[rgb]{0,0,0}\makebox(0,0)[lt]{\lineheight{1.25}\smash{\begin{tabular}[t]{l}$S^1 \times T^*S^1$\end{tabular}}}}%
    \put(0,0){\includegraphics[width=\unitlength,page=2]{Invariants_lift/Figures/dessin10.pdf}}%
    \put(0.79490552,0.12836039){\color[rgb]{0,0,0}\makebox(0,0)[lt]{\lineheight{1.25}\smash{\begin{tabular}[t]{l}\textcolor{red}{$\Lambda^{\circ}$}\end{tabular}}}}%
    \put(0.70186705,0.08214395){\color[rgb]{0,0,0}\makebox(0,0)[lt]{\lineheight{1.25}\smash{\begin{tabular}[t]{l}$T^*S^1$\end{tabular}}}}%
    \put(0,0){\includegraphics[width=\unitlength,page=3]{Invariants_lift/Figures/dessin10.pdf}}%
  \end{picture}%
\endgroup%
		
	\caption{Reeb chords (in blue) of $\Lambda^{\circ} = \{ 0 \} \times 0_{S^1}$ for $\alpha^{\circ}$ (on the left) and for $\alpha_H^{\circ}$ (on the right)}	
	\label{figure perturbed Reeb chords}
\end{figure}

\subsubsection{Conley-Zehnder index}

In order to define the Conley-Zehnder index (see section \ref{subsection Conley-Zehnder index}), we need to choose a family $(h_0, h_1, \dots, h_s)$ of embedded circles in $V^{\circ} = S^1 \times P$ which represent a basis of $H_1 ( V^{\circ} )$, and a symplectic trivialization of $\xi^{\circ}$ over each $h_i$. We let $h_0 = S^1 \times \{ a_0 \}$ be some fiber of $S^1 \times P \to P$, and we fix $\left( h_1, \dots, h_s \right)$ to be any family of embedded circles in $P$ which represent a basis of $H_1 ( P )$. 
We choose a symplectic isomorphism $\psi : \left( T_{a_0} P, - d \lambda_{a_0} \right) \xrightarrow{\sim} \left( \mathbf{C}^n, dx \wedge dy \right)$, and then we choose the symplectic trivialization
\[\left( \xi^{\circ}_{| h_0}, d \alpha^{\circ} \right) \xrightarrow{\sim} \left( h_0 \times \mathbf{C}^n, dx \wedge dy \right), \quad \left( \left( \theta, a_0  \right), \left( \lambda_{a_0} (v), v \right) \right) \mapsto \left( \left( \theta, a_0  \right) , e^{2 i \pi r \theta} \psi (v) \right) \]
where $r$ is some integer, that we call \emph{$r$-trivialization of $\xi^{\circ}$ over the fiber}.
Finally, we choose some trivialization of $\xi^{\circ}$ over each $h_i$, $1 \leq i \leq s$.

\begin{exa}\label{example CZ index in cotangent case}
	
	We compute the Conley-Zehnder index of a Reeb chord in the case of Example \ref{example cotangent case}, i.e. when 
	\[\left( P, \lambda \right) = \left( T^* M, p dq \right), \, L = 0_M, \text{ and } H \left( q, p \right) = h(q), \]
	where $h : M \to \mathbf{R}$ is a Morse function .
	In this case, the Reeb flow is given by
	\[ \varphi_{R_{\alpha_H^{\circ}}}^t \left( \theta , \left( q,p \right) \right) = \left( \theta + t e^{-  h (q)}, \left( q , p - t  e^{-  h (q)} dh \left( q \right) \right) \right) . \]
	Let $c : \left[ 0,T \right] \to V^{\circ}$ be a Reeb chord of $\Lambda^{\circ}$. Then there exists a positive integer $k$ and a critical point $q_0$ of $h$ such that 
	\[c(t) = \left( t e^{-  h (q_0)} , \left( q_0 , 0 \right) \right) \text{ and } T e^{-  h (q_0)} = k. \]
	Observe that $c(0)  = c(T)$, and thus there is no need to choose a capping path for $c$. 
	Besides, for every $u$ in $T_{q_0} M$, we have  
	\[D \varphi_{R_{\alpha_H^{\circ}}}^t \left( c(0) \right) \left( 0, u, 0 \right) = \left( 0, u, -  t e^{-  h (q_0)} D^2 h \left( q_0 \right) u \right) . \]
	In order to compute the index of $c$, we first choose coordinates $\left( x_1, \dots, x_n \right)$ around $q_0 \in M$ in which  
	\[h = h \left( q_0 \right) + \frac{1}{2} \sum_{j=1}^{\mathrm{dim}(M)} \sigma_j x_j^2, \text{ where } \sigma_j = \pm 1, \]
	and we extend it to symplectic coordinates $\left( x_1, \dots, x_n, y_1, \dots, y_n \right)$ around $\left( q_0, 0 \right) \in T^*M$ by setting 
	\[y_j \left( q,p \right) = \langle p, \frac{\partial}{\partial x_j} (q) \rangle . \] 
	Our choice of trivialization for a fiber of $S^1 \times 
	P \to P$ induces the trivialization 
	\[e^{2 i \pi r k t / T} \left( dx + i dy \right) : c^{-1} \xi^{\circ} \xrightarrow{\sim} \left( \mathbf{R} / T \mathbf{Z} \right) \times \mathbf{C}^n \]
	(observe that $\xi^{\circ}_{c(t)} = \{ 0 \} \times T_{\left( q_0, 0 \right)} \left( T^*M \right)$).
	Accordingly, the path $t \mapsto D \varphi_{R_{\alpha_H^{\circ}}}^t \left( T_{c(0)} \Lambda^{\circ} \right)$ induces a path of Lagrangians 
	\[\Gamma_c : t \in \left[ 0, T \right] \mapsto \left\{ \left( e^{2 i \pi m k t / T} \left( u_j - i  t e^{- h(q_0)} \sigma_j u_j \right) \right)_{1 \leq j \leq n} \mid u \in \mathbf{R}^n \right\} \subset \mathbf{C}^n . \]
	We close this path using a counterclockwise rotation $\Gamma$, and call the resulting loop $\overline{\Gamma_c}$. In order to compute the Conley-Zehnder index of $c$, we have to look at how $\overline{\Gamma_c}$ intersects the Lagrangian $i \mathbf{R}^n$ (as explained in \cite[section 2.2]{EES05}).
	Observe that $\Gamma_c$ intersects $i \mathbf{R}^n$ positively $2rk$ times, so that $\Gamma_c$ contributes $2rk$ to the Conley-Zehnder index of $c$. Moreover, since $\Gamma$ is a counterclockwise rotation bringing
	\[\left\{ \left( u_j - i  T e^{- h(q_0)} \sigma_j u_j \right)_{1 \leq j \leq n} \mid u \in \mathbf{R}^n \right\} \text{ to } \mathbf{R}^n, \]
	the contributions to the intersection between $\Gamma$ and $i \mathbf{R}^n$ come from the negative eigenvalues $\sigma_j$. The computation done in \cite[Lemma 3.4]{EES05bisbis} implies that $\Gamma$ contributes $\mathrm{ind} (q_0)$ to the Conley-Zehnder index of $c$. We conclude that the Conley-Zehnder index of $c$ is
	\[CZ \left( c \right) = \mu \left( \overline{\Gamma_c} \right) = 2rk + \mathrm{ind} (q_0) . \] 

\end{exa}

\subsubsection{Main result and strategy of proof}\label{subsection invariants and statement of the result}

Let $j$ be an almost complex structure on $P$ compatible with $\left( -d\lambda \right)$, and let $J^{\circ}$ be its lift to a complex structure on $\xi^{\circ}$.
Recall from section \ref{subsection Chekanov-Eliashberg algebra} the definition of the Chekanov-Eliashberg DG-category of a family of Legendrians. In our situation, $CE_{-*}^r \left( \mathbf{\Lambda^{\circ}} \right) = CE_{-*} \left( \mathbf{\Lambda^{\circ}}, J^{\circ}, \alpha_{H}^{\circ} \right)$ (with grading induced by the $r$-trivialization of $\xi^{\circ}$ over the fiber) is an Adams-graded DG-algebra, where the Adams-degree of a Reeb chord $c$ is the number of times $c$ winds around the fiber. Besides, the map $CE_{-*}^r \left( \mathbf{\Lambda^{\circ}} \right) \to \mathbf{F}$ which sends every Reeb chord to zero (and preserves units) defines an augmentation of $CE_{-*}^r \left( \mathbf{\Lambda^{\circ}} \right)$. 

\begin{rmk}
	
	In the case of Example \ref{example cotangent case}, the cohomological degree of a Reeb chord $c$ in $CE_{-*}^r \left( \mathbf{\Lambda^{\circ}} \right)$ corresponding to a positive integer $k$ and a critical point $q_0$ is 
	\[1 - CZ (c) = 1 -2rk - \mathrm{ind} (q_0) \]
	(see Example \ref{example CZ index in cotangent case}).
	
\end{rmk}

Besides, we denote by $\mcF uk (\mathbf{L})$ the Fukaya category generated by the Lagrangians $L(E)$ (see for example \cite[chapter 2]{Sei08}), and by $\overrightarrow{\mcF uk} (\mathbf{L})$ its directed subcategory:
\[\hom_{\overrightarrow{\mcF uk} (\mathbf{L})} (L(E), L(E')) =
\left\{
\begin{array}{cl}
	\langle L(E) \cap L(E') \rangle & \text{if } E < E' \\
	\mathbf{F} & \text{if } E = E' \\
	0 & \text{if } E > E' \\
\end{array}
\right. \]
(see \cite[paragraph (5n)]{Sei08}). 

Let $\mathbf{F} \left[ t_m \right]$ be the augmented Adams-graded associative algebra generated by a variable $t_m$ of bidegree $\left( m,1 \right)$. 
Observe that if $\mcC$ is a subcategory of an $A_{\infty}$-category $\mcD$ with $\mathrm{ob} (\mcC) = \mathrm{ob} (\mcD)$, then $\mcC \oplus (t_m \mathbf{F} [t_m] \otimes \mcD)$ is naturally an Adams-graded $A_{\infty}$-category, where the Adams degree of $t_m^k \otimes x$ equals $k$.
Moreover, we denote by $E(-) = B(-)^{\#}$ (graded dual of bar construction) the Koszul dual functor (see \cite[section 2.3]{EL21} or \cite[section 2]{LPWZ08}).

\begin{thm}[Theorem \ref{thm mainthm} in the Introduction]\label{thm mainthm restatement}
	
	Koszul duality holds for $CE_{-*}^r \left( \mathbf{\Lambda^{\circ}} \right)$, and there is a quasi-equivalence of augmented Adams-graded $A_{\infty}$-categories
	\[E \left( CE_{-*}^r \left( \mathbf{\Lambda^{\circ}} \right) \right) \simeq \overrightarrow{\mcF uk} (\mathbf{L}) \oplus \left( t_{2r} \mathbf{F} \left[ t_{2r} \right] \otimes \mcF uk (\mathbf{L}) \right). \]
	
\end{thm}
	
\begin{coro}\label{coro topological expression}
	
	If $L$ is a connected compact exact Lagrangian and $\Lambda^{\circ}$ is a Legendrian lift of $L$ in the circular contactization, then there is a quasi-equivalence of augmented DG-algebras 
	\[ CE_{-*}^1 \left( \Lambda^{\circ} \right) \simeq C_{-*} \left( \Omega \left( \mathbf{CP}^{\infty} \rtimes L \right) \right). \]
	
\end{coro}

\begin{proof}
	
	Let $x_0$ be the base point of $\mathbf{CP}^{\infty}$, and set $P := \mathbf{CP}^{\infty} \setminus \{ x_0 \}$. 
	Observe that 
	\[\left( P \times L \right)^* = P^* \wedge L^* = \mathbf{CP}^{\infty} \wedge L^* = \mathbf{CP}^{\infty} \rtimes L . \]
	We have  
	\[\mathbf{F} \oplus \left( t_2 \mathbf{F} \left[ t_2 \right] \otimes CF^* \left( L \right) \right) \simeq \mathbf{F} \oplus \left( t_2 \mathbf{F} \left[ t_2 \right] \otimes C^* \left( L \right) \right) \simeq C^* \left( \left( P \times L \right)^* \right) \simeq C^* \left( \mathbf{CP}^{\infty} \rtimes L \right). \]
	Thus, it follows from Theorem \ref{thm mainthm restatement} that 
	\[E \left( CE_{-*}^1 \left( \Lambda^{\circ} \right) \right) \simeq C^* \left( \mathbf{CP}^{\infty} \rtimes L \right) . \]
	Since Koszul duality holds for $CE_{-*}^1 \left( \Lambda^{\circ} \right)$, 
	\[CE_{-*}^1 \left( \Lambda^{\circ} \right) \simeq E \left( C^* \left( \mathbf{CP}^{\infty} \rtimes L \right)  \right) . \]
	Observe that the graded algebra $H^*\left( \mathbf{CP}^{\infty} \rtimes L \right)$ is locally finite (i.e. each degree component is finitely generated) and simply connected (i.e. its augmentation ideal is concentrated in components of degree strictly greater than 1). Thus, according to the homological perturbation lemma (see \cite[Proposition 1.12]{Sei08}), we can assume that $C^*\left( \mathbf{CP}^{\infty} \rtimes L \right)$ is a locally finite  and simply connected $A_{\infty}$ model for the DG-algebra of cochains on $\mathbf{CP}^{\infty} \rtimes L$. 
	Therefore, \cite[Lemma 10]{EL21} implies that
	\[CE_{-*}^1 \left( \Lambda^{\circ} \right) \simeq \Omega \left( C_{-*} \left( \mathbf{CP}^{\infty} \rtimes L \right) \right). \]
	Now, since $\mathbf{CP}^{\infty} \rtimes L$ is simply connected, Adams result (see \cite{Ada56}, \cite{AH56} and also \cite{EL21}) yields
	\[\Omega \left( C_{-*} \left( \mathbf{CP}^{\infty} \rtimes L \right) \right) \simeq C_{-*} \left( \Omega \left( \mathbf{CP}^{\infty} \rtimes L \right) \right) . \]
	This concludes the proof.
	
\end{proof}

\paragraph{Strategy of proof.}

We explain the startegy to compute $E ( CE_{-*}^r \left( \mathbf{\Lambda^{\circ}} \right) )$.
Recall from the last paragraph of section \ref{subsection Chekanov-Eliashberg algebra} that there is a coaugmented $A_{\infty}$-cocategory $LC_* \left( \mathbf{\Lambda^{\circ}} \right)$ such that 
\[CE_{-*}^r \left( \mathbf{\Lambda^{\circ}} \right) = \Omega ( LC_* \left( \mathbf{\Lambda^{\circ}} \right) ). \]
$LC_* \left( \mathbf{\Lambda^{\circ}} \right)$ inherits an Adams-grading from $CE_{-*}^r \left( \mathbf{\Lambda^{\circ}} \right)$ (the same), and we denote by $LA^* \left( \mathbf{\Lambda^{\circ}} \right)$ its graded dual (see \cite[section 2.1.3]{EL21}). In our situation, $LA^* \left( \mathbf{\Lambda^{\circ}} \right)$ is an augmented Adams-graded $A_{\infty}$-category whose augmentation ideal is generated by the Reeb chords of $\mathbf{\Lambda^{\circ}}$ (and the Adams-degree of a Reeb chord $c$ is the number of times $c$ winds around the fiber).
Since there is a quasi-isomorphism $B \left( \Omega C \right) \simeq C$ for every $A_{\infty}$-cocategory $C$ (see \cite[section 2.2.2]{EL21}), it follows that 
\[E \left( CE_{-*}^r \left( \mathbf{\Lambda^{\circ}} \right) \right) = B \left( CE_{-*}^r \left( \mathbf{\Lambda^{\circ}} \right) \right)^{\#} \simeq LC_* \left( \mathbf{\Lambda^{\circ}} \right)^{\#} = LA^* \left( \mathbf{\Lambda^{\circ}} \right) \]
(graded dual preserves quasi-isomorphisms).

\begin{rmk}
	
	In the case of Example \ref{example cotangent case}, the cohomological degree of a Reeb chord $c$ in $LA^* \left( \mathbf{\Lambda^{\circ}} \right)$ corresponding to a positive integer $k$ and a critical point $q_0$ is 
	\[CZ (c) = 2rk + \mathrm{ind} (q_0) \]
	(see Example \ref{example CZ index in cotangent case}).
	
\end{rmk}

In order to compute $LA^* \left( \mathbf{\Lambda^{\circ}} \right)$, we apply a sequence of geometric modifications to the situation. Each of the section numbered from \ref{subsection lift} to \ref{subsection projection on P} explains one of these modifications and describes how the algebraic invariants change accordingly. 
The main ingredients are respectively: Theorem \ref{thm mapping torus in strict situation} in section \ref{subsection lift}, Lemma \ref{lemma rectify contact form} in section \ref{subsection rectification of the contact form}, Theorem \ref{thm invariance} in section \ref{subsection back to the original acs}, and \cite[Theorem 2.1]{DR16} in section \ref{subsection projection on P}.
Finally, we end the proof in section \ref{subsection proof of the main result} using Theorem \ref{thm mapping torus in weak situation} (Theorem \ref{thm mapping torus in weak situation introduction} in the introduction).

\subsection{Lift to $\mathbf{R} \times P$}\label{subsection lift}

In the following we will consider the contact manifold 
\[ \left( V, \xi \right) = \left( \mathbf{R}_{\theta} \times P, \ker \, \alpha \right) \text{ where } \alpha = d \theta - \lambda, \]
and the family of Legendrian submanifolds
\[\mathbf{\Lambda} := (\Lambda^n(E))_{(n, E) \in \mathbf{Z} \times \mcE}, \text{ where } \Lambda^{\theta}(E) = \left\{ (f_E(x) + \theta, x) \in \mathbf{R} \times P \mid x \in L(E) \right\} . \]
Moreover we set $\Lambda^n := \bigcup_{E \in \mcE} \Lambda^n (E)$ and $\Lambda := \bigcup_{n \in \mathbf{Z}} \Lambda^n$.

Recall from section \ref{subsection perturbed Reeb chords} that we chose a compactly supported function $H : P \to \mathbf{R}$ such that 
\begin{enumerate}
	
	\item $\Lambda^{\circ}$ is chord generic with respect to $\alpha_H^{\circ}$,
	
	\item $H$ is sufficiently close to $0$ so that
	\[d \theta \left( R_{\alpha_H^{\circ}} \right) = e^{-H} \left( 1 + \lambda \left( X_H \right) \right) \geq 1/2 . \]
	
\end{enumerate}
We consider the contact form 
\[ \alpha_{H} = e^{ H} \alpha, \]
with Reeb vector field 
\[ R_{\alpha_H} = e^{-  H} \left( 
\begin{array}{c}
1 +  \lambda \left( X_H \right) \\
 X_H
\end{array} \right). \]
Moreover, we denote by $J$ the lift of $J^{\circ}$ to an almost complex structure on $\xi$. 

\subsubsection{The $A_{\infty}$-category $\mcA$}

\begin{defin}\label{definition category A}
	
	We consider the $A_{\infty}$-category $\mcA$ defined as follows
	\begin{enumerate}
		
		\item the objects of $\mcA$ are the Legendrians $\Lambda^n (E)$, $(n, E) \in \mathbf{Z} \times \mcE$,
		
		\item the space of morphisms from $\Lambda^i (E)$ to $\Lambda^j (E')$ is either generated by the $R_{\alpha_{H}}$-chords from $\Lambda^i (E)$ to $\Lambda^j (E')$ if $(i,E)<(j,E')$, or $\mathbf{F}$ if $(i,E)=(j,E')$, or $0$ otherwise,
		
		\item the operations are such that $1 \in \mcA \left( \Lambda^n (E), \Lambda^n (E) \right)$ is a strict unit, and for every sequence $(i_0, E_0) < \dots < (i_d, E_d)$, for every sequence of Reeb chords 
		\[\left( c_1, \dots, c_d \right) \in \mcR \left( \Lambda^{i_0} (E_0), \Lambda^{i_1} (E_1) \right) \times \dots \times \mcR \left( \Lambda^{i_{d-1}} (E_{d-1}), \Lambda^{i_d} (E_d) \right), \]
		we have 
		\[\mu_{\mcA} \left( c_1, \dots, c_d \right) = \sum \limits_{c_0 \in \mcR \left( \Lambda^{i_0} (E_0), \Lambda^{i_d} (E_d) \right)} \# \mcM_{c_d, \dots, c_1, c_0} \left( \mathbf{R} \times \Lambda, J, \alpha_{H} \right) c_0 \]
		(see Definition \ref{definition moduli spaces} for the moduli spaces).
		
	\end{enumerate} 
	
\end{defin}

\subsubsection{The quasi-autoequivalence $\tau$} 

We introduce an $A_{\infty}$-endofunctor of $\mcA$ that will be important in the following.

\begin{defin}\label{definition functor tau}
	
	We denote by $\tau : \mcA \to \mcA$ the $A_{\infty}$-functor defined as follows:
	\begin{enumerate}
		
		\item on objects, $\tau$ sends $\Lambda^n (E)$ to $\Lambda^{n+1} (E)$,
		
		\item on morphisms, the map 
		\[\tau : \mcA \left( \Lambda^{i_0} (E_0), \Lambda^{i_1} (E_1) \right) \otimes \dots \otimes \mcA \left( \Lambda^{i_{d-1}} (E_{d-1}), \Lambda^{i_d} (E_d) \right) \to \mcA \left( \Lambda^{i_0+1} (E_0), \Lambda^{i_d+1} (E_d) \right)  \]
		is obtained by dualizing the components of the DG-isomorphism
		\[CE_{-*} \left( \left( \Lambda^{n+1} \right)_{i_0 \leq n \leq i_d}, J, \alpha_{H} \right) \to CE_{-*} \left( \left( \Lambda^n \right)_{i_0 \leq n \leq i_d}, J, \alpha_{H} \right) \]
		induced by the path $ \left( \left( \Lambda^{n+1-t} \right)_{i_0 \leq n \leq i_d}, J \right)_{0 \leq t \leq 1}$ (see Theorem \ref{thm invariance}). 
		
	\end{enumerate}
	
\end{defin}

\begin{lemma}\label{lemma tau is a quasi-equivalence}
	
	The $A_{\infty}$-functor $\tau : \mcA \to \mcA$ is a quasi-equivalence.
	
\end{lemma}

\begin{proof}
	
	Consider the $A_{\infty}$-functor $\overline{\tau} : \mcA \to \mcA$ defined as follows:
	\begin{enumerate}
		
		\item on objects, $\overline{\tau}$ sends $\Lambda^n (E)$ to $\Lambda^{n-1} (E)$,

		\item on morphisms, the map 
		\[\overline{\tau} : \mcA \left( \Lambda^{i_0} (E_0), \Lambda^{i_1} (E_1) \right) \otimes \dots \otimes \mcA \left( \Lambda^{i_{d-1}} (E_{d-1}), \Lambda^{i_d} (E_d) \right) \to \mcA \left( \Lambda^{i_0-1} (E_0), \Lambda^{i_d-1} (E_d) \right)  \]
		is obtained by dualizing the components of the \emph{inverse} of the DG-isomorphism
		\[CE_{-*} \left( \left( \Lambda^{n+1} \right)_{i_0 \leq n \leq i_d}, J, \alpha_{H} \right) \to CE_{-*} \left( \left( \Lambda^n \right)_{i_0 \leq n \leq i_d}, J, \alpha_{H} \right) \]
		induced by the path $ \left( \left( \Lambda^{n+1-t} \right)_{i_0 \leq n \leq i_d}, J \right)_{0 \leq t \leq 1}$. 
		
	\end{enumerate} 
	Then $\tau \circ \overline{\tau} = \overline{\tau} \circ \tau = \mathrm{id}_{\mcA}$. 
	
\end{proof}

Here the $\mathbf{Z}$-splitting 
\[\mathbf{Z} \times \mcE \xrightarrow{\sim} \mathrm{ob} \left( \mcA \right), \quad (n, E) \mapsto \Lambda^n (E) \]
is compatible with the quasi-autoequivalence $\tau$ in the sense of Definition \ref{definition group-action}. As explained there, this turns $\mcA$ into an Adams-graded $A_{\infty}$-category: the Adams-degree of a morphism $c \in \mcA \left( \Lambda^i (E), \Lambda^j (E') \right)$ is $\left( j-i \right)$. 

\subsubsection{Relation between $LA^* \left( \Lambda^{\circ} \right)$ and $\left( \mcA, \tau \right)$}

We now explain how $LA^* \left( \Lambda^{\circ} \right)$ and $\left( \mcA, \tau \right)$ are related.
See Figure \ref{figure relation usual and circular contactization}, where we illustrate the action of the projection $\Pi_{S^1 \times P}$ in the case 
\[\left( P, \lambda \right) = \left( T^* S^1, p dq \right), \, L = 0_{S^1}, \text{ and } H \left( q, p \right) = h(q), \]
where $h : S^1 \to \mathbf{R}$ is a Morse function. 

\begin{figure}
	\def\svgwidth{1.3\textwidth}
	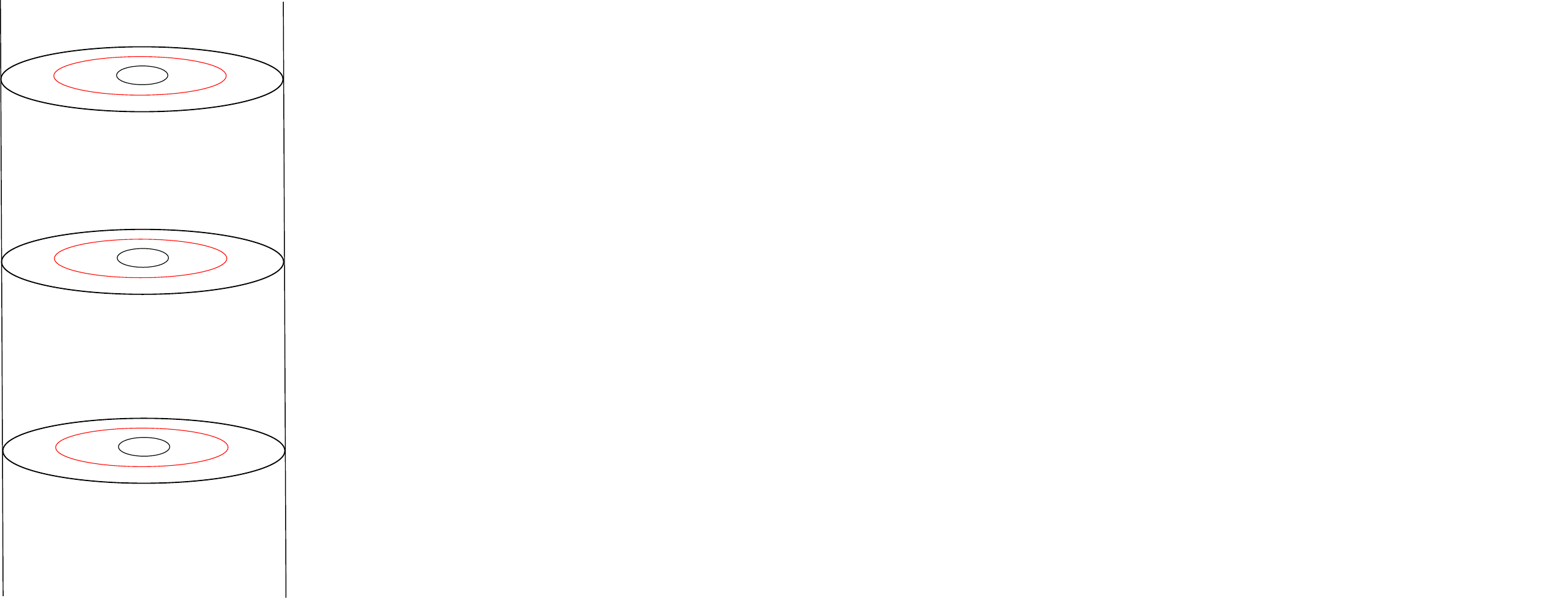		
	\caption{Action of the projection $\Pi_{S^1 \times T^* S^1}$}	
	\label{figure relation usual and circular contactization}
\end{figure}

\begin{lemma}\label{lemma tau is strict}
	
	The $A_{\infty}$-functor $\tau$ is strict, and it sends a Reeb chord $t \mapsto \left( \theta \left( t \right) , x \left( t \right) \right)$ in $\mcA \left( \Lambda^i (E) , \Lambda^j (E') \right)$ to the Reeb chord $t \mapsto \left( \theta \left( t \right) + 1 , x \left( t \right) \right)$ in $\mcA \left( \Lambda^{i+1} (E) , \Lambda^{j+1} (E') \right)$.
	In particular, $\tau$ acts bijectively on hom-sets. 
	
\end{lemma}

\begin{proof}
	
	Recall that $\alpha_H = e^H \alpha$, with $H$ a function defined on the base manifold $P$.
	In particular, the flow $\varphi_{\partial_{\theta}}^t$ of $\partial_{\theta}$ is a strict contactomorphism of $\left( V, \alpha_{H} \right)$.
	Moreover, since $J$ is the lift of an almost complex structure $j$ on $P$, we have
	\[\left( \left( \Lambda^{n+1-t} \right)_{i_0 \leq n \leq i_d}, J \right) = \left( \left( (\varphi_{\partial_{\theta}}^t)^{-1} \Lambda^{n+1} \right)_{i_0 \leq n \leq i_d}, (\varphi_{\partial_{\theta}}^t)^* J \right) . \]
	The result follows from Theorem \ref{thm invariance}.
	
\end{proof}

We denote by $\mcA_{\tau}$ the Adams-graded $A_{\infty}$-category associated to $\tau$ as in Definition \ref{definition category of coinvariants}. 

\begin{lemma}\label{lemma A^rond = A_Z}
	
	There is a quasi-isomorphism of Adams-graded $A_{\infty}$-categories
	\[LA^* \left( \Lambda^{\circ} \right) \simeq \mcA_{\tau} . \]
	
\end{lemma}

\begin{proof}
	
	Consider the map which sends a Reeb chord $c \in \mcR \left( \Lambda^i (E), \Lambda^j (E') \right)$ to the corresponding chord $\Pi_{S^1 \times P} \left( c \right) \in \mcR \left( \Lambda^{\circ} (E), \Lambda^{\circ} (E') \right)$ (where $\Pi_{S^1 \times P} : \mathbf{R} \times P \to S^1 \times P$ is the projection). According to Lemma \ref{lemma tau is strict}, $\Pi_{S^1 \times P} \left( \tau c \right) = \Pi_{S^1 \times P} \left( c \right)$, and thus the map $c \mapsto \Pi_{S^1 \times P} \left( c \right)$ induces a map $\psi : \mcA_{\tau} \to LA^* \left( \Lambda^{\circ} \right)$. Moreover, observe that $\psi$ is a bijection on hom-spaces. It remains to prove that $\psi$ is an $A_{\infty}$-map. This follows from the fact that the map
	\[ u = \left( \sigma , v \right) \mapsto \left( \sigma , \Pi_{S^1 \times P} \circ v \right) \]
	induces a bijection
	\[\mcM_{c_d, \dots, c_1, c_0} \left( \mathbf{R} \times \Lambda, J, \alpha_{H} \right) \xrightarrow{\sim} \mcM_{\psi (c_d), \dots, \psi (c_1), \psi (c_0)} \left( \mathbf{R} \times \Lambda^{\circ}, J^{\circ}, \alpha_{H}^{\circ} \right) . \] 
	
\end{proof}

\begin{lemma}\label{lemma A^rond = mapping torus of tau}
	
	The Adams-graded $A_{\infty}$-category $LA^* \left( \Lambda^{\circ} \right)$ is quasi-equivalent to the mapping torus of $\tau : \mcA \to \mcA$ (see Definition \ref{definition mapping torus}). 
	
\end{lemma}

\begin{proof}
	
	This follows directly from Theorem \ref{thm mapping torus in strict situation} using Lemmas \ref{lemma tau is strict} and \ref{lemma A^rond = A_Z}. 
	
\end{proof}

\subsection{Rectification of the contact form}\label{subsection rectification of the contact form}

Now that we are in the usual contactization, we have the following result. 

\begin{lemma}\label{lemma rectify contact form} 
	
	There exists a contactomorphism $\phi_H$ of $\left( V,\xi \right)$ such that 
	\[ \phi_H^* \alpha_{H} = \alpha . \]
	
\end{lemma}

\begin{proof}
	
	Recall that $\alpha_{H} = e^{ H} \alpha$, with $H$ a compactly supported function on the base manifold $P$ such that $e^{-H} \left( 1 + \lambda \left( X_H \right) \right) \geq 1/2$.
	
	Assume that there is a contact isotopy $\left( \phi_{t} \right)_{0 \leq t \leq 1}$ such that $\phi_0 = \mathrm{id}$ and 
	\begin{equation}\label{eq1}
	\phi_{t}^* \alpha_{t H} = \alpha
	\end{equation}
	for every $t$. Let $\left( F_{t} \right)_{t}$ be the family of functions on $V$ such that 
	\[ \frac{d}{d t} \phi_{t} = Y_{F_{t}} \circ \phi_{t}, \]
	where $Y_F$ is the contact vector field on $V$ satisfying 
	\begin{align*}
	\left\{
	\begin{array}{ccl}
	\alpha \left( Y_F \right) & = & F \\
	\iota_{Y_F} d \alpha & = & d F \left( R_{\alpha} \right) \alpha  - d F . 
	\end{array}
	\right.
	\end{align*}
	Taking the derivative of equation \eqref{eq1} with respect to $t$, we get 
	\begin{equation}\label{eq2}
	H + d \left( e^{t H} F_{t} \right) \left( R_{\alpha_{t H}} \right) = 0
	\end{equation}
	because $Y_F$ satisfies 
	\begin{align*}
	\left\{
	\begin{array}{ccl}
	\alpha_{t H} \left( Y_F \right) & = & e^{t H} F \\
	\iota_{Y_F} d \alpha_{t H} & = & d \left( e^{t H} F \right) \left( R_{\alpha_{t H}} \right) \alpha_{t H}  - d \left( e^{t H} F \right) .
	\end{array}
	\right.
	\end{align*}
	Besides, we deduce from 
	\[ R_{\alpha_{t H}} = e^{- t H} \left( 
	\begin{array}{c}
	1 + t \lambda \left( X_H \right) \\
	t X_H
	\end{array} \right), \quad \iota_{X_H} d \lambda = - dH, \]
	that 
	\[ d H \left( R_{\alpha_{t H}} \right) = 0 . \]
	Then equation \eqref{eq2} gives 
	\begin{equation}\label{eq3}
	d F_{t} \left( R_{\alpha_{t H}} \right) = - H e^{- t H} .
	\end{equation} 
	
	Conversely, if $\left( F_{t} \right)_{t}$ is a family of functions on $V$ satisfying equation \eqref{eq3}, then the contact isotopy $\left( \phi_{t} \right)_{t}$ defined by 
	\[ \phi_0 = \mathrm{id} \text{ and } \frac{d}{d t} \phi_{t} = Y_{F_{t}} \circ \phi_{t}, \]
	satisfies
	\[ \frac{d}{d t} \left( \phi_{t}^* \alpha_{t H} \right) = 0 , \]
	and thus $\phi_H := \phi_1$ gives the desired result.
	
	Therefore, it remains to find a family $\left( F_{t} \right)_{t}$ satisfying equation \eqref{eq3}. First recall that 
	\[ R_{\alpha_{t H}} = e^{- t H} \left( 
	\begin{array}{c}
	1 + t \lambda \left( X_H \right) \\
	t X_H
	\end{array} \right) . \]
	By assumption on $H$, the function $d \theta \left( R_{\alpha_{t H}} \right)$ is greater than $1/2$ for every $t \in \left[ 0,1 \right]$. Thus, for every $t \in \left[ 0,1 \right]$ and every $\left( \theta , x \right)$ in $V$, there exists a unique real number $\rho_{t} \left( \theta , x \right)$ such that 
	\[ \varphi_{R_{\alpha_{t H}}}^{- \rho_{t} \left( \theta , x \right) } \left( \theta , x \right) \in \{ 0 \} \times P . \]
	Then we let 
	\[ F_{t}:= - \rho_{t} H e^{- t H} . \]
	For every real number $t$, we have 
	\[ F_{t} \circ \varphi_{R_{\alpha_{t H}}}^t = - \left( \rho_{t} \circ \varphi_{R_{\alpha_{t H}}}^t \right) H e^{- t H} \text{ because } dH \left( R_{\alpha_{t H}} \right) = 0 . \]
	But the map $\varphi_{R_{\alpha_{t H}}}^{- \rho_{t} \circ \varphi_{R_{\alpha_{t H}}}^t + t}$ takes its values in $\{ 0 \} \times P$ by definition of $\rho_{t}$, so by uniqueness we have
	\[ \rho_{t} \circ \varphi_{R_{\alpha_{t H}}}^t = \rho_{t} + t . \] 
	Then we have 
	\[ F_{t} \circ \varphi_{R_{\alpha_{t H}}}^t = - \left( \rho_{t} + t \right) H e^{- t H}, \]
	and thus 
	\[ d F_{t} \left( R_{\alpha_{t H}} \right) = - H e^{- t H} . \]
	This concludes the proof.
	
\end{proof}

\begin{exa}\label{example rectify contact form in cotangent case}
	
	Assume that we are in the case 
	\[\left( P, \lambda \right) = \left( T^* M, p dq \right), \, L = 0_M, \text{ and } H \left( q, p \right) = h(q), \]
	where $h : M \to \mathbf{R}$ is a Morse function. 
	Then the diffeomorphism $\phi_H$ defined by
	\[\phi_H^{-1} \left( \theta, \left( q,p \right) \right) = \left( \theta e^{h(q)}, \left( q, e^{h(q)} p + \theta e^{h(q)} dh (q) \right) \right) \]
	satisfies $\phi_H^* \alpha_H = \alpha$.
	With this choice of $\phi_H$, we have in particular
	\[\phi_H^{-1} \left( \{ \theta \} \times 0_M \right) = j^1 \left( \theta e^h \right) \subset \mathbf{R} \times T^*M . \] 
	
\end{exa}

\subsubsection{The $A_{\infty}$-category $\mcA_1$}

In the following, we fix a contactomorphism $\phi_H$ as in Lemma \ref{lemma rectify contact form}.
We define an $A_{\infty}$-category $\mcA_1$, which is roughly obtained by pulling back the data of $\mcA$ by $\phi_H$. 

\begin{defin}
	
	Let 
	\[\Lambda_H^{\theta} (E) := \phi_H^{-1} \left( \Lambda^{\theta} (E) \right), \, \Lambda_H^n := \bigcup\limits_{E \in \mcE} \Lambda_H^n (E), \, \Lambda_H := \bigcup\limits_{n \in \mathbf{Z}} \Lambda_H^n \text{ and } J_H := \phi_H^* J. \]
	We consider the $A_{\infty}$-category $\mcA_1$ defined as follows
	\begin{enumerate}
		
		\item the objects of $\mcA_1$ are the Legendrians $\Lambda_H^n (E)$, $(n, E) \in \mathbf{Z} \times E$,
		
		\item the vector space $\mcA_1 \left( \Lambda_H^i (E), \Lambda_H^j (E') \right)$ is either generated by the $R_{\alpha}$-chords from $\Lambda_H^i (E)$ to $\Lambda_H^j (E')$ if $(i,E)<(j,E')$, or $\mathbf{F}$ if $(i,E)=(j,E')$, or $0$ otherwise, and
		
		\item the operations are such that $1 \in \mcA_1 \left( \Lambda_H^n (E), \Lambda_H^n (E) \right)$ is a strict unit, and for every sequence $(i_0, E_0) < \dots < (i_d, E_d)$, for every sequence of Reeb chords 
		\[\left( c_1, \dots, c_d \right) \in \mcR \left( \Lambda_H^{i_0} (E_0), \Lambda_H^{i_1} (E_1) \right) \times \dots \times \mcR \left( \Lambda_H^{i_{d-1}} (E_{d-1}), \Lambda_H^{i_d} (E_d) \right) \]
		we have 
		\[\mu_{\mcA_1} \left( c_1, \dots, c_d \right) = \sum \limits_{c_0 \in \mcR \left( \Lambda_H^{i_0} (E_0), \Lambda_H^{i_d} (E_d) \right)} \# \mcM_{c_d, \dots, c_1, c_0} \left( \mathbf{R} \times \Lambda_H, J_H, \alpha \right) c_0 \]
		(see Definition \ref{definition moduli spaces} for the moduli spaces).
		
	\end{enumerate}
	
\end{defin}

\subsubsection{The quasi-autoequivalence $\tau_1$}

\begin{defin}\label{definition functor tau_1}
	
	We denote by $\tau_1 : \mcA_1 \to \mcA_1$ the $A_{\infty}$-functor defined as follows:
	\begin{enumerate}
		
		\item on objects, $\tau_1 \left( \Lambda_H^n (E) \right) = \Lambda_H^{n+1} (E)$,
		
		\item on morphisms, the map 
		\[\tau_1^d : \mcA_1 \left( \Lambda_H^{i_0} (E_0), \Lambda_H^{i_1} (E_1) \right) \otimes \dots \otimes \mcA_1 \left( \Lambda_H^{i_{d-1}} (E_{d-1}), \Lambda_H^{i_d} (E_d) \right) \to \mcA_1 \left( \Lambda_H^{i_0+1} (E_0), \Lambda_H^{i_d+1} (E_d) \right)  \]
		is obtained by dualizing the components of the DG-isomorphism
		\[CE_{-*} \left( \left( \Lambda_H^{n+1} \right)_{i_0 \leq n \leq i_d}, J_H, \alpha \right) \to CE_{-*} \left( \left( \Lambda_H^n \right)_{i_0 \leq n \leq i_d}, J_H, \alpha \right) \]
		induced by the path $ \left( \left( \Lambda_H^{n+1-t} \right)_{i_0 \leq n \leq i_d}, J_H \right)_{0 \leq t \leq 1}$ (see Theorem \ref{thm invariance}). 
		
	\end{enumerate}
	
\end{defin}

\begin{lemma}\label{lemma tau_1 is a quasi-equivalence}
	
	The $A_{\infty}$-functor $\tau_1 : \mcA_1 \to \mcA_1$ is a quasi-equivalence.
	
\end{lemma}

\begin{proof}
	
	This follows because $\tau_1$ is defined by dualizing the components of a DG-isomorphism (see the proof of Lemma \ref{lemma tau is a quasi-equivalence}). 
	
\end{proof}

Here the $\mathbf{Z}$-splitting 
\[\mathbf{Z} \times \mcE \xrightarrow{\sim} \mathrm{ob} \left( \mcA_1 \right), \quad (n, E) \mapsto \Lambda_H^n (E), \]
is compatible with the quasi-autoequivalence $\tau_1$ in the sense of Definition \ref{definition group-action}. As explained there, this turns $\mcA_1$ into an Adams-graded $A_{\infty}$-category. 

\subsubsection{Relation between $\left( \mcA, \tau \right)$ and $\left( \mcA_1, \tau_1 \right)$}

We now explain how the pairs $\left( \mcA, \tau \right)$ and $\left( \mcA_1, \tau_1 \right)$ are related.
See Figure \ref{figure action of the contactomorphism} where we illustrate the action of the contactomorphism $\phi_H^{-1}$ in the case 
\[\left( P, \lambda \right) = \left( T^* S^1, p dq \right), \, L = 0_{S^1}, \text{ and } H \left( q, p \right) = h(q), \]
where $h : S^1 \to \mathbf{R}$ is a Morse function.

\begin{figure}
	\def\svgwidth{1.4\textwidth}
	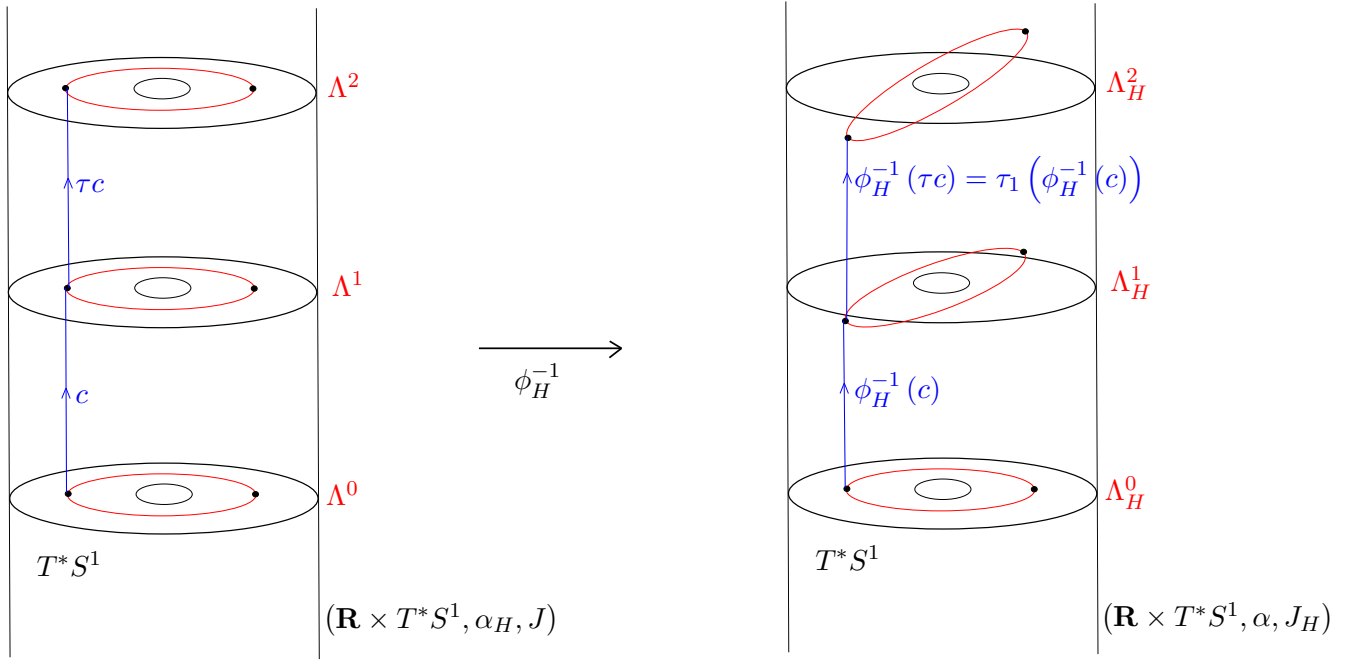		
	\caption{Action of the contactomorphism $\phi_H^{-1}$}	
	\label{figure action of the contactomorphism}
\end{figure}

\begin{lemma}\label{lemma relation A - A_1}
	
	There is a strict $A_{\infty}$-isomorphism $\zeta_1 : \mcA \to \mcA_1$ defined as follows:
	\begin{enumerate}
		
		\item on objects, $\zeta_1 \left( \Lambda^n (E) \right) = \Lambda_H^n (E)$,
		
		\item on morphisms, $\zeta_1$ sends a Reeb chord $c$ in $\mcA \left( \Lambda^i (E), \Lambda^j (E') \right)$ to the Reeb chord 
		\[\zeta_1 \left( c \right) = \phi_H^{-1} \circ c \]
		in $\mcA_1 \left( \Lambda_H^i (E), \Lambda_H^j (E') \right)$.  
		
	\end{enumerate} 

\end{lemma}

\begin{proof}
	
	We have to show that $\zeta_1$ is an $A_{\infty}$-map.
	This follows from the fact that the map 
	\[ u = \left( \sigma , v \right) \mapsto \left( \sigma , \phi_H^{-1} \circ v \right) \]
	induces a bijection
	\[\mcM_{c_d, \dots, c_1, c_0} \left( \mathbf{R} \times \Lambda, J, \alpha_H \right) \xrightarrow{\sim} \mcM_{\phi_H^{-1} (c_d) \dots \phi_H^{-1} (c_1), \phi_H^{-1} (c_0)} \left( \mathbf{R} \times \Lambda_H, J_H, \alpha \right) . \]
	
\end{proof}

\begin{lemma}\label{lemma relation tau - tau_1}
	
	We have 
	\[\tau_1 = \zeta_1 \circ \tau \circ \zeta_1^{-1} . \]
	
\end{lemma}

\begin{proof}
	
	This follows from Theorem \ref{thm invariance} using that $\phi_H^* \alpha_H = \alpha$ and
	\[\left( \left( \Lambda_H^{n+1-t} \right)_{i_0 \leq n \leq i_d}, J_H \right) = \left( \left( \phi_H^{-1} \Lambda^{n+1-t} \right)_{i_0 \leq n \leq i_d}, \phi_H^* J \right). \]
	
\end{proof}

\begin{lemma}\label{lemma mapping torus of tau = mapping torus of tau_1}
	
	The mapping torus of $\tau : \mcA \to \mcA$ is quasi-equivalent to the mapping torus of $\tau_1 : \mcA_1 \to \mcA_1$ (see Definition \ref{definition mapping torus}).
	
\end{lemma}

\begin{proof}
	
	According to Lemma \ref{lemma relation tau - tau_1} the following diagram of Adams-graded $A_{\infty}$-categories is commutative 
	\[\begin{tikzcd}
	\mcA \ar[d, "\zeta_1"] & \mcA \sqcup \mcA \ar[l, "\mathrm{id} \sqcup \mathrm{id}" above] \ar[r, "\mathrm{id} \sqcup \tau"] \ar[d, "\zeta_1 \sqcup \zeta_1"] & \mcA \ar[d, "\zeta_1"] \\
	\mcA_1 & \mcA_1 \sqcup \mcA_1 \ar[l, "\mathrm{id} \sqcup \mathrm{id}" above] \ar[r, "\mathrm{id} \sqcup \tau_1"] & \mcA_1 .
	\end{tikzcd} \]
	Moreover, each vertical arrow is a quasi-equivalence according to Lemma \ref{lemma relation A - A_1}. Thus the result follows from Proposition \ref{prop invariance of homotopy colimits}.
	
\end{proof}

\subsection{Back to the original almost complex structure}\label{subsection back to the original acs}

In this section, we introduce an $A_{\infty}$-category $\mcA_2$ which has same objects and morphisms as $\mcA_1$, but whose operations count punctured discs in $\mathbf{R} \times V$ which are pseudo-holomorphic for the almost complex structure induced by $\alpha$ and $J$ (instead of $J_H$). 

\subsubsection{The $A_{\infty}$-category $\mcA_2$}

Recall that we chose a contactomorphism $\phi_H$ as in Lemma \ref{lemma rectify contact form}, and recall that
\[\Lambda_H^{\theta} (E) := \phi_H^{-1} \left( \Lambda^{\theta} (E) \right), \, \Lambda_H^n := \bigcup\limits_{E \in \mcE} \Lambda_H^n (E), \, \Lambda_H := \bigcup\limits_{n \in \mathbf{Z}} \Lambda_H^n \text{ and } J_H := \phi_H^* J. \]

\begin{defin}
	
	We consider the $A_{\infty}$-category $\mcA_2$ defined as follows
	\begin{enumerate}
		
		\item the objects of $\mcA_2$ are the Legendrians $\Lambda_H^n (E)$, $(n, E) \in \mathbf{Z} \times \mcE$,
		
		\item the vector space $\mcA_2 \left( \Lambda_H^i (E), \Lambda_H^j (E') \right)$ is either generated by the $R_{\alpha}$-chords from $\Lambda_H^i (E)$ to $\Lambda_H^j (E')$ if $(i,E)<(j,E')$, or $\mathbf{F}$ if $(i,E)=(j,E')$, or $0$ otherwise, and
		
		\item the operations are such that $1 \in \mcA_2 \left( \Lambda_H^n, \Lambda_H^n \right)$ is a strict unit, and for every sequence $(i_0, E_0) < \dots < (i_d, E_d)$, for every sequence of Reeb chords 
		\[\left( c_1, \dots, c_d \right) \in \mcR \left( \Lambda_H^{i_0} (E_0), \Lambda_H^{i_1} (E_1) \right) \times \dots \times \mcR \left( \Lambda_H^{i_{d-1}} (E_{d-1}), \Lambda_H^{i_d} (E_d) \right) \]
		we have 
		\[\mu_{\mcA_2} \left( c_1, \dots, c_d \right) = \sum \limits_{c_0 \in \mcR \left( \Lambda_H^{i_0} (E_0), \Lambda_H^{i_d} (E_d) \right)} \# \mcM_{c_d, \dots, c_1, c_0} \left( \mathbf{R} \times \Lambda_H, J, \alpha \right) c_0 . \]
		
	\end{enumerate}
	
\end{defin}

\subsubsection{The quasi-autoequivalence $\tau_2$}

\begin{defin}\label{definition functor tau_2}
	
	We denote by $\tau_2 : \mcA_1 \to \mcA_1$ the $A_{\infty}$-functor defined as follows:
	\begin{enumerate}
		
		\item on objects, $\tau_2 \left( \Lambda_H^n (E) \right) = \Lambda_H^{n+1} (E)$,
		
		\item on morphisms, the map 
		\[\tau_2^d : \mcA_1 \left( \Lambda_H^{i_0} (E_0), \Lambda_H^{i_1} (E_1) \right) \otimes \dots \otimes \mcA_1 \left( \Lambda_H^{i_{d-1}} (E_{d-1}), \Lambda_H^{i_d} (E_d) \right) \to \mcA_1 \left( \Lambda_H^{i_0+1} (E_0), \Lambda_H^{i_d+1} (E_d) \right)  \]
		is obtained by dualizing the components of the DG-isomorphism
		\[CE_{-*} \left( \left( \Lambda_H^{n+1} \right)_{i_0 \leq n \leq i_d}, J, \alpha \right) \to CE_{-*} \left( \left( \Lambda_H^n \right)_{i_0 \leq n \leq i_d}, J, \alpha \right) \]
		induced by the path $ \left( \left( \Lambda_H^{n+1-t} \right)_{i_0 \leq n \leq i_d}, J \right)_{0 \leq t \leq 1}$ (see Theorem \ref{thm invariance} or \cite[Proposition 2.6]{EES07}). 
		
	\end{enumerate}
	
\end{defin}

\begin{lemma}\label{lemma tau_2 is a quasi-equivalence}
	
	The $A_{\infty}$-functor $\tau_2 : \mcA_2 \to \mcA_2$ is a quasi-equivalence.
	
\end{lemma}

\begin{proof}
	
	This follows from the fact that $\tau_2$ is defined by dualizing the components of a DG-isomorphism (see the proof of Lemma \ref{lemma tau is a quasi-equivalence}). 
	
\end{proof}

Here the $\mathbf{Z}$-splitting 
\[\mathbf{Z} \times \mcE \xrightarrow{\sim} \mathrm{ob} \left( \mcA_2 \right), \quad (n, E) \mapsto \Lambda_H^n (E) \]
is compatible with the quasi-autoequivalence $\tau_2$ in the sense of Definition \ref{definition group-action}. As explained there, this turns $\mcA_2$ into an Adams-graded $A_{\infty}$-category. 

\subsubsection{Relation between $\left( \mcA_1, \tau_1 \right)$ and $\left( \mcA_2, \tau_2 \right)$}

\begin{lemma}\label{lemma relation A_1 - A_2}
	
	Choose a generic path $(J_t^{12})_{0 \leq t \leq 1}$ such that $J_0^{12} = J$ and $J_1^{12} = J_H$. 
	There is an $A_{\infty}$-isomorphism $\zeta_{12} : \mcA_1 \to \mcA_2$ defined as follows
	\begin{enumerate}
		
		\item on objects, $\zeta_{12} \left( \Lambda_H^n (E) \right) = \Lambda_H^n (E)$, 
		
		\item on morphisms, the map
		\[\zeta_{12} : \mcA_1 \left( \Lambda_H^{i_0} (E_0), \Lambda_H^{i_1} (E_1) \right) \otimes \dots \otimes \mcA_1 \left( \Lambda_H^{i_{d-1}} (E_{d-1}), \Lambda_H^{i_d} (E_d) \right) \to \mcA_2 \left( \Lambda_H^{i_0} (E_0), \Lambda_H^{i_d} (E_d) \right) \]
		is obtained by dualizing the components of the DG-isomorphism 
		\[CE_{-*} \left( \left( \Lambda_H^n \right)_{i_0 \leq n \leq i_d}, J, \alpha \right) \to CE_{-*} \left( \left( \Lambda_H^n \right)_{i_0 \leq n \leq i_d}, J_H, \alpha \right) \]
		induced by the path $(\left( \Lambda_H^n \right)_{i_0 \leq n \leq i_d}, J_t^{12})_{0 \leq t \leq 1}$ (see Theorem \ref{thm invariance}).
		
	\end{enumerate}
	
\end{lemma}

\begin{proof}
	
	We have to prove that $\zeta_{12}$ is an isomorphism. This follows from the fact that it is defined by dualizing the components of a DG-isomorphism (see the proof of Lemma \ref{lemma tau is a quasi-equivalence}). 
	
\end{proof}

\begin{lemma}\label{lemma relation tau_1 - tau_2}
	
	The $A_{\infty}$-functor $\tau_2$ is homotopic to $\zeta_{12} \circ \tau_1 \circ \zeta_{12}^{-1}$ (see \cite[paragraph (1h)]{Sei08}). 
	
\end{lemma}

\begin{proof}
	
	First recall that $\tau_1$ is obtained by dualizing the components of the DG-map
	\[CE_{-*} \left( \left( \Lambda_H^{n+1} \right)_{i_0 \leq n \leq i_d}, J_H, \alpha \right) \to CE_{-*} \left( \left( \Lambda_H^n \right)_{i_0 \leq n \leq i_d}, J_H, \alpha \right) \]
	induced by the path $\left( \left( \Lambda_H^{n+1-t} \right)_{i_0 \leq n \leq i_d}, J_H \right)_{0 \leq t \leq 1}$. Thus, $\zeta_{12} \circ \tau_1$ is obtained by dualizing the components of the composition
	\[CE_{-*} \left( \left( \Lambda_H^{n+1} \right)_{i_0 \leq n \leq i_d}, J, \alpha \right) \to CE_{-*} \left( \left( \Lambda_H^{n+1} \right)_{i_0 \leq n \leq i_d}, J_H, \alpha \right) \to CE_{-*} \left( \left( \Lambda_H^n \right)_{i_0 \leq n \leq i_d}, J_H, \alpha \right) . \]
	On the other hand, $\tau_2$ is obtained by dualizing the components of the DG-map 
	\[CE_{-*} \left(\left( \Lambda_H^{n+1} \right)_{i_0 \leq n \leq i_d}, J, \alpha \right) \to CE_{-*} \left( \left( \Lambda_H^n \right)_{i_0 \leq n \leq i_d}, J, \alpha \right) \]
	induced by the path $\left( \left( \Lambda_H^{n+1-t} \right)_{i_0 \leq n \leq i_d}, J \right)_{0 \leq t \leq 1}$. Thus, $\tau_2 \circ \zeta_{12}$ is obtained by dualizing the components of the composition
	\[CE_{-*} \left( \left( \Lambda_H^{n+1} \right)_{i_0 \leq n \leq i_d}, J, \alpha \right) \to CE_{-*} \left( \left( \Lambda_H^n \right)_{i_0 \leq n \leq i_d}, J, \alpha \right) \to CE_{-*} \left( \left( \Lambda_H^n \right)_{i_0 \leq n \leq i_d}, J_H, \alpha \right) . \]
	According to Theorem \ref{thm invariance}, the DG-maps used to define $\zeta_{12} \circ \tau_1$ and $\tau_2 \circ \zeta_{12}$ are DG-homotopic. Therefore the $A_{\infty}$-functors $\zeta_{12} \circ \tau_1$ and $\tau_2 \circ \zeta_{12}$ are homotopic. This concludes the proof.
	
\end{proof}

\begin{lemma}\label{lemma mapping torus of tau_1 = mapping torus of tau_2}
	
	The mapping torus of $\tau_1 : \mcA_1 \to \mcA_1$ is quasi-equivalent to the mapping torus of $\tau_2 : \mcA_2 \to \mcA_2$ (see Definition \ref{definition mapping torus}).
	
\end{lemma}

\begin{proof}
	
	Let $\tau_{12} := \zeta_{12} \circ \tau_1 \circ \zeta_{12}^{-1}$.
	Consider the following commutative diagram of Adams-graded $A_{\infty}$-categories
	\[\begin{tikzcd}
	\mcA_1 \ar[d, "\zeta_{12}"] & \mcA_1 \sqcup \mcA_1 \ar[l, "\mathrm{id} \sqcup \mathrm{id}" above] \ar[r, "\mathrm{id} \sqcup \tau_1"] \ar[d, "\zeta_{12} \sqcup \zeta_{12}"] & \mcA_1 \ar[d, "\zeta_{12}"] \\
	\mcA_2 & \mcA_2 \sqcup \mcA_2 \ar[l, "\mathrm{id} \sqcup \mathrm{id}" above] \ar[r, "\mathrm{id} \sqcup \tau_{12}"] & \mcA_2 .
	\end{tikzcd} \]
	Each vertical arrow is a quasi-equivalence according to Lemma \ref{lemma relation A_1 - A_2}, so it follows from Proposition \ref{prop invariance of homotopy colimits} that the mapping torus of $\tau_1$ is quasi-equivalent to the mapping torus of $\tau_{12}$. 
	Now according to Lemma \ref{lemma relation tau_1 - tau_2}, $\tau_{12}$ is homotopic to $\tau_2$. Thus the result follows from Proposition \ref{prop homotopic functors induce quasi-isomorphic homotopy colimits}.
	
\end{proof}

\subsection{Projection to $P$}\label{subsection projection on P}

\subsubsection{The $A_{\infty}$-category $\mcO$}

In order to define the $A_{\infty}$-category $\mcO$, we need to introduce moduli spaces of pseudo-holomorphic discs in $P$. 

\begin{defin}\label{definition moduli spaces for Lagrangians}
	
	Let $\mathbf{L} = (L^n (E))_{(n, E) \in \mathbf{Z} \times \mcE}$ be a family of mutually transverse connected compact exact Lagrangians in $(P, \lambda)$. 
	Consider a sequence of integers $i_0 < \dots < i_d$, and a family of intersection points $\left( x_0, x_1, \dots, x_d \right)$, where
	\[x_0 \in L^{i_0} (E_0) \cap L^{i_d} (E_d) \text{ and } x_k \in L^{i_{k-1}} (E_{k-1}) \cap L^{i_k} (E_k), \, 1 \leq k \leq d. \] 
	\begin{enumerate}
		
		\item If $d=1$, we denote by $\mcM_{x_1, x_0} \left( \mathbf{L}, j \right)$ the set of equivalence classes of maps $u : \mathbf{R} \times \left[ 0, 1 \right] \to P$ such that
		\begin{itemize}                  
			\item $u$ maps $\mathbf{R} \times \{ 0 \}$ to $L^{i_0} (E_0)$ and $\mathbf{R} \times \{ 1 \}$ to $L^{i_1 (E_1)}$,
			\item $u$ satisfies the asymptotic conditions
			\[u \left( s, t \right) \underset{s \to - \infty}{\longrightarrow} x_1 \text{ and } u \left( s, t \right) \underset{s \to + \infty}{\longrightarrow} x_0, \]
			\item $u$ is $(i, j)$-holomorphic,
		\end{itemize}
		where two maps $u$ and $u'$ are identified if there exists $s_0 \in \mathbf{R}$ such that $u' (\cdot, \cdot) = u (\cdot + s_0, \cdot)$.
		
		\item If $d \geq 2$, we denote by $\mcM_{x_d, \dots, x_1, x_0} \left( \mathbf{L}, j \right)$ the set of of pairs $\left( r, u \right)$ such that
		\begin{itemize}
			\item $r \in \mcR^{d+1}$ and $u : \Delta_r \to P$ maps the boundary arc $\left( \zeta_{k+1}, \zeta_k \right)$ of $\Delta_r$ to $L^{i_k} (E_k)$,
			\item $u$ satisfies the asymptotic conditions
			\[\left( u \circ \epsilon_k (r) \right) \left( s, t \right) \underset{s \to - \infty}{\longrightarrow} x_k \text{ and } \left( u \circ \epsilon_0 (r) \right) \left( s, t \right) \underset{s \to + \infty}{\longrightarrow} x_0, \]
			\item $u$ is $(i, j)$-holomorphic.
		\end{itemize}
		
	\end{enumerate}
	
\end{defin}

Recall that we chose a contactomorphism $\phi_H$ as in Lemma \ref{lemma rectify contact form}.
We set 
\[L_H^n := \Pi_{P} \left( \Lambda_H^n (E) \right) \subset P \text{ and } \mathbf{L}_H := \left( L_H^n (E) \right)_{(n, E) \in \mathbf{Z} \times \mcE}. \]

\begin{defin}\label{definition category O}
	
	We denote by $\mcO$ the $A_{\infty}$-category defined as follows:
	\begin{enumerate}
		
		\item the objects of $\mcO$ are the Lagrangians $L_H^n (E)$, $(n, E) \in \mathbf{Z} \times \mcE$, 
		
		\item the space of morphisms from $L_H^i (E)$ to $L_H^j (E')$ is either generated by $L_H^i (E) \cap L_H^j (E')$ if $(i, E) < (j, E')$, or $\mathbf{F}$ if $(i, E) = (j, E')$, or $0$ otherwise, and
		
		\item the operations are such that $1 \in \mcO \left( L_H^n (E), L_H^n (E) \right)$ is a strict unit, and for every sequence $(i_0, E_0) < \dots < (i_d, E_d)$, for every sequence of intersection points 
		\[\left( x_1, \dots, x_d \right) \in \left( L_H^{i_0} (E_0) \cap L_H^{i_1} (E_1) \right) \times \dots \times \left( L_H^{i_{d-1}} (E_{d-1}) \cap L_H^{i_d} (E_d) \right), \]
		we have 
		\[\mu_{\mcO} \left( x_1, \dots, x_d \right) = \sum \limits_{x_0 \in L_H^{i_0} (E_0) \cap L_H^{i_d} (E_d)} \# \mcM_{x_d, \dots, x_1, x_0} \left( \mathbf{L}_H, j \right) x_0 . \]
		
	\end{enumerate} 	
	
\end{defin}

\subsubsection{The quasi-autoequivalence $\gamma$}

Before defining the $A_{\infty}$-functor $\gamma : \mcO \to \mcO$, we recall Legendrian contact homology as defined in \cite{EES07}. To each generic Legendrian $\Lambda$ in $\mathbf{R} \times P$, the authors associate a semi-free DG-algebra $A = A \left( \Lambda, j \right)$ generated by the self-intersection points of $\Pi_P \left( \Lambda \right)$, with a differential $\partial : A \to A$ defined using $j$-holomorphic discs in $P$. In our case, the differential of $A \left( \bigsqcup_k \Lambda_H^k (E), j \right)$ on a generator $x_0 \in L_H^{i_0} (E_0) \cap L_H^{i_d} (E_d)$ is given by 
\[\partial x_0 = \sum_{\left( x_1, \dots, x_d \right)} \# \mcM_{x_d, \dots, x_1, x_0} \left( \mathbf{L}_H, j \right) x_d \cdots x_1  \]
where the sum is over the sequences 
\[\left( x_1, \dots, x_d \right) \in \left( L_H^{i_0} (E_0) \cap L_H^{i_1} (E_1) \right) \times \dots \times \left( L_H^{i_{d-1}} (E_{d-1}) \cap L_H^{i_d} (E_d) \right) . \]
According to \cite[Theorem 2.1]{DR16}, Legendrian contact homology as defined in \cite{EES07} coincides with the version exposed in section \ref{section Legendrian invariants}: 
\[A \left( \Lambda, j \right) = CE_* \left( \Lambda, \left( D \Pi_P \right)_{|\xi}^* j, \alpha \right) . \]
We introduced this version only because it makes clearer the fact that some operations are defined using pseudo-holomorphics polygons in the base $P$.

\begin{defin}\label{definition functor gamma}
	
	We denote by $\gamma : \mcO \to \mcO$ the $A_{\infty}$-functor defined as follows
	\begin{enumerate}
		
		\item on objects, $\gamma \left( L_H^n (E) \right) = L_H^{n+1} (E)$,
		
		\item on morphisms, the map
		\[\gamma : \mcO \left( L_H^{i_0} (E_0), L_H^{i_1} (E_1) \right) \otimes \dots \otimes \mcO \left( L_H^{i_{d-1}} (E_{d-1}), L_H^{i_d} (E_d) \right) \to \mcO \left( L_H^{i_0 + 1} (E_0), L_H^{i_d+1} (E_d) \right) \]
		is obtained by dualizing the components of the DG-isomorphism
		\begin{align*}
		A \left( \bigsqcup \limits_{k=i_0}^{i_d} \Lambda_H^{k+1}, j \right) & = CE_{-*} \left( \mathbf{R} \times \bigsqcup \limits_{k=i_0}^{i_d} \Lambda_H^{k+1},  \left( D \Pi_P \right)_{|\xi}^* j, \alpha \right) \\
		 & \to CE_{-*} \left( \mathbf{R} \times \bigsqcup \limits_{k=i_0}^{i_d} \Lambda_H^k,  \left( D \Pi_P \right)_{|\xi}^* j, \alpha \right) = A \left( \bigsqcup \limits_{k=i_0}^{i_d} \Lambda_H^k, j \right)
		\end{align*}
		induced by the Legendrian isotopy $\left( \bigsqcup \limits_{k=i_0}^{i_d} \Lambda_H^{k+1-t} \right)_{0 \leq t \leq 1}$ (see Theorem \ref{thm invariance}).
		
	\end{enumerate}
	
\end{defin}

\begin{lemma}\label{lemma gamma is a quasi-equivalence}
	
	The $A_{\infty}$-functor $\gamma : \mcO \to \mcO$ is a quasi-equivalence.
	
\end{lemma}

\begin{proof}
	
	This follows from the fact that $\gamma$ is defined by dualizing the components of a DG-isomorphism (see the proof of Lemma \ref{lemma tau is a quasi-equivalence}). 
	
\end{proof}

Here the $\mathbf{Z}$-splitting 
\[\mathbf{Z} \times \mcE \xrightarrow{\sim} \mathrm{ob} \left( \mcO \right), \quad (n, E) \mapsto L_H^n (E) \]
is compatible with the quasi-autoequivalence $\gamma$ in the sense of Definition \ref{definition group-action}. As explained there, this turns $\mcO$ into an Adams-graded $A_{\infty}$-category. 

\subsubsection{Relation between $\left( \mcA_2, \tau_2 \right)$ and $\left( \mcO, \gamma \right)$}

We now explain how the pairs $\left( \mcA_2, \tau_2 \right)$ and $\left( \mcO, \gamma \right)$ are related.
See Figure \ref{figure action of the projection to P}, where we illustrate the action of the projection $\Pi_P$ in the case 
\[\left( P, \lambda \right) = \left( T^* S^1, p dq \right), \, L = 0_{S^1}, \text{ and } H \left( q, p \right) = h(q), \]
where $h : S^1 \to \mathbf{R}$ is a Morse function.

\begin{figure}
	\def\svgwidth{1\textwidth}
	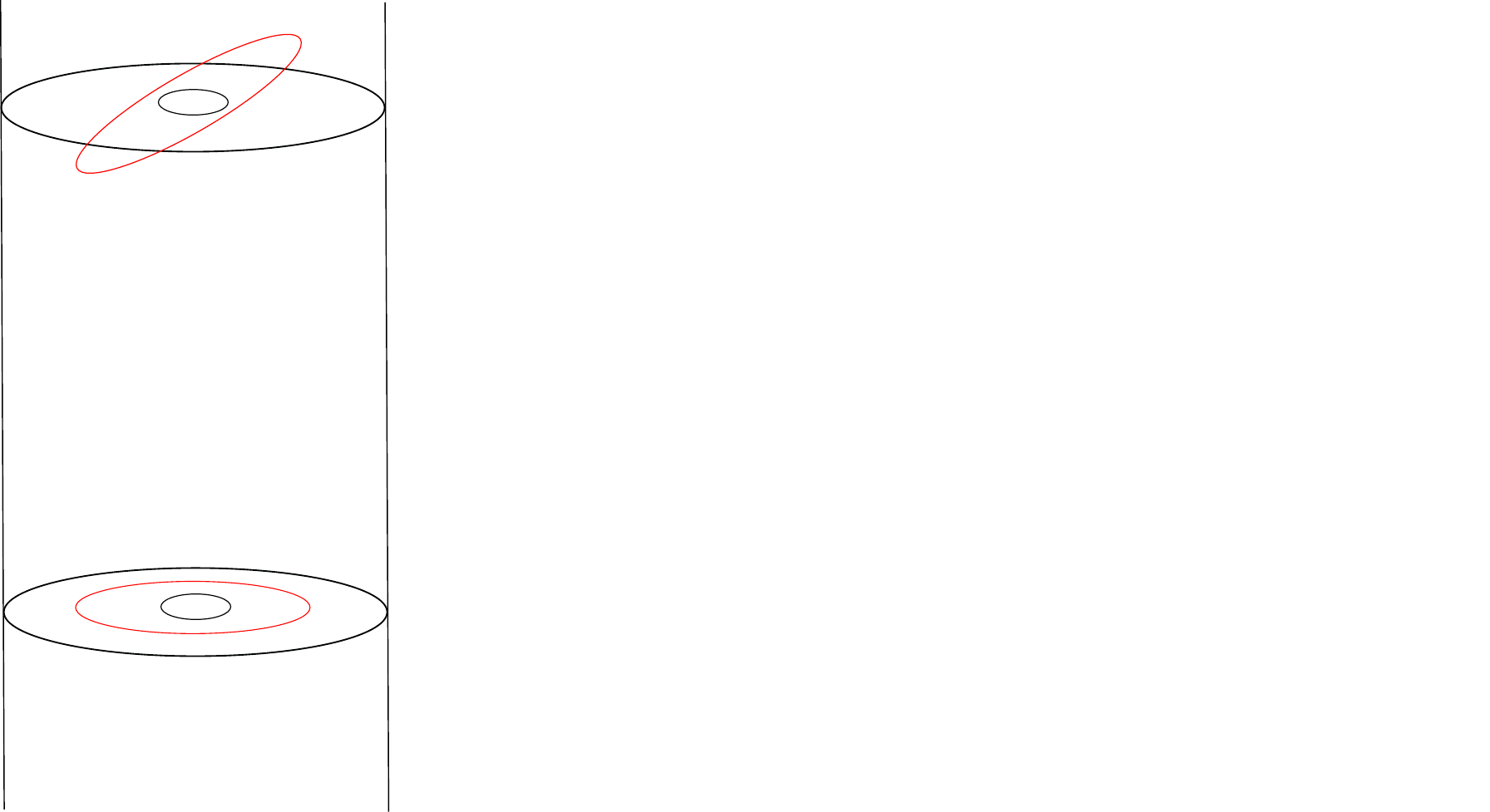		
	\caption{Action of the projection $\Pi_{T^*S^1}$}	
	\label{figure action of the projection to P}
\end{figure}
	
\begin{lemma}\label{lemma relation A_2 - O}
	
	There is a strict $A_{\infty}$-isomorphism $\zeta_2 : \mcA_2 \to \mcO$ defined as follows:
	\begin{enumerate}
		
		\item on objects, $\zeta_2 \left( \Lambda_H^n (E) \right) = L_H^n (E)$,
		
		\item on morphisms, $\zeta_2$ sends a Reeb chord $c$ in $\mcA_2 \left( \Lambda_H^i (E), \Lambda_H^j (E') \right)$ to the intersection point 
		\[\zeta_2 \left( c \right) = \Pi_P \left( c \right) \]
		in $\mcO \left( L_H^i (E), L_H^j (E') \right)$.
		
	\end{enumerate}
	
\end{lemma}

\begin{proof}
	
	We have to show that $\zeta_2$ is an $A_{\infty}$-map.
	Since $J = \left( D \Pi_P \right)_{|\xi}^* j$, it follows from \cite[Theorem 2.1]{DR16} that the map 
	\[ u = \left( \sigma , v \right) \mapsto \Pi_P \circ v \]
	induces a bijection
	\[\mcM_{c_d, \dots, c_1, c_0} \left( \mathbf{R} \times \Lambda_H, J, \alpha \right) \xrightarrow{\sim} \mcM_{\Pi_P (c_d) \dots \Pi_P (c_1), \Pi_P (c_0)} \left( \mathbf{L}_H, j \right) . \]
	This implies the result.
	
\end{proof}

\begin{lemma}\label{lemma relation tau_2 - gamma}
	
	We have
	\[\gamma = \zeta_2 \circ \tau_2 \circ \zeta_2^{-1} . \]
	
\end{lemma}

\begin{proof}
	
	This follows from the definitions of $\tau_2$, $\gamma$, $\zeta_2$ and the fact that $J = \left( D \Pi_P \right)_{|\xi}^* j$.
	
\end{proof}

\begin{lemma}\label{lemma mapping torus of tau_2 = mapping torus of gamma}
	
	The mapping torus of $\tau_2 : \mcA_2 \to \mcA_2$ is quasi-equivalent to the mapping torus of $\gamma : \mcO \to \mcO$ (see Definition \ref{definition mapping torus}).
	
\end{lemma}

\begin{proof}
	
	According to Lemma \ref{lemma relation tau_2 - gamma} the following diagram of Adams-graded $A_{\infty}$-categories is commutative 
	\[\begin{tikzcd}
	\mcA_2 \ar[d, "\zeta_2"] & \mcA_2 \sqcup \mcA_2 \ar[l, "\mathrm{id} \sqcup \mathrm{id}" above] \ar[r, "\mathrm{id} \sqcup \tau_2"] \ar[d, "\zeta_2 \sqcup \zeta_2"] & \mcA \ar[d, "\zeta_2"] \\
	\mcO & \mcO \sqcup \mcO \ar[l, "\mathrm{id} \sqcup \mathrm{id}" above] \ar[r, "\mathrm{id} \sqcup \gamma"] & \mcO .
	\end{tikzcd} \]
	Moreover, each vertical arrow is a quasi-equivalence according to Lemma \ref{lemma relation A_2 - O}. Thus, the result follows from Proposition \ref{prop invariance of homotopy colimits}.
	
\end{proof}

\subsection{Mapping torus of $\gamma$}\label{subsection proof of the main result}

In this section, we show that we can apply Theorem \ref{thm mapping torus in weak situation} (Theorem \ref{thm mapping torus in weak situation introduction} in the introduction) in order to compute the mapping torus of $\gamma : \mcO \to \mcO$. This allows us to finish the proof of Theorem \ref{thm mainthm restatement}.  

Recall that we fixed a contactomorphism $\phi_H$ of $V$ such that $\phi_H^* \alpha_H = \alpha$. Also recall that if $\theta$ is some real number, then 
\[\Lambda^{\theta} (E) = \left\{ \left( f_E(x) + \theta, x \right) \mid x \in L \right\}, \, \Lambda_H^{\theta} (E) = \phi_H^{-1} \left( \Lambda^{\theta} (E) \right), \text{ and } L_H^{\theta} (E) = \Pi_P \left( \Lambda_H^{\theta} (E) \right) . \]

\subsubsection{Continuation elements}

We denote by $\mcO_{2r}$ the $A_{\infty}$-category obtained from $\mcO$ by applying the functor of Definition \ref{definition forgetful functor}.   
Besides, we denote by
\[\Gamma = \left\{ c_n (E) \in \mcO_{2r} (L^n (E), L^{n+1} (E)) \mid (n,E) \in \mathbf{Z} \times \mcE \right\} \]
the set of continuation elements in $\mcO_{2r}$ induced by the exact Lagrangian isotopies $(L_H^{n+t})_{0 \leq t leq 1}$ (see for example \cite[section 3.3]{GPS20}).

Recall that if $\mcC$ is an $A_{\infty}$-category equipped with a $\mathbf{Z}$-splitting of $\mathrm{ob} \left( \mcC \right)$, we denote by $\mcC^0$ the full $A_{\infty}$-subcategory of $\mcC$ whose set of objects corresponds to $\{ 0 \} \times \mcE$.

\begin{lemma}\label{lemma localization of O equals CF}
	
	There are quasi-equivalences of $A_{\infty}$-categories
	\[\mcO_{2r}^0 \simeq \overrightarrow{\mcF uk} (\mathbf{L}_H) \text{ and } \mcO_{2r} \left[ \Gamma^{-1} \right]^0 \simeq \mcF uk (\mathbf{L}_H) . \]
	
\end{lemma}

\begin{proof}
	
	This is explained in \cite[Lecture 10]{Sei13}. See also \cite{GPS20}, where they define Fukaya categories by localization.
	
\end{proof}

\subsubsection{The $\mcO_{2r}$-bimodule map}\label{subsection bimodule map}

In order to apply Theorem \ref{thm mapping torus in weak situation}, we need a degree $0$ closed $\mcO_{2r}$-module map $f : \mcO_{2r} \left( -, - \right) \to \mcO_{2r} \left( -, \gamma \left( - \right) \right)$ such that the elements in $f \left( \text{units} \right)$ satisfy certain hypotheses. 
As usual, we would like to find such an $f$ ``geometrically'', i.e. using some Lagrangian (or Legendrian) isotopy. However, here the unit $1 = e_{L_H^k (E)} \in \mcO \left( L_H^k (E), L_H^k (E) \right)$, which is not a ``geometric'' morphism, is supposed to be sent by $f$ to something in $\mcO \left( L_H^k (E), L_H^{k+1} (E) \right)$, which is generated by the ``geometric'' elements of $L_H^k (E) \cap L_H^{k+1} (E)$.
Therefore, we need to somehow replace this unit by some intersection point between Lagrangians. The idea is that we will slightly perturb $L_H^k (E)$ to $L_H^{k+\delta} (E)$, and then replace $e_{L_H^k (E)}$ by the continuation element in the vector space generated by $L_H^k (E) \cap L_H^{k+\delta} (E)$. 

Observe that if $\delta$ is small enough, $L_H^{k+\delta} (E)$ is a small perturbation of $L_H^k (E)$. Therefore, in a Weinstein neighborhood of $L_H^k (E)$, the Lagrangian $L_H^{k+\delta} (E)$ is the graph of $d h_{\delta, k, E}$, where $h_{\delta, k, E}$ is some Morse function on $L (E)$.
In particular, the intersection points between $L_H^k (E)$ and $L_H^{k+\delta} (E)$ correspond to the critical points of $h_{\delta, k, E}$.
Moreover, the continuation element in the vector space generated by $L_H^k (E) \cap L_H^{k + \delta} (E)$ corresponds to the sum of the minima of $h_{\delta, k, E}$. 

\begin{exa}
	
	Assume that we are in the case
	\[\left( P, \lambda \right) = \left( T^* M, p dq \right), \, L = 0_M, \text{ and } H \left( q, p \right) = h(q), \]
	where $h : M \to \mathbf{R}$ is a Morse function. 
	As explained in example \ref{example rectify contact form in cotangent case}, in this case we have 
	\[L_H^{\theta} = \Pi_{T^*M} \left( j^1 \left( \theta e^h \right) \right) = \mathrm{graph} \left( d \left( \theta e^h \right) \right) . \]
	Thus, $L_H^{k+\delta}$ is the graph of $d \left( \delta e^h \right)$ over $L_H^k$. 
	
\end{exa}

The following result will allow us to ``replace'' the units by geometric morphisms as explained above. We denote by $g = - d\lambda (-, j -)$ the metric on $P$ induced by $j$ and $\left( -d \lambda \right)$.

\begin{lemma}\label{lemma discs with boundary on small deformation}
	
	For every positive integer $n$, there exists $\delta_n > 0$ such that the following holds for every $\delta \in \left] 0, \delta_n \right]$. For every sequence of integers
	\[(-n, E_0) \leq (j_0, E_0) < \dots < (j_p, E_p) \leq (\ell_0, E_0') < \dots < (\ell_q, E_q') \leq (n, E_q'), \quad p,q \geq 0, \]
	the rigid $j$-holomorphic discs in $P$ with boundary on 
	\[L_H^{j_0} (E_0) \cup \dots \cup L_H^{j_p} (E_p) \cup L_H^{\ell_0 + \delta} (E_0') \cup \dots \cup L_H^{\ell_q+\delta} (E_q') \]
	are 
	\begin{enumerate}
		
		\item in bijection with the rigid $j$-holomorphic discs in $P$ with boundary on 
		\[L_H^{j_0} (E_0) \cup \dots \cup L_H^{j_p} (E_p) \cup L_H^{\ell_0} (E_0') \cup \dots \cup L_H^{\ell_q} (E_q') \]
		if $(j_p, E_p) < (\ell_0, E_0')$, or
		
		\item in bijection with the rigid $j$-holomorphic discs in $P$ with boundary on 
		\[L_H^{j_0} (E_0) \cup \dots \cup L_H^{j_{p-1}} (E_{p-1}) \cup L_H^{\ell_0} (E_0') \cup L_H^{\ell_1} (E_1') \cup \dots \cup L_H^{\ell_q} (E_q') \]
		with a flow line of $\left( - \nabla_g h_{\delta, k, E_0'} \right)$ attached on the component in $L_H^{\ell_0} (E_0')$ if $(j_p, E_p) = (\ell_0, E_0')$.
		
	\end{enumerate}
	
\end{lemma}

\begin{proof}
	
	The case $j_p < \ell_0$ follows from transversality of the moduli spaces in consideration. The case $j_p = \ell_0$ follows from the main analytic theorem of \cite{EES09} (Theorem 3.6). 
	
\end{proof}

In order to define the $\mcO_{2r}$-bimodule map $f$ properly, we will use Lemma \ref{lemma discs with boundary on small deformation} to modify the $A_{\infty}$-category $\mcO_{2r}$.
In the following, we fix a \emph{decreasing} sequence of positive numbers $\left( \delta_n \right)_{n \geq 1}$ such that, for every $n$,
\begin{enumerate}
	
	\item Lemma \ref{lemma discs with boundary on small deformation} holds with $\delta_n$, and
	
	\item $\delta_n$ is small enough so that there is no handle slide instant in the Legendrian isotopy
	\[\bigcup_{\ell = -n}^n \Lambda_H^{\ell + \delta_n t} = \bigcup_{\ell = -n}^n \bigcup_{E \in \mcE} \Lambda_H^{\ell + \delta_n t} (E), \quad t \in \left[ 0, 1 \right]. \]
	
\end{enumerate}

We define two families of $A_{\infty}$-categories $\left( \mcO_{n, k} \right)_{n,k}$ and $\left( \widetilde{\mcO}_{n,k} \right)_{n,k}$ indexed by the couples $(n,k)$, where $n \geq 1$ and $-n \leq k \leq n$. 
The $A_{\infty}$-category $\mcO_{n, k}$ is basically obtained from $\mcO_{2r}$ by restricting to objects $L_H^i (E)$, $-n \leq i \leq n$, and adding a copy of the object $L_H^k (E)$.

\begin{defin}\label{definition category O n,k}
	
	For every $(j,E) \in \mathbf{Z} \times \mcE$, let $\overline{L_H^j} (E)$ be a copy of $L_H^j (E)$.
	We denote by $\mcO_{n,k}$ the $A_{\infty}$-category defined as follows:
	\begin{enumerate}
		
		\item the set of objects of $\mcO_{n,k}$ is  
		\[\mathrm{ob} \left( \mcO_{n,k} \right) = \left\{ L_H^j (E) \mid -n \leq j \leq k, E \in \mcE \right\} \cup \left\{ \overline{L_H^{\ell}} (E) \mid k \leq \ell \leq n, E \in \mcE \right\}, \]
		
		\item the spaces of morphisms in $\mcO_{n,k}$ are the corresponding spaces of morphisms in $\mcO_{2r}$ when we replace $\overline{L_H^{\ell}} (E)$, $k \leq \ell \leq n$, by $L_H^{\ell} (E)$, except that 
		\[\mcO_{n,k} \left( \overline{L_H^k} (E), L_H^k (E) \right) = \{ 0 \}, \]
		
		\item the operations are the same as in $\mcO_{2r}$.
		
	\end{enumerate} 	
	
\end{defin}

The $A_{\infty}$-category $\widetilde{\mcO}_{n,k}$ is obtained from $\mcO_{n,k}$ by perturbing the objects $\overline{L_H^{\ell}} (E)$, $k \leq \ell \leq n$, to $L_H^{\ell + \delta_n} (E)$.

\begin{defin}\label{definition category O n,k tilde}
	
	Let 
	\[\Theta_{n,k} := \left\{ -n, \dots, k \right\} \cup \left\{ \ell + \delta_n \mid k \leq \ell \leq n \right\} \subset \mathbf{R}, \text{ and } \widetilde{\mathbf{L}_H} := \left( L_H^{\theta} (E) \right)_{(\theta, E) \in \Theta_{n,k} \times \mcE}. \]
	We denote by $\widetilde{\mcO}_{n,k}$ the $A_{\infty}$-category defined as follows:
	\begin{enumerate}
		
		\item the objects of $\widetilde{\mcO}_{n,k}$ are the Lagrangians $L_H^{\theta} (E)$, $(\theta, E) \in \Theta_{n,k} \times \mcE$,
		
		\item the space of morphisms from $L_H^{\theta} (E)$ to $L_H^{\theta'} (E')$ is either generated by $L_H^{\theta} (E) \cap L_H^{\theta'} (E')$ if $(\theta, E) < (\theta', E')$, or $\mathbf{F}$ if $(\theta, E) = (\theta', E')$, or $0$ otherwise,

		\item the operations are such that $e_{L_H^{\theta} (E)} = 1 \in \widetilde{\mcO}_{n,k} \left( L_H^{\theta} (E), L_H^{\theta} (E) \right)$ is a strict unit, and for every sequence $(\theta_0, E_0) < \dots < (\theta_d, E_d)$, for every sequence of intersection points 
		\[\left( x_1, \dots, x_d \right) \in \left( L_H^{\theta_0} (E_0) \cap L_H^{\theta_1} (E_1) \right) \times \dots \times \left( L_H^{\theta_{d-1}} (E_{d-1}) \cap L_H^{\theta_d} (E_d) \right), \]
		we have 
		\[\mu_{\widetilde{\mcO}_{n,k}} \left( x_1, \dots, x_d \right) = \sum \limits_{x_0 \in L_H^{\theta_0} (E_0) \cap L_H^{\theta_d} (E_d)} \# \mcM_{x_d, \dots, x_1, x_0} \left( \widetilde{\mathbf{L}_H}, j \right) x_0 . \]
		
	\end{enumerate} 	
	
\end{defin}

These $A_{\infty}$-categories being defined, Lemma \ref{lemma discs with boundary on small deformation} implies the following result.

\begin{lemma}\label{lemma rho is a functor}
	
	There is a strict $A_{\infty}$-functor $\rho_{n,k} : \mcO_{n,k} \to \widetilde{\mcO}_{n,k}$  defined as follows:
	\begin{enumerate}
		
		\item on objects, $\rho_{n,k} \left( L_H^j (E) \right) = L_H^j (E)$ if $-n \leq j \leq k$ and $\rho_{n,k} \left( \overline{L_H^{\ell}} (E) \right) = L_H^{\ell + \delta_n} (E)$ if $k \leq \ell \leq n$,
		
		\item on morphisms, $\rho_{n,k}$ sends the unit of $\mcO_{n,k} \left( L_H^k (E), \overline{L_H^k} (E) \right) = \mathbf{F}$ to the continuation element in $\widetilde{\mcO}_{n,k} \left( L_H^k (E), L_H^{k + \delta_n} (E) \right)$, and it sends any other morphism of $\mcO_{n,k}$ to the corresponding one in $\widetilde{\mcO}_{n,k}$. 
		
	\end{enumerate}
	
\end{lemma}

\begin{proof}
	
	Consider a sequence $\left( x_0, \dots, x_{d-1} \right)$ of morphisms in $\mcO_{n,k}$.
	If in this sequence there is no morphism from $L_H^k (E)$ to $\overline{L_H^k} (E)$, then the relation 
	\[\mu_{\widetilde{\mcO}_{n,k}} \left( \rho_{n,k} x_0, \dots, \rho_{n,k} x_d \right) = \rho_{n,k} \left( \mu_{\mcO_{n,k}} \left( x_0, \dots, x_d \right) \right) \]
	follows directly from the first item of Lemma \ref{lemma discs with boundary on small deformation}.
	Now assume that there is $p \in \{ 0, \dots, d-1 \}$ such that $x_p = e_{L_H^k (E)} \in \mcO_{n,k} \left( L_H^k (E), \overline{L_H^k} (E) \right)$. Recall that the continuation element in $\widetilde{\mcO}_{n,k} \left( L_H^k (E), L_H^{k + \delta_n} (E) \right)$ corresponds to the sum of the minima of $h_{\delta_n, k, E}$. 
	Then the second item of Lemma \ref{lemma discs with boundary on small deformation} implies that 
	\[\mu_{\widetilde{\mcO}_{n,k}} \left( \rho_{n,k} x_0, \dots, \rho_{n,k} x_d \right) = \left\{
	\begin{array}{ll}
	\rho_{n,k} x_1 & \text{if } d=1 \text{ and } p=0 \\
	\rho_{n,k} x_0 & \text{if } d=1 \text{ and } p=1 \\
	0 & \text{otherwise} .
	\end{array}
	\right. \]
	Thus, the $A_{\infty}$-relation for $\rho_{n,k}$ is still satisfied according to the behavior of the operations $\mu_{\mcO_{n,k}}$ with respect to the unit $e_{L_H^k (E)}$.
	
\end{proof}

Now that we have in some sense ``replace'' the unit $e_{L_H^k (E)} \in \mcO_{n,k} \left( L_H^k (E), \overline{L_H^k} (E) \right)$ by the continuation element in $\widetilde{\mcO}_{n,k} \left( L_H^k (E), L_H^{k + \delta_n} (E) \right)$, we can define geometrically an $A_{\infty}$-functor that will finally allow us to define the $\mcO_{2r}$-bimodule map $f$.

\begin{defin}\label{definition functor nu}
	
	We denote by $\nu_{n,k} : \widetilde{\mcO}_{n, k} \to \mcO_{2r}$ the $A_{\infty}$-functor defined as follows:
	\begin{enumerate}
		
		\item on objects, $\nu_{n,k} \left( L_H^j (E) \right) = L_H^j (E)$ if $-n \leq j \leq k$, and $\nu_{n,k} \left( L_H^{\ell + \delta_n} (E) \right) = L_H^{\ell + 1} (E)$ if $k \leq \ell \leq n$,
		
		\item on morphisms, $\nu_{n,k}$ is obtained by dualizing the components of the DG-isomorphism
		\[A \left( \bigsqcup\limits_{i=-n}^{n+1} \Lambda_H^i \right) \xrightarrow{\sim} A \left( \bigsqcup\limits_{j=-n}^{k} \Lambda_H^j \sqcup \bigsqcup\limits_{\ell=k}^{n} \Lambda_H^{\ell+\delta_n} \right) . \]
		induced by the Legendrian isotopy  
		\[\left( \bigsqcup\limits_{j=-n}^{k} \Lambda_H^j \right) \sqcup \left( \bigsqcup\limits_{\ell=k}^{n} \Lambda_H^{\ell+1-t(1-\delta_n)} \right), \quad t \in \left[ 0,1 \right] \]  
		(see Theorem \ref{thm invariance} or \cite[Proposition 2.6]{EES07}).
		
	\end{enumerate}
	
\end{defin}

\begin{rmks}\label{rmk properties of sigma}
	
	We point out some properties of the $A_{\infty}$-functors
	\[\sigma_{n,k} := \nu_{n,k} \circ \rho_{n,k} : \mcO_{n,k} \to \mcO_{2r} . \]
	\begin{enumerate}
		
		\item Let $n \leq p$ be two positive integers, and let $k \in \{ -n, \dots, n \}$. Recall that we chose $\delta_n$ small enough so that there is no handle slide instant in the Legendrian isotopy 
		\[\bigsqcup_{\ell=-n}^n \Lambda_H^{\ell + \delta_n t}, 0 \leq t \leq 1. \]
		Since $\delta_p \leq \delta_n$, neither is there any handle slide instant in the Legendrian isotopy 
		\[\bigsqcup_{\ell=-n}^n \Lambda_H^{\ell + \delta_p t}, 0 \leq t \leq 1. \]
		Therefore, $\sigma_{p,k}$ agrees with $\sigma_{n,k}$ on $\mcO_{n,k} \subset \mcO_{p,k}$.
		
		\item 	Consider a sequence of integers 
		\[-n \leq j_0 < \dots < j_p \leq k_1 < k_2 \leq \ell_0 < \dots < \ell_q \leq n , \]
		and a sequence of morphisms 
		\begin{align*}
			& \left( x_0, \dots, x_{p-1}, u, y_0, \dots, y_{q-1} \right) \in \mcO_{n,k_i} \left( L_H^{j_0} (E_0), L_H^{j_1} (E_1) \right) \times \dots \times \mcO_{n,k_i} \left( L_H^{j_{p-1}} (E_{p-1}), L_H^{j_p} (E_p) \right) \\
			& \times \mcO_{n,k_i} \left( L_H^{j_p} (E_p), \overline{L_H^{\ell_0}} (E_0') \right) \times \mcO_{n,k_i} \left( \overline{L_H^{\ell_0}} (E_0'), \overline{L_H^{\ell_1}} (E_1') \right) \times \dots \times \mcO_{n,k_i} \left( \overline{L_H^{\ell_{q-1}}} (E_{q-1}'), \overline{L_H^{\ell_q}} (E_q') \right) .
		\end{align*}
		Since the Legendrian isotopy defining $\nu_{n,k_i}$ is 
		\[\left( \bigsqcup\limits_{j=-n}^{k_i} \Lambda_H^j \right) \sqcup \left( \bigsqcup\limits_{\ell=k_i}^{n} \Lambda_H^{\ell+1-t(1-\delta_n)} \right), \quad t \in \left[ 0,1 \right], \] 
		we have
		\[\left\{
		\begin{array}{l}
		\sigma_{n, k_1} \left( x_0, \dots, x_{p-1} \right) = \delta_{1 p} x_0 \\
		\sigma_{n, k_2} \left( y_0, \dots, y_{q-1} \right) = \gamma \left( y_0, \dots, y_{q-1} \right) \\
		\sigma_{n, k_2} \left( x_0, \dots, x_{p-1}, u, y_0, \dots, y_{q-1} \right) = \sigma_{n, k_1} \left( x_0, \dots, x_{p-1}, u, y_0, \dots, y_{q-1} \right) .
		\end{array}
		\right. \]
		
		\item By construction, the $A_{\infty}$-functor $\nu_{n,k}$ sends the continuation element in $\widetilde{\mcO}_{n,k} \left( L_H^k (E), L_H^{k + \delta_n} (E) \right)$ (corresponding to the sum of the minima of $h_{\delta_n, k, E}$) to the continuation element $c_k (E)$ in $\mcO_{2r} \left( L_H^k (E), L_H^{k+1} (E) \right)$.
		In other words, $\sigma_{n,k}$ sends the unit $e_{L_H^k (E)} \in \mcO_{n,k} \left( L_H^k (E), \overline{L_H^k} (E) \right)$ to $c_k (E)$.
		
		\item The map $\sigma_{n,k} : \mcO_{n,k} \left( L_H^j(E) , \overline{L_H^k}(E') \right) \to \mcO_{2r} \left( L_H^j(E), L_H^{k+1} (E') \right)$ is a quasi-somorphism for every $j < k$ and $E, E' \in \mcE$
		
	\end{enumerate}
	
\end{rmks}

We can now state and prove the desired result. 

\begin{lemma}\label{lemma bimodule map}
	
	There exists a degree $0$ closed $\mcO_{2r}$-bimodule map $f : \mcO_{2r} \left( -, - \right) \to \mcO_{2r} \left( -, \gamma (-) \right)$ which sends the unit $e_{L_H^k (E)} \in \mcO_{2r} \left( L_H^k (E), L_H^k (E) \right)$ to the continuation element $c_k (E) \in \mcO_{2r} \left( L_H^k (E), L_H^{k+1} (E) \right) \cap \Gamma$, and such that $f : \mcO_{2r} \left( L_H^j(E) , L_H^k(E') \right) \to \mcO_{2r} \left( L_H^j(E), L_H^{k+1} (E') \right)$ is a quasi-somorphism for every $j< k$ and $E, E' \in \mcE$.
	
\end{lemma}

\begin{proof}
	
	Consider a sequence  
	\[(j_0, E_0) < \dots < (j_p, E_p) \leq (k ,E) = (\ell_0, E_0') < \dots < (\ell_q, E_q') , \]
	and a sequence of morphisms 
	\begin{align*}
	& \left( x_0, \dots, x_{p-1}, u, y_0, \dots, y_{q-1} \right) \in \mcO_{2r} \left( L_H^{j_0} (E_0), L_H^{j_1} (E_1) \right) \times \dots \times \mcO_{2r} \left( L_H^{j_{p-1}} (E_{p-1}), L_H^{j_p} (E_p) \right) \\
	& \times \mcO_{2r} \left( L_H^{j_p} (E_p), L_H^{k} (E_0') \right) \times \mcO_{2r} \left( L_H^{k} (E_0'), L_H^{\ell_1} (E_1') \right) \times \dots \times \mcO_{2r} \left( L_H^{\ell_{q-1}} (E_{q-1}'), L_H^{\ell_q} (E_q') \right) .
	\end{align*}
	We choose $n \geq 1$ such that $-n \leq j_0 \leq \ell_q \leq n$, and we set 
	\[f \left( x_0, \dots, x_{p-1}, u, y_0, \dots, y_{q-1} \right) := \sigma_{n, k} \left( x_0, \dots, x_{p-1}, u, y_0, \dots, y_{q-1} \right) \in \mcO_{2r} \left( L_H^{j_0} (E_0), \gamma L_H^{\ell_q} (E_q') \right), \]
	where on the right hand side we consider that 
	\begin{align*}
	& \left( x_0, \dots, x_{p-1}, u, y_0, \dots, y_{q-1} \right) \in \mcO_{n,k} \left( L_H^{j_0} (E_0), L_H^{j_1} (E_1) \right) \times \dots \times \mcO_{n,k} \left( L_H^{j_{p-1}} (E_{p-1}), L_H^{j_p} (E_p) \right) \\
	& \times \mcO_{n,k} \left( L_H^{j_p} (E_p), \overline{L_H^k} (E_0') \right) \times \mcO_{n,k} \left( \overline{L_H^k} (E_0'), \overline{L_H^{\ell_1}} (E_1') \right) \times \dots \times \mcO_{n,k} \left( \overline{L_H^{\ell_{q-1}}} (E_{q-1}'), \overline{L_H^{\ell_q}} (E_q') \right) .
	\end{align*}
	Observe that $f$ is well defined (it does not depend on the choice of $n$) according to the first item of Remark \ref{rmk properties of sigma}.

	We now verify that $f$ is closed. 
	According to Definition \ref{definition bimodule}, we have 
	\begin{align*}
	\mu_{\mathrm{Mod}_{\mcC, \mcC}}^1 & \left( f \right) \left( x_0, \dots, x_{p-1}, u, y_0, \dots, y_{q-1} \right) \\
	& = \sum \sigma_{n, k} \left( \dots, \mu_{\mcO_{2r}} \left( \dots \right), \dots, u, \dots \right) \\ 
	& + \sum \sigma_{n, \ell_s} \left( \dots, \mu_{\mcO_{2r}} \left( x_r, \dots, x_{p-1}, u, y_0, \dots, y_{s-1} \right), \dots \right) \\
	& + \sum \sigma_{n, k} \left( \dots, u, \dots, \mu_{\mcO_{2r}} \left( \dots \right), \dots \right) \\ 
	& + \sum \mu_{\mcO_{2r}} \left( \dots, \sigma_{n, k} \left( \dots, u, \dots \right), \gamma \left( \dots \right), \dots, \gamma \left( \dots \right) \right) .
	\end{align*}
	Now according to the second item of Remark \ref{rmk properties of sigma}, we have 
	\begin{align*}
	\sum \sigma_{n, \ell_s} & \left( \dots, \mu_{\mcO_{2r}} \left( x_r, \dots, x_{p-1}, u, y_0, \dots, y_{s-1} \right), \dots \right) \\
	& = \sum \sigma_{n, k} \left( \dots, \mu_{\mcO_{2r}} \left( x_r, \dots, x_{p-1}, u, y_0, \dots, y_{s-1} \right), \dots \right)
	\end{align*}
	and 
	\begin{align*}
	\sum & \mu_{\mcO_{2r}} \left( \dots, \sigma_{n, k} \left( \dots, u, \dots \right), \gamma \left( \dots \right), \dots, \gamma \left( \dots \right) \right) \\
	& = \sum \mu_{\mcO_{2r}} \left( \sigma_{n, k} \left( \dots \right), \dots, \sigma_{n, k} \left( \dots \right), \sigma_{n, k} \left( \dots, u, \dots \right), \sigma_{n, k} \left( \dots \right), \dots, \sigma_{n, k} \left( \dots \right) \right) .
	\end{align*}
	Therefore, we get 
	\[\mu_{\mathrm{Mod}_{\mcC, \mcC}}^1 \left( f \right) \left( x_0, \dots, x_{p-1}, u, y_0, \dots, y_{q-1} \right) = 0 \]
	from the fact that $\sigma_{n, k}$ is an $A_{\infty}$-functor. 
	
	Now $f$ sends the unit $e_{L_H^k (E)} \in \mcO_{2r} \left( L_H^k (E), L_H^k (E) \right)$ to the continuation element $c_k (E) \in \mcO_{2r} \left( L_H^k (E), L_H^{k+1} (E) \right) \cap \Gamma$ according to the third item of Remark \ref{rmk properties of sigma}. 
	Finally, the map $f : \mcO_{2r} \left( L_H^j(E) , L_H^k(E') \right) \to \mcO_{2r} \left( L_H^j(E), L_H^{k+1} (E') \right)$ is a quasi-somorphism for every $j< k$ and $E, E' \in \mcE$ according to the last item of Remark \ref{rmk properties of sigma}.
	
\end{proof}

\subsubsection{Proof of the main result}

We end the section with the proof of Theorem \ref{thm mainthm restatement} (Theorem \ref{thm mainthm} in the Introduction). 

Recall that we denote by $\mathbf{F} \left[ t_m \right]$ the augmented Adams-graded associative algebra generated by a variable $t_m$ of bidegree $(m, 1)$, and by $t_m \mathbf{F} \left[ t_m \right]$ its augmentation ideal (or equivalently, the ideal generated by $t_m$).
The key result is the following.

\begin{lemma}\label{lemma mapping torus of gamma}
	
	The mapping torus of $\gamma$ is quasi-equivalent to the Adams-graded $A_{\infty}$-category $\overrightarrow{\mcF uk} (\mathbf{L}) \oplus \left( t_{2r} \mathbf{F} \left[ t_{2r} \right] \otimes \mcF uk (\mathbf{L}) \right)$.
	
\end{lemma}

\begin{proof}
	
	Let $f : \mcO_{2r} \left( -, - \right) \to \mcO_{2r} \left( -, \gamma (-) \right)$ be the degree $0$ closed bimodule map of Lemma \ref{lemma bimodule map}. According to the latter, the hypotheses of Theorem \ref{thm mapping torus in weak situation} are satisfied, and $f(\mathrm{units}) = \Gamma$. 
	Thus the mapping torus of $\gamma$ is quasi-equivalent to the Adams-graded $A_{\infty}$-algebra $\mcO_{2r}^0 \oplus \left( t_{2r} \mathbf{F} \left[ t_{2r} \right] \otimes \mcO_{2r} \left[ \Gamma^{-1} \right]^0 \right)$ (recall that if $\mcC$ is an $A_{\infty}$-category equipped with a $\mathbf{Z}$-splitting $\mathbf{Z} \times \mcE \simeq \mathrm{ob} \left( \mcC \right)$, we denote by $\mcC^0$ the full $A_{\infty}$-subcategory of $\mcC$ whose set of objects corresponds to $\{ 0 \} \times \mcE$).
	According to Lemma \ref{lemma localization of O equals CF} we have
	\[\mcO_{2r}^0 \simeq \overrightarrow{\mcF uk} (\mathbf{L}_H) \text{ and } \mcO_{2r} \left[ \Gamma^{-1} \right]^0 \simeq \mcF uk (\mathbf{L}_H) . \]
	The result follows from invariance of the Fukaya category
	\[\overrightarrow{\mcF uk} (\mathbf{L}_H) \simeq \overrightarrow{\mcF uk} (\mathbf{L}) \text{ and } \mcF uk (\mathbf{L}_H) \simeq \mcF uk (\mathbf{L}) . \]
	
\end{proof}

We now give the proof of Theorem \ref{thm mainthm restatement} (Theorem \ref{thm mainthm} in the introduction).
According to \cite[Theorem 2.4]{LPWZ08}, Koszul duality holds for the augmented Adams-graded DG-algebra $CE_{-*}^r \left( \mathbf{\Lambda^{\circ}} \right)$ because it is \emph{Adams connected} (see \cite[Definition 2.1]{LPWZ08}). Indeed, recall from section \ref{subsection invariants and statement of the result} that the Adams-degree in $CE_{-*}^r \left( \mathbf{\Lambda^{\circ}} \right)$ of a Reeb chord $c$ is the number of times $c$ winds around the fiber. 
Besides, recall from section \ref{subsection invariants and statement of the result} that there is a coaugmented Adams-graded $A_{\infty}$-cocategory $LC_* \left( \mathbf{\Lambda^{\circ}} \right)$ such that 
\[CE_{-*}^r \left( \mathbf{\Lambda^{\circ}} \right) = \Omega \left( LC_* \left( \mathbf{\Lambda^{\circ}} \right) \right) \text{ and } LA^* \left( \mathbf{\Lambda^{\circ}} \right) = LC_* \left( \mathbf{\Lambda^{\circ}} \right)^{\#}. \]
Since there is a quasi-isomorphism $B \left( \Omega C \right) \simeq C$ for every $A_{\infty}$-cocategory $C$ (see \cite[section 2.2.2]{EL21}), it follows that 
\[E \left( CE_{-*}^r \left( \mathbf{\Lambda^{\circ}} \right) \right) = B \left( CE_{-*}^r \left( \mathbf{\Lambda^{\circ}} \right) \right)^{\#} \simeq LC_* \left( \mathbf{\Lambda^{\circ}} \right)^{\#} = LA^* \left( \mathbf{\Lambda^{\circ}} \right) \]
(graded dual preserves quasi-isomorphisms).
Now the quasi-equivalence 
\[LA^* \left( \mathbf{\Lambda^{\circ}} \right) \simeq \overrightarrow{\mcF uk} (\mathbf{L}) \oplus \left( t_{2r} \mathbf{F} \left[ t_{2r} \right] \otimes \mcF uk (\mathbf{L}) \right) \]
follows from Lemmas \ref{lemma A^rond = mapping torus of tau}, \ref{lemma mapping torus of tau = mapping torus of tau_1}, \ref{lemma mapping torus of tau_1 = mapping torus of tau_2}, \ref{lemma mapping torus of tau_2 = mapping torus of gamma} and \ref{lemma mapping torus of gamma}.
This concludes the proof.

%\appendix
%\input{./Extended_algebra/extended_algebra}

\phantomsection % To have a correct link in the table of contents
\addcontentsline{toc}{section}{Bibliography}
\printbibliography

\end{document}